\documentclass[12pt]{amsart} 
\usepackage[margin= 1.1 in]{geometry}
\usepackage[usenames,dvipsnames]{xcolor}
\usepackage[hidelinks]{hyperref}
\hypersetup{colorlinks = true,
linkcolor = blue, anchorcolor =blue,
citecolor = ForestGreen,
filecolor = blue,urlcolor = blue,
            pdfauthor=author}
\usepackage[capitalize]{cleveref}

\usepackage{mathtools}

\usepackage{amsmath,amssymb,bm}
\usepackage[alphabetic]{amsrefs}
\usepackage{indentfirst}
\usepackage{graphicx}
\usepackage{amsthm}

\usepackage{enumitem}

\makeatletter 
\@mparswitchfalse%
\makeatother
\normalmarginpar 

\usepackage{todonotes}

\newtheorem{theorem}{Theorem}[section]
\newtheorem{lemma}[theorem]{Lemma}
\newtheorem{proposition}[theorem]{Proposition}
\newtheorem{corollary}[theorem]{Corollary}

\numberwithin{equation}{section}

\newtheorem{assumption}{Assumption}

\theoremstyle{remark}
\newtheorem{remark}{Remark}

\theoremstyle{definition}
\newtheorem{definition}[theorem]{Definition}

\DeclareMathOperator{\dist}{dist}

\newcommand{\ud}{\,\mathrm{d}}

\newcommand{\R}{\mathbb{R}}
\newcommand{\C}{\mathbb{C}}

\def \Or {\mathcal{O}}

\newcommand{\ZZ}{\mathbb{Z}}

\newcommand{\rr}{\mathcal{R}}

\newcommand{\dd}{\cdot}
\newcommand{\ti}{\times}

\IfFileExists{mathabx.sty}%
  {\DeclareFontFamily{U}{mathx}{\hyphenchar\font45}%
   \DeclareFontShape{U}{mathx}{m}{n}{<->mathx10}{}%
   \DeclareSymbolFont{mathx}{U}{mathx}{m}{n}%
   \DeclareFontSubstitution{U}{mathx}{m}{n}%
   \DeclareMathAccent{\widebar}{0}{mathx}{"73}%
}{%
  \PackageWarning{mathabx}{%
    Package mathabx not available, therefore\MessageBreak substituting
    widebar with overline\MessageBreak }%
  \newcommand{\widebar}[1]{\overline{#1}}%
}
\newcommand{\wb}[1]{\widebar{#1}}

\newcommand{\mc}[1]{\mathcal{#1}}

\newcommand{\eps}{\epsilon}
\newcommand{\lad}{\lambda}
\newcommand{\si}{\sigma}

\newcommand{\p}{\partial}
\def \h {\hat} 

\def \ww {\omega}
\def \l {\langle} 
\def \r {\rangle}
\def \d {\delta}
\def \vp {\varphi}

\def \kk {\mc{K}}
\def \wkk {\w{\mc{K}}_D}
\def \wph {\w{\phi}}
\def \rr {\mc{R}}

\def \br {\breve}

\newcommand{\norm}[1]{\lVert#1\rVert}

\renewcommand{\Re}{\mathfrak{Re}\,} 
\renewcommand{\Im}{\mathfrak{Im}\,}

\def \q {\quad}
\def \w {\widetilde}
\def \mm {\left[\begin{matrix}}
\def \nn {\end{matrix}\right]}

\newcommand{\kb}[1]{\mathcal{K}_{D,#1}}

\title[Nonlinear dielectric resonances]{Dielectric scattering resonances for high-refractive resonators with cubic nonlinearity}

\author{Habib Ammari} %
\address[H. Ammari]{Department of Mathematics, ETH Z\"{u}rich, R\"{a}mistrasse 101, CH-8092 Z\"{u}rich, Switzerland; Hong Kong Institute for Advanced Study, City University of Hong Kong, Kowloon Tong, Hong Kong}
\email{habib.ammari@math.ethz.ch}

\author{Bowen Li} %
\address[B. Li]{Department of Mathematics,  City University of Hong Kong, Kowloon Tong, Hong Kong SAR}
\email{bowen.li@cityu.edu.hk}

\date{}

\begin{document}

\begin{abstract}
This work establishes a rigorous mathematical framework for the analysis of nonlinear dielectric resonances in wave scattering by high-index resonators with Kerr-type nonlinearities. We consider both two- and three-dimensional settings and prove the existence of nonlinear dielectric resonances in the subwavelength regime, bifurcating from the zero solution at the corresponding linear resonances. Furthermore, we derive asymptotic expansions for the nonlinear resonances and states in terms of the high contrast parameter $\tau$ and the normalization constant $\mathcal{N}$. For a symmetric dimer of resonators, these small-amplitude nonlinear resonant states exhibit either symmetric or antisymmetric profiles. In three dimensions, under conditions valid in the dilute regime, we prove that as the field amplitude $\mathcal{N}$ increases, mode hybridization induces a symmetry-breaking bifurcation along the principal symmetric solution branch at a critical amplitude $ \mathcal{N}_{\mathrm{crit}}$. This bifurcation gives rise to two asymmetric resonant states, each localized on one of the particles in the dimer.  Remarkably, in two dimensions, we show that no such symmetry-breaking bifurcation exists along the principal solution branches, owing to the distinct scaling behavior of the principal nonlinear subwavelength resonance arising from the logarithmic singularity.
\end{abstract}

\maketitle

\bigskip

 \par

\bigskip

\\

\tableofcontents

\section{Introduction}
Subwavelength photonic resonators made from all-dielectric materials with high refractive indices, such as silicon nanoparticles, have recently emerged as promising building blocks for advanced optical devices. These resonators provide a compelling alternative to traditional plasmonic structures, offering broad functionality across various applications \cites{kuznetsov2016optically,arbabi2015dielectric}. Unlike plasmonic systems, all-dielectric nanoresonators inherently exhibit low radiative losses and support engineered magnetic Mie-type modes, enabling unprecedented control over light at subwavelength scales \cite{baranov2017all}.  
The rise of all-dielectric nanophotonics has also unlocked new possibilities, particularly in nonlinear optics. By exploiting their resonant properties, dielectric nanostructures have demonstrated exceptional efficiency in nonlinear optical processes, surpassing metallic nanoparticles with plasmonic resonances by several orders of magnitude \cites{butet2015optical}. The combination of a high refractive index contrast and Kerr nonlinearity in these structures facilitates dynamic light manipulation through effects such as self-focusing and intensity-dependent refractive index changes.  These advancements underscore the immense potential of dielectric nanoparticles to confine and guide electromagnetic waves at subwavelength scales, driving progress toward the long-standing goal of achieving robust wave control and transport \cites{kivshar2018all,koshelev2020subwavelength,smirnova2020nonlinear}. 



Despite significant advances in experimental and numerical modeling of nonlinear optical scattering, the mathematical understanding of the nonlinear resonant behavior in high-index resonators remains quite limited, particularly in comparison to the well-understood linear scattering regime. 
For the analysis of linear resonant responses, a particularly intriguing phenomenon observed experimentally \cites{evlyukhin2012demonstration, kuznetsov2016optically} is that silicon nanoparticles can resonate in the optical domain with a high $Q$-factor, where the excited electric and magnetic dipole moments are of comparable magnitudes. These properties are mathematically associated with the so-called dielectric resonances. 
In the specific case of spherical nanoparticles, dielectric resonances are well described by Mie scattering theory \cite{tzarouchis2018light}. More generally, in the high-contrast limit, these resonances can be treated as a linear eigenvalue problem for the Newtonian potential \cites{meklachi2018asymptotic, ammari2019subwavelength}, restricted to the TM or TE polarization model governed by the Helmholtz equation. For the full Maxwell equations, comprehensive mathematical theories have been developed for both single dielectric nanoparticles and clusters of nanoparticles, as detailed in \cite{ammari2023mathematical} and \cite{cao2022electromagnetic}, respectively. Moreover, recent work \cite{ammari2024fano} provides a solid mathematical foundation for Fano resonances in asymmetric dielectric metasurfaces. Specifically, it shows the existence of embedded eigenvalues for symmetric metasurfaces under normal incidence. When the system experiences slight perturbations, these embedded eigenvalues shift into the lower complex half-plane with a small imaginary part, leading to the appearance of sharp, asymmetric Fano-type line shapes in the transmission and reflection spectra. 

 Nevertheless, a deep mathematical understanding of nonlinear dielectric scattering resonances remains largely undeveloped. Existing works \cite{meklachi2018asymptotic} and \cite{ammari2025nonlinear} address high-contrast nonlinear scatterers and nonlinear acoustic bubbly media, respectively, using asymptotic analysis based on small volume size and contrast parameters. Both, however, assume a priori the existence of nonlinear resonances with a convergent expansion series. Additionally, the numerical experiments in \cite{ammari2025nonlinear} reveal that, compared to the linear case, nonlinearities can induce extra resonant modes in a dimer of resonators. In this work, we aim to address these gaps by rigorously proving the existence of nonlinear dielectric resonances in the high-contrast regime and establishing the emergence of extra resonances in the dimer setting.

\subsection{Main results and layout}
Electromagnetic scattering by inhomogeneous media is mathematically described by Maxwell’s equations with various boundary conditions and material coefficients. In this work, we focus on scattering by high-index dielectric resonators subjected to high-intensity radiation, where the refractive index of the medium is large and depends on the magnitude of the scattered field. Specifically, we consider Kerr-type nonlinearity, a prominent example in nonlinear optics \cites{boyd2008nonlinear}. 
For simplicity, we examine the transverse field case under time-harmonic propagation, which reduces the problem to the nonlinear Helmholtz equation in $\R^d$ ($d = 2,3$):  
\begin{align} \label{eq:mainmodel}
    -\Delta u - \ww^2 u = \ww^2 \tau \chi_D (1 + |u|^2) u + f,
\end{align}
subject to an appropriate outgoing radiation condition. Here, $f$ represents a given source derived from the incident wave, $D$ denotes the configuration of particles, and $\tau > 0$ is the contrast parameter, which we assume to be sufficiently large. The term $|u|^2$ arises from the Kerr nonlinearity. 

We study the associated resonance problem for \eqref{eq:mainmodel}, \textit{i.e.}, with $f = 0$, building on previous works in the linear regime \cites{ammari2019subwavelength,ammari2020superresolution,ammari2023mathematical,ammari2024fano}. Using the volume potential \eqref{eq:helmpotential}, which is closely related to the scattering resolvent of the linear Helmholtz equation \cite{dyatlov2019mathematical}, we characterize the nonlinear dielectric scattering resonances via a nonlinear Lippmann-Schwinger equation (\textit{cf.}\,\cref{def:scareso}).  By an argument similar to Rellich's lemma, it is straightforward to show that nonlinear dielectric resonances, like their linear counterparts, are located in the lower complex plane $\{\ww \in \C \,;\, \Im \ww < 0\}$ and are symmetric with respect to the imaginary axis (\textit{cf.}\,\cref{prop:nonvanish}). However, unlike the linear problem, discussing the existence of nonlinear resonances requires fixing the amplitude of the corresponding resonant states, namely, enforcing the normalization constraint \eqref{eq:normalizecond}. First, by leveraging the Lyapunov-Schmidt reduction technique, which helps overcome the $S^1$ equivariance (\textit{cf.}\,\cref{lem:invarphaseu}), along with the implicit function theorem in Banach spaces \cite{nirenberg1974topics}, in \cref{thm:3dexist,thm:2dexist}, we demonstrate, for both two- and three-dimensional settings, the existence of nonlinear subwavelength resonances with small-amplitude resonant states bifurcating from the zero solution at linear resonances. Additionally, we derive their asymptotic expansions with respect to the high-contrast parameter $\tau$ and the normalization constant $\mc{N}$. The analysis relies heavily on previously established results for linear subwavelength dielectric resonances, which we recall in \cref{sec:lineareso}. 

Our primary focus is on the principal resonances, \textit{i.e.}, the resonant frequencies with the smallest real parts, analogous to the ground-state energy in quantum systems. In this case, both the linear and nonlinear resonant states are unique; see the discussion at the end of \cref{sec:lineareso}.  
The second main part of this work (\cref{sec:existasympdimer}) is devoted to the analysis of nonlinear subwavelength resonances for a symmetric dimer of dielectric resonators, a setting reminiscent of the Schrödinger operators with double-well potentials \cite{harrell1980double}. For such a symmetric configuration, the small-amplitude nonlinear resonant states constructed in \cref{sec:existasymp} must either be symmetric or antisymmetric, and the principal resonant states are necessarily symmetric (\textit{cf.}\,\cref{lem:symresostate,coro:smallsym}).  
Motivated by the analogy between optics and quantum physics, we extend the results in \cite{kirr2008symmetry} for nonlinear Schrödinger equations to our nonlinear scattering problem. Specifically, in the three-dimensional case, we show that when the normalization constant $\mc{N}$ is sufficiently small, these small-amplitude symmetric or antisymmetric states are the only nontrivial solutions to the nonlinear resonance problem. However, as $\mc{N}$ increases beyond a critical value $\mc{N}_{\rm crit}$, a second symmetry-breaking bifurcation occurs along the principal symmetric solution branch due to the hybridization of the two principal symmetric and antisymmetric states. The resonant states associated with this symmetry-breaking bifurcation are asymmetric and tend to localize predominantly in one of the dielectric particles in the symmetric dimer; see \cref{thm:2ndbirfuc} and the remarks following it.  
Interestingly, in the two-dimensional case, the principal two nonlinear resonances scale as $\Or(1 / \sqrt{\tau \ln \tau})$ and $\Or(1)$, respectively. This separation of scaling implies that, while mode hybridization does still exist, it is insufficient to induce a symmetry-breaking bifurcation on the principal solution branches; see \cref{thm:2d2ndbif} for the mathematical statement.

\subsection{Related works}
Compared to the relatively limited resonance analyses in the literature, there is already a substantial body of well-posedness and numerical studies for the nonlinear Helmholtz equation, which generally takes the form:  
\begin{align} \label{eq:genenlh}  
    - \Delta u - \ww^2 u = f(x,u)\,,  
\end{align}  
subject to a radiation condition. For numerical solutions, the equation is typically approximated on a bounded domain with appropriate boundary conditions. In \cite{wu2018finite}, a mathematical and numerical analysis of \eqref{eq:genenlh} with localized Kerr nonlinearity and impedance boundary conditions was carried out. The study established the well-posedness of \eqref{eq:genenlh}, developed its finite element approximation, and derived stability estimates for the numerical solution with explicit dependence on the wave number. These results were extended in \cite{jiang2022finite} to incorporate the more accurate perfectly matched layer (PML) boundary condition.  Earlier works include \cites{fibich2001high,fibich2005numerical,baruch2007high}, among others. For instance, \cite{fibich2001high} proposed a fourth-order finite difference scheme to solve \eqref{eq:genenlh} using a so-called two-way artificial boundary condition. This approach was later refined in \cite{fibich2005numerical} by introducing Sommerfeld-type local radiation boundary conditions, offering improved numerical performance.

For the well-posedness of the nonlinear scattering problem \eqref{eq:genenlh}, most existing works are restricted to the case of incident waves of small amplitude. In this case, the nonlinear problem can be treated as a perturbation of the linear regime, enabling the use of contraction mappings in conjunction with resolvent estimates for the linear Helmholtz operator. For instance, \cite{jalade2004inverse} established well-posedness in the three-dimensional case for incident plane waves of small amplitude and compactly supported nonlinearities. Similarly, \cite{gutierrez2004non} proved the existence of solutions to \eqref{eq:genenlh} with cubic power nonlinearity in dimensions $d = 3, 4$, for small Herglotz-type incident waves (which do not include plane waves). More recently, \cite{gell2020existence} extended the results of \cite{gutierrez2004non} to accommodate more general nonlinearities, albeit with a more restricted class of incident Herglotz waves. To remove the smallness assumption on the incident wave, \cite{chen2021complex} employed topological fixed-point theory and global bifurcation theory to solve the Lippmann-Schwinger equation corresponding to \eqref{eq:genenlh}, leveraging new a priori estimates for the solutions. Specifically, the authors of \cite{chen2021complex} considered the three-dimensional case ($d = 3$) and proved well-posedness for a class of linearly bounded nonlinearities $f(x,u)$, and for power-type nonlinearities of the form $f(x,u) = Q(x) |u|^{p - 2} u$ in the defocusing case $Q \leq 0$, assuming that $Q$ is compactly supported.

We emphasize that all the results reviewed above focus on complex-valued solutions $u$, generally assuming a fixed frequency $\ww > 0$. In contrast, there has been a growing body of work in recent years aimed at establishing the existence of real-valued standing wave solutions of \eqref{eq:mainmodel}; see \cites{evequoz2013real,evequoz2015dual,evequoz2020dual,mandel2017oscillating} and references therein.  
In particular, by reducing the problem to finding critical points of an energy functional, 
\cite{evequoz2013real} developed a variational method to prove the existence of nontrivial solutions in the case of compactly supported nonlinearities, including $f(x,u) = Q(x) |u|^{p - 2} u$ with $Q$ compactly supported. Inspired by \cite{ambrosetti1973dual}, the later work \cite{evequoz2015dual} introduced a dual variational framework to detect real solutions of \eqref{eq:genenlh} in dimension $d = 3$, for nonlinearities of the form $f(x,u) = Q(x) |u|^{p - 2} u$ with periodic $Q$ or $Q(x) \to 0$ as $Q(x) \to 0$. This variational approach has been further developed in \cites{evequoz2020dual,mandel2017oscillating,mandel2025dual} for various types of nonlinearities. However, our work focuses on scattering resonances, a setting lacking a variational structure, which distinguishes it from the mentioned studies.

\subsection*{Notation} We will use standard Sobolev spaces on a domain $D \subset \mathbb{R}^d$. For example, $ H^2_{\text{loc}}(D)$ denotes the space of functions that locally have $H^2$ regularity. For two complex Banach spaces $X$ and $Y$, we denote by $\mathcal{L}(X, Y)$ the space of bounded linear operators from $X$ to $Y$, and by $\mathcal{L}(X)$ if $Y = X$. We will use $\mathcal{L}_{\mathbb{R}}(X, Y)$ to emphasize the real linearity. We write \(\| \cdot \|_X \) for the norm on the space \( X \), and \( \| \cdot \|_{X \to Y} \) for the operator norm in \( \mathcal{L}(X, Y) \). For simplicity, we denote the \( L^2 \)-inner product on \( D \) by \( \langle \cdot, \cdot \rangle_D \), \textit{i.e.}, \( \langle f, g \rangle_D = \int_D \overline{f} g \, \mathrm{d}x \), and we use the partial derivative symbol $\partial$ to denote the Fr\'echet derivative. Moreover, $U_a$ represents a generic neighborhood of $a \in X$ in $X$, and $C > 0$ denotes a generic positive constant, which may vary from line to line.

\section{Problem formulation and preliminaries}

\subsection{Nonlinear wave scattering} \label{sec:basicsetting}
In this section, we mathematically formulate the nonlinear dielectric resonance problem for the wave scattering by resonators with high refractive indices and a cubic nonlinearity of Kerr type. 


Let $D \subset \mathbb{R}^d$ ($d \ge 2$) be a \emph{bounded connected open} subset with a sufficiently \emph{smooth boundary} $\partial D$ and unit outer normal vector  $\nu$, and let $\chi_D$ denote the associated characteristic function. Without loss of generality, by scaling, suppose that $D$ has a characteristic size of order one. We denote by $|D|$ the volume of the domain $D$. 
We consider the nonlinear wave equation on $\R^d$ for $d = 2, 3$, 
\begin{align} \label{eq:timedepmodel}
  \frac{\p^2}{\p t^2} u(t,x) - \Delta u(t,x) = h(x, u(t,x))\,,\q (t, x) \in \R \ti \R^d\,,
\end{align}
which models acoustic or electromagnetic scattering by dielectric particles with high refractive indices and nonlinear responses. Specifically, we assume that the nonlinearity  $h(x, u)$ takes the following form \cites{griesmaier2022inverse,ammari2025nonlinear}: for $\ww \in \mathbb{C}$,
\begin{align} \label{eq:nonterm}
   h(x,u(t,x)) := \ww^2 n(x)^2 q_\eta(u(t,x))u(t,x)\,,
\end{align}
where $n(x)$ represents the relative refractive index of high-contrast particles, defined as 
\begin{equation} \label{eq:contrast}
    n(x)^2 := \tau \chi_D(x),\q \tau \gg 1\,,
\end{equation}
with real $\tau$ being the contrast and the function 
\begin{align} \label{eq:nonfunc}
  q_\eta(z) := 1 + \eta |z|^2:\ \C^d \to \R\,,
\end{align}
models the nonlinear response of the refractive index to the intensity of the wave field. Here, \( \eta > 0 \) is a real parameter that characterizes the strength of nonlinearity.

To formulate the associated scattering resonance problem, we consider a time-harmonic wave ansatz with the same frequency $\omega \in \mathbb{C}$ as the one in \eqref{eq:nonterm}:
\begin{align*}
  u(t,x) = u(x) e^{-i \omega t}\,.
\end{align*}
Substituting this ansatz into \eqref{eq:timedepmodel} leads to the following nonlinear Helmholtz equation:
\begin{equation} \label{eq:timeharhelm}
  - \Delta u - \ww^2 u = \ww^2 \tau \chi_D q_\eta(u) u\,,
\end{equation}
subject to the transmission boundary conditions on  $\partial D$:
\begin{align*}
 u|_+ = u|_-, \quad \frac{\p u}{\p \nu}\Big|_+ = \frac{\p u}{\p \nu}\Big|_-\,,
\end{align*}
and the Sommerfeld radiation condition for $\ww > 0$:
\begin{align} \label{eq:radiation}
   r^{\frac{d - 1}{2}} \left| \frac{\p u}{\p r} - i \ww u \right| \to 0\,,\q \text{as $r = |x| \to \infty$}. 
\end{align}
When $\ww$ is complex, the outgoing radiation condition can be specified using the free-space Green's function, as in the case of linear scattering. 
Moreover, observe that the function $1 + \tau \chi_D q_\eta(u)$ denotes the nonlinear refractive index in $\mathbb{R}^d$, characterized by high contrast and Kerr-type nonlinearity.

A nontrivial solution $u \neq 0$ of the eigenvalue problem \eqref{eq:timeharhelm} is referred to as a \emph{resonant state}, and the corresponding frequency $\ww \in \C$ is called the 
\emph{nonlinear dielectric resonance}. 
In this work, we shall investigate this nonlinear scattering resonance problem using the Lippmann-Schwinger equation, building on the framework developed in \cites{ammari2023mathematical,ammari2024fano} for the linear case.  Let \(G^\ww(x) \) denote the outgoing Green's function of \( -(\Delta + \omega^2) \) in \( \mathbb{R}^d\) ($d = 2, 3$), given by, for $\ww \in \C$, 
\begin{align} \label{eq:green2d3d}
  G^\ww(x) := \begin{dcases}
     \frac{i}{4}H_0^{(1)}(\ww|x|)\,, &\  d = 2\,,\\ 
      \frac{e^{i \omega |x|}}{4 \pi |x|}\,, &\  d = 3\,,
  \end{dcases} 
\end{align}
where $H_0^{(1)}$ is the Hankel function of the first kind and of order zero.
We define the associated volume integral operator by 
\begin{align} \label{eq:helmpotential}
    \mathcal{K}_D^\omega[\varphi] := \int_D G^\ww(x-y) \varphi(y) \, \mathrm{d}y: \quad L^2(D) \to H^2(D)\,,
\end{align}
which satisfies 
\begin{equation*}
  - (\Delta + \omega^2) \kk_D^\omega[\varphi] = \chi_D \varphi \q \text{on $\R^d$}\,.
\end{equation*}
It is helpful to recall the Sobolev embedding $ H^2(D) \hookrightarrow C^{0,\frac{1}{2}}(\wb{D})$ when $d = 3$, and $H^2(D) \hookrightarrow C^{0,\gamma}(\wb{D})$ for any $0 < \gamma < 1$ when $d = 2$, and the Rellich-Kondrachov compact embedding $H^2(D) \Subset C(\wb{D})$; see \cite{ciarlet2013linear}*{Theorems 6.6-1 and 6.6-3}, where $C^{0,\gamma}(\wb{D})$ denotes the $\gamma$-H\"older space. Then the following lemma holds. 
\begin{lemma} \label{lem:compactop}
    The volume potential $\mathcal{K}_D^\omega$ in \eqref{eq:helmpotential}, with $\omega \in \mathbb{C}$, is a compact linear operator on both $L^2(D)$ and $C(\wb{D})$.
\end{lemma}

Moreover, we have the following regularity result for $\kk_D^\ww$ from the standard theory of elliptic PDEs; see \cite{gilbarg1977elliptic}*{Lemmas 4.1 and 4.2} and \cite{yang2020newtonian}*{Theorem 2.3}. 

\begin{lemma} \label{lem:regkdw}
Let $D \subset \R^d$, $ d \ge 2$, be a bounded connected open set. It holds that, for $\ww \in \C$, 
\begin{align} \label{eq:reg1}
    \kk_D^\ww[\vp] \in C^1(\R^d)\,,\q \forall\, \vp \in L^\infty(D)\,,
\end{align}
and 
\begin{align} \label{eq:reg2}
    \kk_D^\ww[\vp] \in C^2(D)\,,\q \forall\, \vp \in C^{0,\alpha}(D)\,,\ \alpha \in (0,1)\,.
\end{align}
\end{lemma}

We are now ready to formulate the nonlinear eigenvalue problem \eqref{eq:timeharhelm} in $(\ww,u)$ in terms of a Lippmann-Schwinger equation. 

\begin{definition} \label{def:scareso}
$\ww \in \C$ is a \emph{nonlinear dielectric (scattering) resonance} for the nonlinear Helmholtz equation \eqref{eq:timeharhelm} if there is a nontrivial $0 \not \equiv u \in H^2(D)$ such that
\begin{align} \label{eq:nonLippscalar}
    u = \tau \ww^2 \kk_D^\ww \left[q_\eta(u)u\right],
\end{align}
where the function $q_\eta(\dd)$ is given in \eqref{eq:nonfunc}. In particular, a nonlinear resonance $\ww$ is called \emph{subwavelength} if $\ww(\tau) \to 0$ as $\tau \to \infty$.   
\end{definition}

To justify the well-posedness of equation \eqref{eq:nonLippscalar}, it suffices to note from the compact embedding $H^2(D) \Subset C(\wb{D})$ that for $u \in H^2(D)$, 
\begin{align} \label{eq:regularitylipp}
  q_\eta(u) u \in L^2(D)\,,\q \kk_D^\ww\left[q_\eta(u)u\right] \in H^2(D) \Subset C(\wb{D})\,,
\end{align}
and any $u \in H^2(D)$ satisfying \eqref{eq:nonLippscalar} can be extended to a solution $u \in H^2_{\rm loc}(\mathbb{R}^d)$ of the Helmholtz equation \eqref{eq:timeharhelm}. 
This also shows that one can equivalently seek the solution to \eqref{eq:nonLippscalar} in the space $L^\infty(D)$ or $C(\wb{D})$. Indeed, if $u \in L^\infty(D)$ or $C(\wb{D})$, the regularity \eqref{eq:regularitylipp} still holds and the equation \eqref{eq:nonLippscalar} is well-defined on $L^\infty(D)$ or $C(\wb{D})$. 

The next result shows that any solution $u \in H^2(D)$ to \eqref{eq:nonLippscalar} is also a classical solution of the nonlinear Helmholtz equation \eqref{eq:timeharhelm}. 

\begin{proposition}
    Let $u \in H^2(D)$ be a solution to the Lippmann-Schwinger equation \eqref{eq:nonLippscalar}. Then $u \in C^2(D)$ solves the Helmholtz equation \eqref{eq:timeharhelm} on $D$. 
\end{proposition}

\begin{proof}
    It suffices to note that $q_\eta(u)u \in L^\infty(D)$ for $u \in H^2(D) \Subset C(\wb{D})$, and \eqref{eq:reg1} in \cref{lem:regkdw} gives $u = \tau \ww^2 \kk_D^\ww \left[q_\eta(u)u\right] \in C^1(\R^d)$ and $q_\eta(u)u \in C^1(D)$. Then applying the property \eqref{eq:reg2}, we can conclude that $u \in C^2(D)$ is a classical solution. 
\end{proof}

In the linear case of $\eta = 0$, we have that $q_\eta(u) = 1$ and therefore \eqref{eq:nonLippscalar} reduces to 
\begin{align} \label{prob:lineareso}
u = \tau \ww^2 \kk_D^\ww \left[u\right].  
\end{align}
Hence, linear dielectric resonances are characterized by the characteristic values of the analytic Fredholm operators $1 - \tau \ww^2 \kk_D^\ww$. In \cites{ammari2019subwavelength,ammari2023mathematical}, linear resonances have been shown to exist in the subwavelength regime $\{\ww \in \C\,;\ |\ww| \ll 1\}$ when the contrast $\tau$ is large enough. In \cref{sec:lineareso} below, we present the fundamental properties of linear subwavelength dielectric resonances and derive their asymptotic formulas. These results refine those in \cite{ammari2019subwavelength} using the analytical framework developed in \cites{ammari2023mathematical,ammari2024fano}, which will serve as a basis for our subsequent discussions on the nonlinear case (see \cref{sec:existasymp,sec:existasympdimer}). 

Analogously to the linear resonance regime, we obtain the following result on the distribution of nonlinear resonances. The proof follows arguments similar to those in \cite{ammari2023mathematical}*{Lemma 3.19} and is provided in \cref{app:A} for completeness. 


\begin{proposition} \label{prop:nonvanish}
All the nonlinear dielectric resonances lie in the lower half-plane $\{\omega \in \mathbb{C}\,; \Im \omega < 0\}$ and are symmetric about the imaginary axis.
\end{proposition}

\subsection{Linear dielectric resonances and asymptotics}  \label{sec:lineareso} We start with the asymptotic expansion of the Green's function $G^\ww(x)$ in $\ww \in \C$ near zero: for $d = 3$, 
\begin{equation} \label{eq:expgreengn}
    G^\ww(x) = \sum_{n = 0}^\infty \ww^n G_n(x) \quad \text{with}\q G_n(x) := \frac{i^n|x|^{n-1}}{4 \pi n!}\,,
\end{equation}
while for $d = 2$, we have 
\begin{equation} \label{eq:expgreengn2d}
    G^\ww(x) = \h{G}_0^\ww + \sum_{j=1}^\infty \left(\omega^{2j} \ln \omega \right) G_n^{(1)}(x) + \sum_{j=1}^\infty \omega^{2j} G_n^{(2)}(x)\,,
\end{equation}
where 
\begin{align} \label{def:expgf2}
  \h{G}_0^\ww(x) := - \frac{1}{2\pi} \ln |x| - \eta_\omega\,, \q G_n^{(1)}(x) := - b_j |x|^{2j}\,,\q G_n^{(2)}(x) := - (b_j \ln |x| + c_j) |x|^{2 j}\,,
\end{align}
with 
\begin{align*}
   b_j := \frac{(-1)^j}{2\pi} \frac{1}{2^{2j}(j!)^2}\,, \quad c_j := b_j \left(\gamma - \ln 2 - \frac{\pi i}{2} - \sum_{l=1}^j \frac{1}{l} \right)\,,
\end{align*}
and 
\begin{align} \label{eq:constgaome}
    \eta_\omega := \frac{1}{2\pi} \big(\ln \omega + \h{\gamma} \big)\,,\q \h{\gamma} := \gamma - \ln 2 - \frac{i \pi}{2}\,.
\end{align}
Here, $\gamma$ is the Euler constant.
Thanks to the expansions of $G^\ww(x)$ in \eqref{eq:expgreengn} and \eqref{eq:expgreengn2d}, we have the following lemma. 
\begin{lemma} \label{lem:opasym}
  For $d = 3$, the volume potential $\kk_D^\ww$ in \eqref{eq:helmpotential} has the expansion:
\begin{align} \label{eq:expop3d}
   \mc{K}_D^{\ww} = \kk_D + \sum_{j = 1}^\infty \ww^j \kb{j}\q \text{with} \q   \kb{j}[\vp] := \int_D G_j(x-y) \vp(y) \ud y\,,
\end{align}
where $\{G_n\}$ are from \eqref{eq:expgreengn}, and the series converges in the operator norm. In particular, 
\begin{align} \label{eq:leading3d}
  \mc{K}_{D}[\vp] := \int_D \frac{1}{4 \pi |x-y|} \vp(y) \ud y\,,\q \mc{K}_{D,1}[\vp] := \frac{i}{4 \pi} \int_D \vp(y) \ud y\,.
\end{align}
For $d = 2$, a similar expansion of $\kk_D^\ww$ holds: 
\begin{align} \label{eq:expop2d}
    \kk^\omega_D = \h{\kk}^\omega_D + \sum_{j=1}^\infty \left(\omega^{2j} \ln \omega \right) \kk^{(1)}_{D,j} + \sum_{j=1}^\infty \omega^{2j} \kk^{(2)}_{D,j}\,,
\end{align}
where 
\begin{align} \label{eq:leading2d}
    \h{\kk}^\omega_D[\vp] := \kk_D[\vp] - \eta_\ww \int_D \vp(y) \ud y\,,\q \kk_D[\vp] := \int_D - \frac{1}{2 \pi} \ln|x-y| \vp(y) \ud y\,,
\end{align}
and with $\{G^{(i)}_n\}$, $i = 1,2$, from \eqref{def:expgf2}, 
\begin{align*}
  \kk^{(1)}_{D,j}[\vp](x) := \int_{D} G_n^{(1)}(x-y) \vp(y) \ud y\,,\q \kk^{(2)}_{D,j}[\vp](x) := \int_{D} G_n^{(2)}(x-y) \vp(y) \ud y\,.
\end{align*}
\end{lemma}
Clearly,  $\kk_D[\dd]$ in \eqref{eq:leading3d} and \eqref{eq:leading2d} is the integral operator associated with the static Laplace Green function $G^{\ww = 0}$, known as \emph{Newtonian potential}, which is compact and self-adjoint on $L^2(D)$. Next, we consider the linear dielectric resonance problem \eqref{prob:lineareso}. Following \cite{ammari2023mathematical}, we introduce the \emph{scaled frequency} $\h{\ww} \in \C$ by 
\begin{align} \label{eq:scalefreq}
    \h{\ww} := \epsilon^{-1} \ww\,, \q \epsilon := \sqrt{\tau}^{-1},
\end{align}
and rewrite \eqref{prob:lineareso} as $u = \h{\ww}^2 \kk_D^{\epsilon \h{\ww}} [u]$, which, in the case of $d = 3$, reduces to a linear eigenvalue problem, as $\tau \to \infty$ (\textit{i.e.}, $\eps \to 0$), 
\begin{align*}
  u = \h{\ww}^2\kk_D[u]\,.
\end{align*}
Then, by direct application of the Gohberg-Sigal theory \cites{gohberg1971operator,gohberg1990classes}, one can conclude that the following result holds. 
\begin{proposition} \label{prop:limit_scalar3d}
Let $d = 3$. When the contrast \( \tau \) is sufficiently large, the scattering resonances for \eqref{prob:lineareso} exist in the subwavelength regime, which has asymptotic behavior:
\begin{align*} 
  \omega(\tau) = \frac{1}{\sqrt{\tau \lambda_j}} + \mathcal{O}(\tau^{-1})\,,  \q \text{as $\tau \to \infty$}\,,
\end{align*}
with the associated normalized resonant state: 
\begin{align} \label{eq:limit_resta3d}
  u = \vp_j + \Or(\sqrt{\tau}^{-1})\,,\q \norm{\vp_j}_{L^2(D)} = 1\,.
\end{align}
Here, $(\lambda_j,\vp_j)$ is some eigenpair of Newtonian potential $\kk_D$, \textit{i.e.}, $\kk_D[\vp_j] = \lad_j \vp_j$. 
\end{proposition}

The two-dimensional case is similar but more complicated by the logarithmic singularity of $G^{\omega}$ as $\omega \to 0$; see the expansion in \eqref{eq:expgreengn2d}. To address this issue, we define the space 
\begin{align*}
L^2_0(D) := \left\{u \in L^2(D) \,;\ \int_D u \ud x = 0 \right\},
\end{align*}
and decompose a function $u \in L^2(D)$ as  
\begin{align} \label{eq:decompol2}
  u = \langle 1_D, u \rangle_D + \bar{u}\,,\q  \bar{u} \in L^2_0(D)\,,
\end{align}
where $1_D := 1 / |D|$ is the normalized constant function with $|D|$ being the volume of $D$.  Note that this decomposition is unique and orthogonal. Using \eqref{eq:decompol2}, equation \eqref{prob:lineareso} can be equivalently formulated as follows: 
\begin{align} \label{eq:limit2d}
    \mm 
     \l 1_D, u\r_D \\
    \bar{u}
    \nn -
    \h{\ww}^2 \mm
     \l 1_D, \mc{K}^{\epsilon \h{\ww}}_D[1]\r_D &   \l 1_D,   \mc{K}^{\epsilon \h{\ww}}_D[\dd]\r_D \\
     \mc{K}^{\epsilon \h{\ww}}_D[1] -  \l 1_D,  \mc{K}^{\epsilon \h{\ww}}_D[1]\r_D  &  \mc{K}^{\epsilon \h{\ww}}_D[\dd] - \l 1_D,   \mc{K}^{\epsilon \h{\ww}}_D[\dd]\r_D
    \nn  \mm 
    \l 1_D, u\r_D \\
    \bar{u}
    \nn = 0\,,
\end{align}
where we have used the scaled frequency $\h{\ww} = \sqrt{\tau} \ww$ in \eqref{eq:scalefreq}. Recall from \eqref{eq:expop2d} that 
\begin{align} \label{eq:npt2d}
      \kk^{\epsilon\h{\ww}}_D[\vp] = \kk_D[\vp] - \eta_{\epsilon \h{\ww}} \int_D \vp(y) \ud y + \mc{O}\left( (\epsilon \h{\ww})^2 \ln \left(\epsilon \h{\ww}\right)\right),
\end{align}
where the constant $\eta_{\eps \h{\ww}} = \frac{1}{2\pi} (\ln (\eps \h{\ww}) + \h{\gamma})$ is given in \eqref{eq:constgaome}. It follows from \eqref{eq:limit2d} that 
\begin{align} \label{eq:asym_form_eig}
 \mm 1 + \h{\ww}^2\eta_{\eps \h{\ww}} |D| + \mc{O}(\h{\ww}^2)
    & \mc{O}(\h{\ww}^2) \\  -
    \h{\ww}^2
    \w{\phi}  + \mc{O}(\eps^2 \h{\ww}^4 \ln(\eps \h{\ww}))  & 1 -
    \h{\ww}^2 \w{\mc{K}}_D + \mc{O}(\eps^2 \h{\ww}^4 \ln(\eps \h{\ww})) 
    \nn   \mm 
    \l 1_D, u\r_D \\
    \bar{u}
    \nn = 0\,,
\end{align}
with the function:  
\begin{align} \label{auxfunc2d}
    \w{\phi} := \kk_D[1] - \l 1_D, \kk_D[1] \r_{D} \in L^2_0(D)\,,
\end{align}
and the compact self-adjoint operator: 
\begin{align} \label{eq:limopscal}
\w{\kk}_D[\dd]:= \kk_D[\dd] - \l 1_D,  \kk_D[\dd] \r_{D}: \ L_0^2(D) \to L^2_0(D)\,.
\end{align}

We now demonstrate that resonances occur in the following two regimes: 
\begin{align*}
   \hat{\omega}^2 \ln (\eps \hat{\omega}) = \mathcal{O}(1) \q \text{and} \q  \hat{\omega} = \mathcal{O}(1)\,.
\end{align*}
The discussion in the following is outlined, but a rigorous analysis can be carried out similarly to \cites{ammari2023mathematical} using the Gohberg--Sigal theory. 
In the first case, we have $\h{\ww} = o(1)$ as $\eps = \sqrt{\tau}^{-1} \to 0$, and there holds $\h{\ww} = \mc{O}( \sqrt{-\ln \eps}^{-1})$. Equation \eqref{eq:asym_form_eig} admits the following asymptotics: 
\begin{align*} 
\mm 1 + \frac{\h{\ww}^2 \ln (\eps \h{\ww})}{2 \pi} |D|
    & 0 \\ 0 & 1 
    \nn   \mm 
    \l 1_D, u\r_D \\
    \bar{u}
    \nn  = \mc{O}\left(\frac{1}{\ln \eps}\right).
\end{align*}
It follows that a scaled resonant frequency $\h{\ww}$ exists with the leading-order term $\h{\ww}_0$ given by 
\begin{align} \label{eq:leadreso2d}
  \h{\ww}_0^2 \ln (\eps \h{\ww}_0) = - \frac{2 \pi}{|D|}\,,
\end{align}
and the estimate:
\begin{align*}
    \h{\ww} = \h{\ww}_0 + \mc{O}\left(\frac{1}{\ln \eps}\right). 
\end{align*}
The associated resonant state $u$ is almost constant on $D$ in the sense that $u = a + \Or((\ln \eps)^{-1})$ for some $a \in \C$.  To solve \eqref{eq:leadreso2d}, we introduce $s := 2 \ln (\eps \h{\ww})$ and then \eqref{eq:leadreso2d} gives 
\begin{equation*}
  s e^s =  - \frac{4 \pi}{|D|} \eps^2\,,
\end{equation*}
the real solution to which is expressed in terms of the Lambert $W$ function \cite{corless1996lambert}:
\begin{align} \label{eq:solus}
  s = W_{-1} \left(- \frac{4 \pi}{|D|} \eps^2\right). 
\end{align}
Here, we take the branch $W_{-1}$ because the argument $- 4 \pi \eps^2/|D|$ is small and negative, and the solution $s$ should be  large and negative due to $s = 2 \ln (\eps \h{\ww}) = \mc{O}(1/\h{\ww}^2) = \mc{O}(\ln \eps)$ by \eqref{eq:leadreso2d}. Thanks to the expansion $W_{-1}(x)$ for small $x < 0$:  
\begin{align*}
  W_{-1}(x) = \ln(-x) - \ln(-\ln(-x)) + \text{h.o.t.}\,,
\end{align*}
we have, from \eqref{eq:leadreso2d} and \eqref{eq:solus}, 
\begin{equation} \label{eq:2dresolea}
    \begin{aligned}
         \h{\ww}_0 = \sqrt{- \frac{4 \pi}{|D| s}} &  = \sqrt{\frac{4 \pi/|D|}{ - W_{-1} \left(- 4 \pi \eps^2/|D|\right)}} \\
            & = \sqrt{\frac{4 \pi/|D|}{ - \ln \left(\eps^2\right)}} + \Or\left( \frac{\ln(-\ln \left(\eps \right))}{\ln \left(\eps\right)} \right) + \Or\left(\frac{1}{\ln \eps}\right).
    \end{aligned}
\end{equation}

In the second regime of $\h{\ww} = \mc{O}(1)$, we divide the first row of \eqref{eq:asym_form_eig} by $\ln \eps$ and then find 
\begin{align}  \label{eq:asymptsec}
  &  \mm  \frac{\h{\ww}^2}{2 \pi} |D| 
    & 0 \\  -
    \h{\ww}^2
    \w{\phi}  & 1 -
    \h{\ww}^2 \w{\mc{K}}_D 
   \nn  \mm 
    \l 1_D, u\r_D \\
    \bar{u}
    \nn  = \mc{O}\left(\frac{1}{\ln \eps}\right).
\end{align}
Let the compact self-adjoint operator $\w{\mc{K}}_D$ in \eqref{eq:limopscal} admit the spectral decomposition: 
\begin{align}\label{eq:specwkd}
    \w{\kk}_D[u] = \sum_{j = 0}^\infty \mu_j  \l v_j, u\r_D  v_j\,,\q u \in L^2_0(D)\,.
\end{align}
Then, the asymptotics \eqref{eq:asymptsec} readily implies another class of (scaled) resonances:
\begin{align*}
   \h{\ww} = \frac{1}{\sqrt{\mu_j}} + \mc{O}\left(\frac{1}{\ln \eps}\right)\,,
\end{align*}
with associated resonant states $u = a v_j + \Or((\ln \eps)^{-1})$ on $D$ for $a \in \C$. We summarize the above discussion for the linear dielectric resonance
in the two-dimensional case as follows. 

\begin{proposition} \label{prop:limit_scalar2d}
Let $d = 2$. For the large enough contrast $\tau$, the scattering resonances $\ww$ for \eqref{prob:lineareso} exist in the subwavelength regime. Moreover, any resonance $\ww$ either satisfies
\begin{align} \label{asym_scalar}
 \ww(\tau) = \sqrt{\frac{4 \pi}{ |D| \tau \ln \tau }} + \Or\left( \frac{\ln(\ln \tau
  )}{\sqrt{\tau} \ln \tau} \right),
\end{align}
with the $L^2$ normalized resonant state:
$$
 u = \frac{1}{\sqrt{|D|}} + \Or\left(\frac{1}{\ln \tau}\right),
$$
or there holds
\begin{align} \label{asym_scalar2}
   \ww(\tau) = \frac{1}{\sqrt{\tau \mu_j}} + \Or\left(\frac{1}{\sqrt{\tau} \ln \tau}\right) \q \text{for some}\ j \ge 0\,, 
\end{align}
with the $L^2$ normalized resonant state:
$$
u = v_j + \Or\left(\frac{1}{\ln \tau}\right),
$$
where $(\mu_j, v_j)$ satisfies 
\begin{align*}
  \w{\mc{K}}_D[v_j] = \mu_j v_j\,, \q \norm{v_j}_{L^2(D)} = 1\,,
\end{align*}
\textit{i.e.}, it is an eigenpair of $\w{\mc{K}}_D$ in \eqref{eq:limopscal} on $L_0^2(D)$. 
\end{proposition}

The above \cref{prop:limit_scalar3d,prop:limit_scalar2d} address the existence of linear dielectric resonances and their leading-order behaviors.
The higher order asymptotics of a subwavelength resonance $\omega(\tau)$ in terms of the contrast $\tau$ can be derived using the generalized argument principle \cite{gohberg1971operator} and the theory of symmetric polynomials \cite{artin2011algebra}; see \cite{ammari2023mathematical}*{Section 3.3} for details. In particular, it can be shown that $\omega(\tau)$ is, in general, a multivalued algebraic function with potential algebraic singularities at $\tau = \infty$, if the associated eigenvalue $\lambda_j$ of  $\mathcal{K}_D$ (see \cref{prop:limit_scalar3d}) or $\mu_j$ of  $\widetilde{\mathcal{K}}_D$ (see \cref{prop:limit_scalar2d}) has multiplicity greater than one.  
It should be noted that, in the three-dimensional case, the Krein-Rutman theorem \cite{deimling2013nonlinear} ensures that the largest eigenvalue of $\mathcal{K}_D$, called \emph{principal eigenvalue} and denoted by $\lambda_0$, is simple. Specifically, we recall the following result.

\begin{lemma}[Krein-Rutman Theorem]\label{lem:krthm}
 Let $X$ be a Banach space, and let $K \subset X$ be a cone such that the set  $\{u - v\,;\ u, v \in K\}$ is dense in $X$. Assume that $T : X \to X$ is a compact linear operator that is positive, \textit{i.e.}, $T(K) \subset K$, and has a positive spectral radius $r(T) > 0$. Then $r(T)$ is an eigenvalue of $T$ with an eigenvector $0 \neq u \in K$, such that $Tu = r(T)u$.
Moreover, if the cone \( K \) has a nonempty interior\footnote{In this case, it necessarily holds that \( \{u - v \,;\ u, v \in K\} = X \).} \( K^o \), and if \( T \) is strongly positive, \textit{i.e.}, $T\left(K\backslash\{0\}\right) \subset K^o$, then \( r(T) > 0 \) is a simple eigenvalue of \( T \) with an eigenvector \( v \in K^o \). Furthermore, \( T \) has no other eigenvalue associated with a positive eigenvector, and \( |\lambda| < r(T) \) for all eigenvalues \( \lambda \neq r(T) \).
\end{lemma}

To apply \cref{lem:krthm} to the Newtonian potential $\kk_D$ for $d = 3$, we recall from \cref{lem:compactop} that $\kk_D$ is a compact linear operator from  $C(\wb{D})$ to $C(\wb{D})$. 
Furthermore, any $L^2$ eigenfunction of $\mathcal{K}_D$ has $H^2$ regularity and is therefore continuous by $H^2(D) \Subset C(\wb{D})$. We define the positive cone $P$ in $C(\wb{D})$ as  
\begin{equation}  \label{eq:conekd}
P := \left\{\varphi \in C(\wb{D}) \,;\ \varphi(x) \geq 0 \text{ for all } x \in \wb{D}\right\},  
\end{equation}  
which has a nonempty interior. Moreover, it is straightforward to verify that $\kk_D$ is strongly positive. That is, if $\varphi(x) \geq 0$ and $\varphi \not\equiv 0 $, then $\mathcal{K}_D[\varphi](x) > 0$ for all $x \in D$. To see this, assume instead that $\mathcal{K}_D[\varphi](x) = 0$ for some $x \in D$. Since the kernel (Green's function) of $\kk_D$: $$G^0(x - y) = \frac{1}{4 \pi |x - y|} > 0\,,$$ is positive, $\mathcal{K}_D[\varphi](x) = 0$ implies $G^0(x - y) \varphi(y) = 0$ and hence $\varphi(y) = 0$ almost everywhere in $D$, contradicting $\varphi \not\equiv 0$. Thus, the Krein-Rutman theorem applies to the compact operator $\mathcal{K}_D$ on the Banach space  $X = C(\wb{D})$ with the cone $P$ in \eqref{eq:conekd}.

\begin{corollary} \label{coro:simplenewton}
The largest eigenvalue $\lambda_0$ of $\mathcal{K}_D: L^2(D) \to L^2(D)$ for $d = 3$ is simple, with a strictly positive eigenfunction $\varphi_0(x) > 0$ for all $x \in D$. Moreover, there are no other positive eigenfunctions except positive multiples of $\varphi_0$. 

As a result, the subwavelength resonance  $\omega(\tau) = 1 / \sqrt{\tau \lambda_0} + \mathcal{O}(\tau^{-1})$, associated with $\lambda_0$, is the unique resonance near $1 / \sqrt{\tau \lambda_0}$ for large enough $\tau$. 
\end{corollary}

\begin{remark}  \label{rem:krthm3d}
    It is worth noting that the above argument based on the Krein-Rutman Theorem does not apply to the operator $\widetilde{\mathcal{K}}_D$ in \eqref{eq:limopscal} for the two-dimensional case. This is because $ \widetilde{\mathcal{K}}_D$ is not a positive operator, and the average-zero space $L_0^2(D)$ does not contain any positive functions.  
\end{remark}  

In this work, a scattering resonance $\omega \in \mathbb{C}$ with the smallest real part shall often be referred to as the \emph{principal resonance}, as its associated resonant state typically dominates the scattered field from a physical perspective. In the high-contrast regime, \textit{i.e.}, when the contrast $\tau$ is sufficiently large, the principal resonance for the three-dimensional case, as shown in \cref{prop:limit_scalar3d,coro:simplenewton}, is given by the subwavelength resonance $\omega \sim 1 / \sqrt{\tau \lambda_0}$. This resonance is of multiplicity one, \textit{i.e.}, $\dim (1 - \tau \omega^2 \mathcal{K}_D^\omega) = 1$, and its associated resonant state is proportional to an approximately positive function.  
A similar conclusion holds in the two-dimensional case by \cref{prop:limit_scalar2d}. In this scenario, the resonance with the smallest real part is also unique and scales as $\Or(1 / \sqrt{\tau \ln \tau})$ from \eqref{asym_scalar}, with the associated resonant state being proportional to a function that is approximately a positive constant. 

\section{Nonlinear dielectric resonances} \label{sec:existasymp}
The remainder of this work is devoted to the analysis of nonlinear dielectric resonances (see \cref{def:scareso}). Without loss of generality, we set the parameter $\eta = 1$ in the following discussion. We define the function $N(z) = z + |z|^2 z$ and the associated map:  
\begin{equation}\label{eq:nonfuncl2}  
  N(u) := u + |u|^2 u : H^2(D) \to L^2(D),  
\end{equation}  
which allows us to reformulate \eqref{eq:nonLippscalar} as  
\begin{align} \label{eq:nonLippscalar0}  
    u = \tau \omega^2 \mathcal{K}_D^\omega \left[N(u)\right].  
\end{align}
The following estimate is simple but useful: for $u \in H^2(D) \Subset C(\wb{D})$, 
\begin{align} \label{eq:priestcub}
    \norm{|u|^2 u}_{L^2(D)} \le \norm{u}_{L^\infty(D)}^2 \norm{u}_{L^2(D)} \le C \norm{u}_{H^2(D)}^2 \norm{u}_{L^2(D)}\,.
\end{align}
We introduce the scaled frequency 
$\hat{\omega} = \epsilon^{-1} \omega$ with $\epsilon = \sqrt{\tau}^{-1}$ as in \eqref{eq:scalefreq} for the linear case. Under this scaling, equation \eqref{eq:nonLippscalar0} becomes  
\begin{align} \label{eq:nonLippscalar2}  
    u = \hat{\omega}^2 \mathcal{K}_D^{\epsilon \hat{\omega}} \left[N(u)\right].  
\end{align}
However, unlike the linear case \eqref{prob:lineareso}, to study the existence and properties of the nonlinear resonances, we need to restrict the Lippmann-Schwinger equation \eqref{eq:nonLippscalar0} or \eqref{eq:nonLippscalar2} to the level set:
\begin{align} \label{eq:normalizecond}
\norm{u}_{L^2(D)}^2 = \mathcal{N},
\end{align}
where $\mathcal{N} > 0$ is the \emph{normalization constant}.

In this section, we shall show in \cref{thm:3dexist} and \cref{thm:2dexist} for the three-dimensional and two-dimensional cases, respectively,  
that nonlinear subwavelength dielectric resonances exist near the linear ones when the contrast $\tau$ is sufficiently large and the associated resonant state has a small magnitude (\textit{i.e.}, the constant $\mc{N}$ in \eqref{eq:normalizecond} is small). Moreover, we also establish the asymptotic behaviors of these nonlinear resonances and states with respect to $\tau \to \infty$ and $\mc{N} \to 0$.  

\subsection{Three-dimensional case} \label{sec:3dexistence}
Let us first consider the three-dimensional case $d = 3$. For clarity, we recall the spectral decomposition of the Newtonian potential $\mathcal{K}_D$:  
\begin{equation}\label{eq:speckdd}  
    \mathcal{K}_D[u] = \sum_{j=0}^\infty \lambda_j \langle \varphi_j, u \rangle_D \varphi_j\,,  
\end{equation}  
where $\lambda_0 > \lambda_1 \geq \lambda_2 \geq \cdots$, counting multiplicities, and $\|\varphi_j\|_{L^2(D)} = 1$. Additionally, it is straightforward to observe that the function $N(z)$ in \eqref{eq:nonfuncl2} is real differentiable, with  
\begin{align} \label{eq:derivfunc}  
    N'(z)[v] = (1 + 2|z|^2)v + z^2 \bar{v} \in \mathcal{L}_\mathbb{R}(\mathbb{C}, \mathbb{C})\,,  
\end{align}  
and it is equivariant under the transformations $z \to e^{i\theta}z$ with $\theta \in \mathbb{R}$, called \emph{$S^1$ equivariance}, and $z \to \wb{z}$. Then the following basic lemma readily follows. 

\begin{lemma} \label{lem:invarphaseu}
For any function $u$ and $\theta \in \R$, there holds 
  \begin{align*}
      N(e^{i \theta} u) = e^{i \theta} N(u)\,,\q N(\wb{u}) = \wb{N(u)}\,.
  \end{align*}
    Let $(\omega, u) \in \mathbb{C} \times H^2(D)$ be a solution of \eqref{eq:nonLippscalar0}. Then $(-\overline{\omega}, \overline{u})$ and $(\omega, e^{i\theta} u)$ for any $\theta \in \mathbb{R}$ also solve \eqref{eq:nonLippscalar0}.
\end{lemma}

In the high-contrast regime $\tau \gg 1$, with the scaling \eqref{eq:scalefreq}, let $\h{\omega}_*$ denote a scaled linear subwavelength resonance with resonant state $\varphi_*$ and the asymptotic expansion $\h{\omega}_* = 1/\sqrt{\lambda_j} + \mathcal{O}(\sqrt{\tau}^{-1})$ by \cref{prop:limit_scalar3d}. Assuming that  $\lambda_j$ is a simple eigenvalue, which is the case for the principal eigenvalue $\lambda_0$ by \cref{coro:simplenewton}, \cref{thm:smallampl,coro:nonsmall} below prove that the Lippmann-Schwinger equation \eqref{eq:nonLippscalar2} for nonlinear resonance admits a unique solution $(\h{\omega}, u)$, up to a phase factor $e^{i \theta}$ of $u$, 
bifurcating from $(\h{\omega}_*, 0)$, with the corresponding small-amplitude nonlinear resonant states being, to leading order, small multiples of $\varphi_*$.  

The proof essentially adapts the bifurcation theory for simple eigenvalues \cite{crandall1971bifurcation}, but it is complicated by the complex-valued function spaces and the $S^1$ equivariance of the solution $u$ to \eqref{eq:nonLippscalar2}, that is, for any solution $(\h{\ww},u)$ to \eqref{eq:nonLippscalar2}, there is an associated trivial solution curve $\{(\h{\omega}, e^{i \theta} u)\}_{\theta \in \R}$, as discussed in \cref{lem:invarphaseu}.  To eliminate this $S^1$ equivariance, we will employ Lyapunov-Schmidt reduction techniques, and the existence of nonlinear dielectric resonances is then guaranteed by the implicit function theorem.
 
\begin{theorem} \label{thm:smallampl}
Let $\h{\ww}_j := \sqrt{\lambda_j}^{-1}$, where $\lad_j$ is a simple eigenvalue of the Newtonian potential $\mathcal{K}_D$ with $L^2$ normalized eigenfunction $\vp_j$, \textit{i.e.}, $\norm{\vp_j}_{L^2(D)} = 1$. Then, for a neighborhood $U := \{(a,\eps) \in \mathbb{R}\times \R_+ \,;\ |a| + \eps  < \delta\}$ with sufficiently small $\d > 0$, a nonlinear scaled dielectric resonance $\h{\ww}(a,\eps): U \to \mathbb{C}$ exists for \eqref{eq:nonLippscalar2}, and 
the associated resonant state $u$ (\textit{i.e.}, the solution to \eqref{eq:nonLippscalar2}) is unique, up to the phase factor $e^{i \theta}$, $\theta \in [0, 2\pi)$, which can be represented as: for some $h: U \to H^2(D)$,
\begin{align} \label{eq:nonresostate}  
u(a,\eps) = a \varphi_j + h(a,\eps)\,, \quad \langle \varphi_j, h(a,\eps) \rangle_D = 0\,.
\end{align}
Here $\h{\ww}(\dd)$ and $h(\dd)$ are $C^1$ functions defined on $U$, 
satisfying: as $(a,\eps) \to 0$ in $\R^2$, 
\begin{align*}  
\hat{\omega}(a,\eps) \to \hat{\omega}_j\,, \q h(a,\eps) \to 0\,.
\end{align*}  
\end{theorem}

\begin{proof}
By the spectral decomposition of $\mathcal{K}_D$ in \eqref{eq:speckdd}, and recalling that $\lad_j$ is an eigenvalue of multiplicity one, let  $P_j u = \langle \varphi_j, u \rangle_D \varphi_j$ denote the eigenprojection associated with $\lambda_j$, and $P_r u = \sum_{\ell \neq j}^\infty \langle \varphi_\ell, u \rangle_D \varphi_\ell$ represent the projection corresponding to the remaining eigenvalues of $\mathcal{K}_D$. Then we parametrize $u \in H^2(D) \subset L^2(D)$ by
\begin{align} \label{parau}
    u(a,h) = a \vp_j + h\,,\q   a := \l \vp_j, u \r_{D}\,,
\end{align}
where 
\begin{align*}
  h := P_r u = u - a \vp_j \in H_\sim^2(D): = H^2(D) \cap P_r L^2(D)\,.  
\end{align*}
We now define, for $u \in H^2(D)$, 
\begin{align} \label{def:fofr}
    f_j(a,h) := \l \vp_j, |u|^2 u \r_D \in \C\,,\q f_r(a,h): = P_r (|u|^2 u) \in L^2(D)\,,
\end{align}
such that $$|u|^2 u  = f_j(a,h) \vp_j + f_r(a,h)\,.$$  
It follows that $N(u) = u + |u|^2 u$ is also parametrized by $(a,h) \in \C \times H_\sim^2(D)$ 
\begin{align}\label{eq:paranonlinear}
   N(u)(a,h) = (a + f_j(a,h))\vp_j + (h + f_r(a,h))\,.
\end{align}

To solve \eqref{eq:nonLippscalar2}, we apply the projections $P_j$ and $P_r$ and find 
\begin{align}
   & a \vp_j = \h{\ww}^2 \kk_D \left[ a \vp_j + f_j(a, h) \vp_j \right] + \h{\ww}^2 P_j \rr^{\eps \h{\ww}}[N(u)]\,,  \label{eq:proj0} \\
   & h = \h{\ww}^2 \kk_D \left[ h + f_r(a, h) \right] + \h{\ww}^2 P_r \rr^{\eps \h{\ww}}[N(u)]\,, \label{eq:proj1}
\end{align} 
where $$\mc{R}^{\eps \h{\ww}} := \kk_D^{\eps \h{\ww}} - \kk_D\,,$$ and we have used 
\begin{align*}
  P_j \kk_D = \kk_D  P_j\,, \q P_r \kk_D = \kk_D P_r\,.
\end{align*}
It is clear that \eqref{eq:proj0} is equivalent to 
\begin{align}\label{eq:proj00}
    a = \h{\ww}_j^{-2} \h{\ww}^2  \left( a + f_j(a, h) \right) + \h{\ww}^2 \left\l \vp_j, \rr^{\eps \h{\ww}}[N(u)] \right\r_D\,,
\end{align}
by taking the inner product with $\vp_j$ and using $\h{\ww}_j^{2}  \kk_D[\vp_j] = \vp_j$. For later use, we also compute the Fr\'echet derivatives of $|u(a,h)|^2 u(a,h)$ in $a$ and $h$, by using \eqref{eq:derivfunc}, 
\begin{equation}\label{eq:deva}
   \p_a (|u|^2 u)[b]  = 2 |u|^2 b \vp_j + u^2 \wb{b \vp_j} \in \mc{L}_{\R}(\C, L^2(D))\,,
\end{equation}
and 
\begin{equation}\label{eq:devh}
\p_h (|u|^2 u)[v]  = 2 |u|^2 v + u^2 \wb{v} \in \mc{L}_{\R}(H_\sim^2(D), L^2(D))\,.
\end{equation}

We shall solve the system \eqref{eq:proj1} and \eqref{eq:proj00} by first expressing $h \in H_\sim^2(D)$ locally in terms of  $(\h{\omega}, a, \eps) \in \mathbb{C}^2 \times \R$ near $(\h{\ww}_j, 0, 0) $ via \eqref{eq:proj1}. Substituting $ h(\h{\omega}, a, \eps)$ into \eqref{eq:proj00} then yields an equation involving only $(\h{\omega}, a, \eps)$, which allows us to further solve $\h{\ww}(a,\eps)$. Note that for $\h{\ww}$ near $\h{\ww}_j$,
the operator $(1 - \h{\ww}^2\kk_D)^{-1}$ is well-defined on $P_r L^2(D)$, which is an invariant subspace of $\kk_D$.
Thus, the solution to \eqref{eq:proj1} is given by the zero set of the following nonlinear map:
\begin{align} \label{eq:nonlinff}
    \mc{F}(\h{\ww}, a, h, \eps) := h - \left(\h{\ww}^{-2} - \kk_D \right)^{-1}  \kk_D [f_r(a,h)] - \left(\h{\ww}^{-2} - \kk_D\right)^{-1} P_r \rr^{\eps \h{\ww}}[N(u)(a,h)]\,,
\end{align}
which is real analytic from $U_{\h{\ww}_j} \times \C \times H_\sim^2(D) \times \R$ to $H_\sim^2(D)$. Here, $U_{\h{\ww}_j} \subset \C$ is a neighborhood of $\h{\ww}_j$. It is easy to compute 
\begin{align*}
    \mc{F}\left(\h{\ww}_j, 0, 0, 0\right) = 0\,, \q \p_h \mc{F}\left(\h{\ww}_j, 0, 0, 0\right) = I \,,
\end{align*}
by using 
\begin{align*}
  \p_h f_r(0,0) = 0\,,\q \p_h N(u)(0,0) = I\,,\q \rr^{\eps \ww}|_{\eps = 0} = 0\,,
\end{align*}
from \eqref{def:fofr} and \eqref{eq:devh}. Since the linear map $\p_h \mc{F}\left(\h{\ww}_j, 0, 0, 0\right)$ is invertible, by the implicit function theorem, there exists a unique $C^1$ function: for a small neighborhood $U_{\left(\h{\ww}_j, 0, 0\right)}$ of $(\h{\ww}_j, 0, 0)$, 
\begin{align} \label{eq:functionh}
  h(\h{\ww}, a, \eps):\   U_{\left(\h{\ww}_j, 0, 0\right)}  \to H_{\sim}^2(D)\,,
\end{align}
such that $\mc{F}\left(\h{\ww},a,h\left(\h{\ww},a, \eps \right), \eps\right) = 0$. 
That is, \eqref{eq:proj1} admits a unique solution $h \in H_{\sim}^2(D)$ for any $(\h{\ww}, a, \eps) \in \C^2 \times \R$ close to $(\h{\ww}_j, 0, 0)$. It is also helpful to note that $h = 0$ satisfies \eqref{eq:proj1} for any given $\h{\ww},\eps$, and $a = 0$, and thus, by uniqueness, 
\begin{align} \label{eq:hctnzero}
    h(\h{\ww},0,\eps) = 0\,.
\end{align}

As observed in \cref{lem:invarphaseu}, for any $(a,h)$ solving equations \eqref{eq:proj1} and \eqref{eq:proj00}, $\left(e^{i \theta} a, e^{i \theta} h \right)$ with $\theta \in \R$ is also a solution. By the uniqueness of $h$ for given $(\h{\ww},a, \eps) \in U_{(\h{\ww}_j, 0, 0)}$, there necessarily holds
\begin{align*}
  h(\h{\ww},a, \eps) = \frac{a}{|a|} h(\h{\ww}, |a|, \eps)\,.
\end{align*}
Therefore, we only need to consider $h(\h{\ww},a, \eps)$ with real $a \in \R$ and find, by plugging it into 
\eqref{eq:proj00} and dividing the equation by $a$,
\begin{align} \label{eq:proj000}
    1 = \h{\ww}_j^{-2} \h{\ww}^2  \left(1 + a^{-1} f_j(a, h(\h{\ww}, a, \eps)) \right) + a^{-1} \h{\ww}^2 \left\l \vp_j, \rr^{\eps \h{\ww}}[N(u)] \right\r_D\,. 
\end{align}
Here, both $a^{-1} f_j(a, h(\h{\ww}, a, \eps))$ and $a^{-1} N(u)$ are well-defined at $a = 0$, since 
\begin{align*}
 h(\h{\ww},0, \eps) = 0\,,\q f_j(0, h(\h{\ww}, 0, \eps)) = 0\,,\q N(u)(0,h(\h{\ww},0, \eps)) = 0\,,
\end{align*}
and $f_j(a, h(\h{\ww}, a, \eps))$ and $N(u)(a,h(\h{\ww}, a, \eps))$ are real differentiable in $a$. 
Specifically, we have
\begin{align} \label{eq:dervf0a}
        \lim_{a \to 0} a^{-1} f_j(a, h(\h{\ww}, a, \eps)) = \p_a f_j(0,0) + \p_h f_j(0,0) \p_a h(\h{\ww},0,\eps) = 0\,,
\end{align}
for $\h{\ww}$ near $\h{\ww}_j$, due to $\p_a f_j(0,0) = \p_h f_j(0,0) = 0$.

Now, we solve \eqref{eq:proj000}. We define the nonlinear map:
\begin{align} \label{eq:nonligg}
     \mc{G}(\h{\ww}, a, \eps) := \h{\ww}_j^{-2} \h{\ww}^2  \left(1 + a^{-1} f_j(a, h(\h{\ww}, a, \eps)) \right) + a^{-1} \h{\ww}^2 \left\l \vp_j, \rr^{\eps \h{\ww}}[N(u)] \right\r_D - 1\,,
\end{align}
which is $C^1$ from $\{\h{\ww} \in \C\,;\ |\h{\ww} - \h{\ww}_j| \le \d\} \times \{(a, \eps) \in \R^2\,;\ |a| + |\eps| \le \d\}$ to $\C$ for a small $\d > 0$. It is direct to compute 
\begin{align} \label{auxeq1}
    \mc{G}(\h{\ww}_j, 0, 0) = \h{\ww}_j^{-2} \h{\ww}_j^2 
    - 1 = 0\,,
\end{align}
where we have used \eqref{eq:dervf0a}.
We then compute
\begin{align*}
  \p_{\h{\ww}}\mc{G}(\h{\ww}_j, 0, 0) = 2 \h{\ww}_j^{-1}\,.
\end{align*}
Again, since $\p_{\h{\ww}}\mc{G}(\h{\ww}_j,0, 0) \neq 0$, the implicit function theorem
implies the existence of a $C^1$ function $\h{\ww}(a,\eps)$ for small enough $(a,\eps) \in \R^2$.
Substituting $\h{\ww}(a, \eps)$ into $h(\h{\ww},a,\eps)$ readily gives $h(a,\eps)$, which satisfies $h(a,\eps) \to 0$ in $H^2(D)$ as $(a, \eps) \to 0$ in $\R^2$. The proof is complete. 
\end{proof}

From \cref{thm:smallampl}, it easily follows that nonlinear dielectric resonances exist near the linear ones for a given high contrast $\tau$, accompanied by small-amplitude resonant states (see \cref{coro:nonsmall} below). Specifically, by \cref{prop:limit_scalar3d}, let $\omega_*$ be a linear subwavelength dielectric resonance with the resonant state $\varphi_*$, corresponding to $(\lad_j,\vp_j)$, which satisfies
\begin{align} \label{approx:lineareso}  
   \omega_* = \sqrt{\tau}^{-1} \hat{\omega}_j + \mathcal{O}(\tau^{-1})\,, \quad \hat{\omega}_j = \sqrt{\lambda_j}^{-1}\,,  
\end{align}  
and  
\begin{align} \label{eq:linear2}  
  \varphi_* = \varphi_j + \mathcal{O}(\sqrt{\tau}^{-1})\,, \quad \| \varphi_j \|_{L^2(D)}^2 = 1\,, \quad \varphi_* - \varphi_j \perp \varphi_j\,.  
\end{align}  
Here, $\lad_j$ is the simple eigenvalue as in \cref{thm:smallampl}. 

\begin{corollary} \label{coro:nonsmall}
Under the assumptions of \cref{thm:smallampl}, let $\hat{\omega}_*(\eps) = \eps^{-1} \omega_*(\eps)$ be the scaled linear dielectric resonance, and $\vp_*(\eps)$ solve  
\begin{equation} \label{eq:linear1}  
    \varphi_* = \hat{\omega}_*^2 \mathcal{K}_D^{\eps \hat{\omega}_*}[\varphi_*]\,,
\end{equation}  
with the estimates \eqref{approx:lineareso} and \eqref{eq:linear2}. 
It holds that, for $C^1$ functions $\h{\ww}(a,\eps)$ and $h(a,\eps)$ defined in \cref{thm:smallampl}, as $a \to 0$ with fixed small $\eps > 0$,
\begin{align} \label{auxeq:hw}
    h(a,\eps) \to h(0,\eps) = 0\,,\q \h{\ww}(a,\eps) \to \h{\ww}(0,\eps) = \h{\ww}_*(\eps)\,,
\end{align}
and for nonlinear resonant state $u(a,\eps)$ in \eqref{eq:nonresostate},
\begin{align} \label{auxeq:stau}
    a^{-1} u(a,\eps) \to \vp_*(\eps)\,.
\end{align}
\end{corollary}

\begin{proof}
We adopt the notation used in the proof of \cref{thm:smallampl}. 
For proving \eqref{auxeq:hw}, by continuity, it suffices to show $h(0,\eps) = 0$ and $\h{\ww}(0,\eps) = \h{\ww}_*$, where $h(0,\eps) = 0$ follows directly from \eqref{eq:hctnzero}.
Note that $\h{\ww}(0,\eps) = \h{\ww}_*$ is equivalent to 
\begin{align} \label{eq:zerogm}
    \mc{G}(\h{\ww}_*,0,\eps) = 0\,,
\end{align}
where the mapping $\mc{G}(\dd)$ is defined in \eqref{eq:nonligg}. To prove \eqref{eq:zerogm}, recalling $h(\h{\ww},a,\eps)$ introduced in \eqref{eq:functionh}, we compute 
\begin{equation} \label{eq:dervna}
     \lim_{a \to 0} a^{-1} N(u)(a,h(\h{\ww},a, \eps))  = \vp_j + \p_a h(\h{\ww},0, \eps)\,,
\end{equation}
and 
\begin{align*}
    \p_a h(\h{\ww},0, \eps) & = - \p_h \mc{F}\left(\h{\ww}, 0, 0, \eps\right)^{-1} \p_a \mc{F}\left(\h{\ww}, 0, 0, \eps\right) \\
      & = \left(1 - \h{\ww}^2 \kk_D - \h{\ww}^2 P_r \rr^{\eps \h{\ww}}\right)^{-1} \h{\ww}^2 P_r \rr^{\eps \h{\ww}}[\vp_j]\,.    
\end{align*}
At $\h{\ww} = \h{\ww}_*$, we further have 
\begin{align} \label{eq:dervna2}
    \p_a h(\h{\ww}_*, 0, \eps) =  \vp_* - \vp_j \,,
\end{align}
by observing from \eqref{eq:linear1} that 
\begin{equation*}
     P_r \left(1 -\h{\ww}_*^2 \kk_D^{\eps \h{\ww}_*}\right)[\vp_* - \vp_j] = - P_r (1 -\h{\ww}_*^2 \kk_D^{\eps \h{\ww}_*})[\vp_j] = \h{\ww}_*^2  P_r  \rr^{\eps \h{\ww}_*} [\vp_j]\,,
\end{equation*}
and from $\vp_* - \vp_j \perp \vp_j$ in \eqref{eq:linear2} that
\begin{align*}
   P_r \left(1 -\h{\ww}_*^2 \kk_D^{\eps \h{\ww}_*}\right)[\vp_* - \vp_j] = \left(1 - \h{\ww}_*^2 \kk_D - \h{\ww}_*^2 P_r \rr^{\eps \h{\ww}_*}\right)[\vp_* - \vp_j]\,,
\end{align*}
where the operator $1 - \h{\ww}_*^2 \kk_D - \h{\ww}_*^2 P_r \rr^{\eps \h{\ww}_*}$ is invertible for small $\eps$.
Substituting \eqref{eq:dervna2} into \eqref{eq:dervna}, we obtain  
\[
\lim_{a \to 0} a^{-1} N(u)(a, h(\hat{\omega}_*, a, \eps)) = \varphi_* \,.  
\]  
Then, using \eqref{eq:nonligg} with \eqref{eq:dervf0a}, we have  
\begin{align*}  
    \mathcal{G}(\hat{\omega}_*, 0, \eps) = \hat{\omega}_j^{-2} \hat{\omega}_*^2 + \hat{\omega}_*^2 \left\langle \varphi_j, \mathcal{R}^{\eps \hat{\omega}_*}[\varphi_*] \right\rangle_D - 1\,.
\end{align*}  
Taking the inner product of \eqref{eq:linear1} with $\varphi_j$ and $\vp_* - \vp_j \perp \vp_j$, it follows that 
$$
\mc{G}(\h{\ww}_*,0,\eps) = 0\,,
$$
from which we conclude that $\h{\ww}(a,\eps) \to \h{\ww}(0,\eps) = \h{\ww}_*$ as $a \to 0$. 

Finally, we consider the nonlinear Lippmann-Schwinger equation:
\begin{align*}
        u(a,\eps) =  \h{\omega}(a,\eps)^2 \mathcal{K}_D^{\eps \h{\omega}(a,\eps)} \left[N(u)(a,\eps)\right].  
\end{align*}
Dividing it by small $a \in \R$ and taking $a \to 0$, we find 
\begin{align*}
    \lim_{a \to 0} a^{-1} u(a,\eps) & =  \h{\ww}_*^2 \mathcal{K}_D^{\eps \h{\ww}_*} \left[\lim_{a \to 0} a^{-1} N(u)(a,\eps)\right] \\
    & = \h{\ww}_*^2 \mathcal{K}_D^{\eps \h{\ww}_*} \left[\lim_{a \to 0} a^{-1} u(a,\eps)\right],
\end{align*}
which gives \eqref{auxeq:stau} by $|u(a,\eps)|^2 u(a,\eps) = \Or(|a|^3)$ due to $h(a,\eps) = \Or(a)$ from \eqref{auxeq:hw}, and by
uniqueness of the normalized solution to \eqref{eq:linear1}. 
\end{proof}

\cref{thm:smallampl,coro:nonsmall} have established the leading order behaviors of the scaled nonlinear resonance $\hat{\omega}(a, \eps)$ and the corresponding resonant state $u(a, \eps)$ near $(a, \eps) = (0, 0)$. The higher order asymptotics can be readily derived using the Taylor expansions of $\hat{\omega}(a, \eps)$ and $u(a, \eps)$, noting that $ \hat{\omega}$ and $u$ are real analytic in $(a, \eps)$ by the implicit function theorem, since the nonlinear mappings $ \mathcal{F}$ and $\mathcal{G}$ defined in \eqref{eq:nonlinff} and \eqref{eq:nonligg} are real analytic. 

\begin{corollary}  \label{coro:nonsmall2}
  Under the assumptions of \cref{thm:smallampl,coro:nonsmall}, the following asymptotics hold for the scaled nonlinear subwavelength dielectric resonances $\hat{\omega}(a, \eps)$ and the associated small-amplitude resonant states $u(a, \eps)$: as $(a,\eps) \to 0$ in $\R^2$,  
\begin{align*}  
    \hat{\omega}(a, \eps) = \hat{\omega}_j - \frac{i \hat{\omega}_j^4}{8 \pi} \left(\int_D \varphi_j \, dx \right)^2 \eps + \mathcal{O}(a^2 + \eps^2)\,,  
\end{align*}  
and  
\begin{align*}  
    u(a, \eps) = a \varphi_j + \mathcal{O}(a^2 + \eps^2)\,.  
\end{align*}
\end{corollary}
    
\begin{proof}
    We follow the analysis in \cref{coro:nonsmall}, which proved the following:
    \begin{align*}
        \h{\ww}(a,\eps) \underset{\eqref{auxeq:hw}}{=}  & \h{\ww}_*(\eps) + \p_a \h{\ww}(0,0) a + \Or(a^2 + |a|\eps) \\  =\ \, & \h{\ww}_j + \p_\eps \h{\ww}_*(0) \eps + \p_a \h{\ww}(0,0) a + \Or(a^2 + \eps^2)\,,
    \end{align*}
    and 
    \begin{align*}
        u(a,\eps) \underset{\eqref{auxeq:stau}}{=} a \vp_*(\eps) + \Or(a^2) \underset{\eqref{eq:limit_resta3d}}{=} a \vp_0 + \Or(a^2 + \eps^2)\,.
    \end{align*}
  To complete the proof, we now compute $ \partial_\eps \hat{\omega}_*(0)$ and $\partial_a \hat{\omega}(0,0)$. Here, $\partial_\eps \hat{\omega}_*(0)$ follows directly from the asymptotics of the linear subwavelength dielectric resonances derived in \cites{ammari2019subwavelength,ammari2023mathematical}:  
\begin{align*}  
    \hat{\omega}_*(\eps) = \hat{\omega}_j - \frac{i \hat{\omega}_j^4}{8 \pi} \left(\int_D \varphi_j \, dx \right)^2 \eps + \mathcal{O}(\eps^2)\,,  
\end{align*}  
which implies  
\begin{align*}  
    \partial_\eps \hat{\omega}_*(0) = -\frac{i \hat{\omega}_j^4}{8 \pi} \left(\int_D \varphi_j \, dx \right)^2.  
\end{align*}  
To compute $\partial_a \hat{\omega}(0,0)$, we take the derivative of \eqref{eq:proj000} with respect to $a$ and evaluate it at the point $(a, \eps) = (0,0)$:  
\begin{align*}  
    0 = 2 \hat{\omega}_j^{-2} \hat{\omega}_j \partial_a \hat{\omega}(0,0)\,.  
\end{align*}  
This implies $\partial_a \hat{\omega}(0,0) = 0$, and completes the proof. 
\end{proof}

For completeness, we now translate the results of \cref{thm:smallampl,coro:nonsmall,coro:nonsmall2} for the scaled equation \eqref{eq:nonLippscalar2} back to the original one \eqref{eq:nonLippscalar0}. 

\begin{theorem} \label{thm:3dexist}
Let $\lad_j$ be a simple eigenvalue of the Newtonian potential $\mathcal{K}_D$ with normalized eigenfunction $\vp_j$, namely $\norm{\vp_j}_{L^2(D)} = 1$. Moreover, let $(\ww_*(\tau), \vp_*(\tau))$ be the associated linear dielectric resonance and resonant state given in \eqref{approx:lineareso} and \eqref{eq:linear2}. 
Then, there exists an open set $U := \{(a,\tau) \in  \R_+^2 \,;\ a + \tau^{-1} < \delta\}$ with sufficiently small $\d > 0$ such that
\begin{itemize}
\item  A nonlinear dielectric resonance $\ww(a, \tau): U \to \C$ exists for \eqref{eq:nonLippscalar0} with the corresponding resonant state parameterized as 
\begin{align*}
    u(a, \tau) = a \varphi_j + h(a, \tau)\,, \quad \langle \varphi_j, h(a, \tau) \rangle_D = 0\,,
\end{align*}
which is unique, up to the symmetry $u\to e^{i \theta} u$, $\theta \in [0, 2\pi)$, where $h(a,\tau): U \to H^2(D)$.  
\item Functions $\ww(a,\tau)$ and $u(a,\tau)$ are real analytic in $(a,\tau)$. Furthermore, they satisfy the following asymptotics: for a fixed high contrast $\tau$, as $a \to 0$, 
 \begin{align} \label{eq:nonresoandstate}
      \omega(a, \tau) = \ww_*(\tau) + \Or\left(\frac{a^2}{\sqrt{\tau}} + \frac{a}{\tau}\right),\q  u(a,\tau)  = a \vp_*(\tau) + \Or(a^2)\,.
 \end{align}
 In particular, as $a \to 0$ and $\tau \to 0$, 
\begin{align*}
    u(a,\tau) = a \vp_j + \mathcal{O}\left(a^2 + \frac{1}{\tau}\right),
\end{align*}
and
\begin{equation*}
        \omega(a, \tau) = \frac{1}{\sqrt{\lad_j \tau}} - 
 \frac{i}{8 \pi \lad_j^2 \tau} \left(\int_D \varphi_j \, dx \right)^2 + \mathcal{O}\left(\frac{a^2}{\sqrt{\tau}} + \frac{1}{\tau^{\frac{3}{2}}}\right).
\end{equation*} 
\end{itemize}
\end{theorem}

\begin{remark} \label{rem:invernorcont}
   In the proof of \cref{thm:3dexist}, we have seen $h(a,\tau) = \Or(a)$. Combining it with the equation $\mathcal{F}(\h{\ww}, a, h, \eps) = 0$ for $\mathcal{F}$ defined in \eqref{eq:nonlinff}, we obtain the refined estimate:
\begin{align*}
    h(a,\tau) = \Or(|a|^3 + |a| \tau^{-1/2})\,.
\end{align*}
The normalization condition \eqref{eq:normalizecond} then yields
\begin{equation*}
    \|u\|_{L^2(D)}^2 = \Or\big((1 + \tau^{-1}) |a|^2\big) = \mathcal{N}.
\end{equation*}
Consequently, for sufficiently small $|a| > 0$ and large $\tau$, the mapping $|a| \mapsto \mathcal{N}$ is invertible. This allows us to reparameterize the nonlinear resonance and the corresponding state in \eqref{eq:nonresoandstate} as $\ww(\mathcal{N},\tau)$ and $u(\mathcal{N},\tau)$, respectively.
Therefore, we have shown that for small amplitude $\mathcal{N}$ and high contrast $\tau$, the nonlinear dielectric resonances exist in the subwavelength regime (see \cref{def:scareso}).
\end{remark}

\subsection{Two-dimensional case} \label{subsec:2dsolution}
In this section, we establish the two-dimensional analogue of \cref{thm:3dexist}, proving that nonlinear subwavelength resonances exist near the linear resonances given in \cref{prop:limit_scalar2d} with small-amplitude resonant states. 

\begin{theorem} \label{thm:2dexist}
In the high contrast regime $\tau \gg 1$, for $d = 2$, let $\ww_*(\tau)$ be a linear subwavelength resonance satisfying either 
\begin{align} \label{case1} 
 \ww_*(\tau) = \sqrt{\frac{4 \pi}{ |D| \tau \ln \tau }} + \Or\left( \frac{\ln(\ln \tau
  )}{\sqrt{\tau} \ln \tau} \right),
\end{align}
with normalized resonant state $\vp_*(\tau) = 1 + \Or\left(\frac{1}{\ln \tau}\right)$, or 
\begin{align} \label{case2} 
   \ww_*(\tau) = \frac{1}{\sqrt{\tau \mu_j}} + \Or\left(\frac{1}{\sqrt{\tau} \ln \tau}\right) \q \text{for some}\ j \ge 0\,, 
\end{align}
with normalized resonant state $\vp_*(\tau) = v_j + \Or\left(\frac{1}{\ln \tau}\right)$, 
where $(\mu_j, v_j)$ satisfies $\w{\mc{K}}_D[v_j] = \mu_j v_j$ with $\mu_j$ being simple and $\norm{v_j}_{L^2(D)} = 1$. 

Then,  there exists an open set $U := \{(a,\tau) \in \R_+^2 \,;\ a + \tau^{-1} < \delta\}$ with sufficiently small $\d > 0$ such that a nonlinear dielectric resonance $\ww(a, \tau): U \to \C$ exists for \eqref{eq:nonLippscalar0} with unique
resonant state $u(a, \tau)$, up to $u \to e^{i \theta} u$, $\theta \in [0, 2\pi)$. Moreover, the following asymptotic holds: $$u(a,\tau)  = a \vp_*(\tau) + \Or(a^2)\,,$$ and 
\begin{align*}
    \ww(a,\tau) = \begin{dcases}
    \ww_*(\tau) + \Or\left(\frac{a^2}{\sqrt{\tau \ln \tau}} + \frac{a}{\sqrt{\tau} \ln \tau}\right)  & \text{for case \eqref{case1}}\,,\\
    \ww_*(\tau) + \Or\left(\frac{a^2}{\sqrt{\tau}} + \frac{a}{\sqrt{\tau \ln \tau}}\right)  & \text{for case \eqref{case2}}\,.
    \end{dcases}            
\end{align*}
\end{theorem}

The remainder of this section is devoted to the proof of \cref{thm:2dexist}. As in the three-dimensional case, it suffices to analyze the rescaled equation \eqref{eq:nonLippscalar2} with the frequency scaling \eqref{eq:scalefreq}: $\h{\ww} = \eps^{-1} \ww$. We first recall from \eqref{eq:leading2d} and \eqref{eq:npt2d} the decomposition
\begin{align} \label{decomkdw}
    \kk^{\epsilon\h{\ww}}_D  = \h{\kk}^{ \eps \h{\omega}}_D + \rr^{\eps \h{\ww}}\,,
\end{align}
with 
\begin{align} \label{decomkdw2}
    \h{\kk}^{\eps \h{\omega}}_D[\vp] = \kk_D[\vp] - \eta_{\eps \h{\ww}} \int_D \vp(y) \ud y\,,\q \rr^{\eps \h{\ww}} = \mc{O}\left( (\epsilon \h{\ww})^2 \ln \left(\epsilon \h{\ww}\right)\right).
\end{align}
Under the assumption of \cref{thm:2dexist}, let the scaled resonant pair $(\h{\ww}_*(\eps),\vp_*(\eps))$ solve 
\begin{align} \label{screso2d}
      \varphi_* = \hat{\omega}_*^2 \mathcal{K}_D^{\eps \hat{\omega}_*}[\varphi_*]\,,
\end{align}
with either 
\begin{align} \label{screso2d1}
    \h{\ww}_*(\eps) = \sqrt{\frac{2 \pi}{ - \ln \left(\eps\right) |D|}} + \Or\left( \frac{\ln(-\ln \left(\eps \right))}{\ln \left(\eps\right)} \right)\,,\q \vp_*(\eps) = 1 + \Or\left(\frac{1}{\ln \eps}\right)\,,
\end{align}
or 
\begin{align} \label{screso2d2}
    \h{\ww}_*(\eps) = \frac{1}{\sqrt{\mu_j}} + \mc{O}\left(\frac{1}{\ln \eps}\right)\,, \q \vp_*(\eps) = v_j + \Or\left(\frac{1}{\ln \eps}\right)\,.
\end{align}

\begin{proof}[Proof of \cref{thm:2dexist}]
Since the arguments are parallel to those of the three-dimensional case (\cref{thm:3dexist}), we outline only the main steps, highlighting the main differences arising from the logarithmic singularity in \eqref{screso2d}. Next, we consider the two cases in \eqref{screso2d1} and \eqref{screso2d2} separately.

\medskip

\noindent \underline{\emph{Case I: $\h{\ww}_* = \Or(\sqrt{-\ln \eps}^{-1})$}.} We introduce projections: for $u \in H^2(D) \subset L^2(D)$, 
\begin{align} \label{defproj2dd}
    P_0 u = \l 1_D, u \r_D: L^2(D) \to \C \,,\q P_r = I - P_0: L^2(D) \to L_0^2(D)\,.
\end{align}
Then, the solution $u$ for \eqref{eq:nonLippscalar2} can be parameterized as: 
\begin{align*}
    u(a,h) = a + h\,,\q \text{with }\, a = P_0 u\,, \ h = P_r u\,.
\end{align*}
Similarly to \eqref{def:fofr}, define 
\begin{align*}
    f_0(a,h) := P_0(|u|^2 u) = \l 1_D, |u|^2 u  \r_D\,,\q f_r(a,h) := P_r(|u|^2 u)\,.
\end{align*}
Applying projections $P_0$ and $P_r$ to \eqref{eq:nonLippscalar2}, along with \eqref{decomkdw}, implies 
\begin{align}
   & a = \h{\ww}^2 \l 1_D, \h{\kk}_D^{\eps \h{\ww}} \left[ N(u) \right] \r_D + \h{\ww}^2 \l 1_D, \rr^{\eps \h{\ww}}[N(u)] \r_D\,,  \label{eq:proj02} \\
   & h = \h{\ww}^2 P_r \h{\kk}_D^{\eps \h{\ww}} \left[ N(u) \right] + \h{\ww}^2 P_r \rr^{\eps \h{\ww}}[N(u)]\,. \label{eq:proj12}
\end{align} 
With the help of $\wph = P_r \kk_D[1]$ and $\wkk = P_r \kk_D P_r$ in \eqref{auxfunc2d} and \eqref{eq:limopscal}, we have 
\begin{align*}
    P_r \h{\kk}_D^{\eps \h{\ww}} \left[ N(u) \right] \underset{\eqref{decomkdw2}}{=} P_r \kk_D \left[ N(u) \right] = (a + f_0(a,h)) \w{\phi} + \w{\kk}_D[h + f_r(a,h)]\,.
\end{align*}
We also introduce new scaled variables:
\begin{align} \label{eq:newscaling}
    \br{\ww} = \sqrt{-\ln \eps} \, \h{\ww}\,,\q s = \frac{1}{\sqrt{-\ln \eps}}\,,\q \textit{i.e.},\q \eps = e^{-\frac{1}{s^2}}.
\end{align}
Note that the function $\eps(s) = e^{-\frac{1}{s^2}}$ with $\eps(0) = 0$ is infinitely differentiable at  $s = 0$, and all its derivatives at $s = 0$ are $0$. Moreover, the scaled subwavelength resonance $\h{\ww}_*(\eps)$ in \eqref{screso2d1} admits 
\begin{align*}
    \h{\ww}_*(\eps) = s \br{\ww}_0  + \Or( s^2 \ln s)\,,\q \br{\ww}_0 := \sqrt{\frac{2 \pi}{|D|}}\,.
\end{align*}
It follows that \eqref{eq:proj12} can be reformulated as 
\begin{align*}
    \mc{F}(\br{\ww},a,h, s) = 0\,,
\end{align*}
where 
\begin{multline*}
    \mc{F}(\br{\ww},a,h, s) := h - s^2(a + f_0(a,h)) (\br{\ww}^{-2} - s^2 \wkk)^{-1}[\w{\phi}] \\- s^2(\br{\ww}^{-2} - s^2 \wkk)^{-1} \w{\kk}_D[f_r(a,h)] - s^2  (\br{\ww}^{-2} - s^2 \wkk)^{-1} P_r \rr^{\eps s \br{\ww}}[N(u)]\,.
\end{multline*}
Then it is easy to compute 
\begin{align*}
    \mc{F}(\br{\ww}_0, 0, 0, 0) = 0\,, \q \p_{h}\mc{F}(\br{\ww}_0, 0, 0, 0) = I\,.
\end{align*}
The implicit function theorem readily gives a unique $C^1$ function 
$h(\br{\ww},a,s)$ on a small neighborhood of $(\br{\ww}_0,0,0)$ satisfying 
\begin{align*}
  h(\br{\ww},a,s) = \frac{a}{|a|} h(\br{\ww}, |a|, s)\,.
\end{align*}
We can therefore assume that $a \in \R$ and plug $h(\dd)$ into \eqref{eq:proj02} to obtain the following equation: 
\begin{align*}
    \mc{G}(\br{\ww},a, s) = 0\,,
\end{align*}
where 
\begin{align} \label{deffuncg}
    \mc{G}(\br{\ww},a, s) := s^2 \br{\ww}^2 a^{-1} \l 1_D, \h{\kk}_D^{\eps s \br{\ww}} \left[ N(u) \right] \r_D + s^2 \br{\ww}^2 a^{-1} \l 1_D, \rr^{\eps s \br{\ww}}[N(u)] \r_D - 1\,.
\end{align}
Here, $\mc{G}(\br{\ww},a, s)$ is well-defined at $a = 0$ since $\lim_{a \to 0} a^{-1} N(u)$ exists, and we have 
\begin{align*}
    \l 1_D, \h{\kk}_D^{\eps s \br{\ww}} \left[ N(u) \right] \r_D = \l 1_D, \kk_D \left[ N(u) \right] \r_D - \eta_{\eps s \br{\ww}} (a + f_0(a,h)) |D|\,,
\end{align*}
with by \eqref{eq:constgaome},
\begin{align*}
    \eta_{\eps s \br{\ww}} = \frac{1}{2\pi} \big(\ln (\eps s \br{\ww}) + \h{\gamma} \big)\,.
\end{align*}
We now compute 
\begin{align*}
    \mc{G}(\br{\ww}_0, 0, 0) =  - \br{\ww}^2_0 \lim_{s \to 0} s^2\eta_{\eps s \br{\ww}} |D| - 1 = 0\,,
\end{align*}
and 
\begin{align*}
    \p_{\br{\ww}}\mc{G}(\br{\ww}_0, 0, 0) = 2 \br{\ww}_0 \frac{|D|}{2 \pi} \neq 0\,.
\end{align*}
Again, the implicit function theorem allows us to solve a unique $C^1$ function $\br{\ww}(a,s) \in \C$ for a sufficiently small $(a,s) \in \R_+^2$, and then define $h(a,s) : = h(\br{\ww}(a,s),a,s)$.

Next, we derive the asymptotics of $\br{\ww}(a,s)$ and $h(a,s)$ as $(a,s) \to (0,0)$.  We first note $h(0,s) = 0$ since $(a,h) = (0,0)$ satisfies \eqref{eq:proj12}, and proceed to prove $\br{\ww}(0,s) = s^{-1} \h{\ww}_*(\eps) =: \br{\ww}_*(s)$, where $\h{\ww}_*(\eps)$ is the scaled linear resonance \eqref{screso2d1}. For this, it suffices to prove 
\begin{align} \label{auxeqq2d}
    \mc{G}(\br{\ww}_*(s), 0, s) = 0\,.
\end{align}
We compute, with $h(\br{\ww},a,s)$ solved from \eqref{eq:proj12},
\begin{align*}
    \lim_{a \to 0} a^{-1} N(u) & = 1 + \p_a h(\br{\ww},0,s) \\
     & = 1 - \p_h \mc{F}\left(\br{\ww}, 0, 0, s\right)^{-1} \p_a \mc{F}\left(\br{\ww}, 0, 0, s\right) \\
     & \underset{\h{\ww} = s \br{\ww}}{=} 1 + (1 - \h{\ww}^2 P_r \kk_D^{\eps \h{\ww}})^{-1} \left(\h{\ww}^2 P_r \kk_D^{\eps \h{\ww}}[1]\right).
\end{align*}
Moreover, from \eqref{screso2d} we note that 
\begin{align*}
     P_r (1 - \h{\ww}^2\kk_D^{\eps \h{\ww}})[\vp_* - 1] = - P_r (1 - \h{\ww}^2\kk_D^{\eps \h{\ww}})[1] = \h{\ww}^2 P_r  \kk_D^{\eps \h{\ww}}[1]\,.
\end{align*}
It follows that $\lim_{a \to 0} a^{-1} N(u) = \vp_*$, and then \eqref{auxeqq2d} holds by taking the inner product of \eqref{screso2d} with $1_D = 1/|D|$. Similarly to the proof of \cref{coro:nonsmall}, we can show that $\lim_{a \to 0} a^{-1} u(a,s) = \vp_*$. Then, a direct expansion gives 
\begin{align*}
    \br{\ww}(a,s) = \br{\ww}_*(s) + \p_a \br{\ww}(0,0)a + \Or(a^2 + |a| s)\,,\q u(a,s) = a \vp_* + \Or(a^2)\,,
\end{align*}
with $\p_a \br{\ww}(0,0) = 0$ obtained by taking the derivative of $\mc{G}(\dd)$ in \eqref{deffuncg} with respect to $a$ at the point $(\br{\ww}_0, 0, 0)$. We obtain the desired estimates by changing the variables $(s,\br{\ww})$ back to $(\tau,\ww)$.  

\medskip

\noindent \underline{\emph{Case II: $\h{\ww}_* = \Or(1)$}.} For simplicity, we sketch only the main steps. We define projections $P_j u := \l v_j , u\r_D v_j$ and $\w{P}_r := I - P_0 - P_j$, with $P_0$ given in \eqref{defproj2dd}. We then write 
\begin{align*}
    u = a + b v_j + h\,,\q a = P_0 u\,,\  b = \l v_j, u\r_D\,,\ h = \w{P}_r u\,,
\end{align*}
and 
\begin{align*}
    f_0(a, b, h) = \l 1_D, |u|^2 u  \r_D\,, \q f_1(a, b, h) = \l v_j, |u|^2 u \r_D\,,\q f_r(a,b, h) := \w{P}_r(|u|^2 u)\,.
\end{align*}
Similarly, equation \eqref{eq:nonLippscalar2} is equivalent to the following system: 
\begin{align}
   & a = \h{\ww}^2 \l 1_D, \h{\kk}_D^{\eps \h{\ww}} \left[ N(u) \right] \r_D + \h{\ww}^2 \l 1_D, \rr^{\eps \h{\ww}}[N(u)] \r_D\,,  \label{eq:proj03} \\ 
   & b = \h{\ww}^2 \l v_j, \kk_D \left[ N(u) \right] \r_D + \h{\ww}^2 \l v_j, \rr^{\eps \h{\ww}}[N(u)] \r_D\,,  \label{eq:proj04} \\
   & h = \h{\ww}^2 \w{P}_r \kk_D \left[ N(u) \right] + \h{\ww}^2 \w{P}_r \rr^{\eps \h{\ww}}[N(u)]\,. \label{eq:proj15}
\end{align} 

We first solve $h$ from \eqref{eq:proj15} by considering $\mc{F}(\h{\ww},a,b,h,s) = 0$, where 
\begin{align*}
    \mc{F}(\h{\ww},a,b,h,s) := h - (\h{\ww}^{-2} - \w{\kk}_D)^{-1} \w{P}_r \kk_D \left[a + |u|^2 u\right] - (\h{\ww}^{-2} - \w{\kk}_D)^{-1} \w{P}_r \rr^{\eps \h{\ww}}[N(u)]\,,
\end{align*}
with $\eps = e^{-\frac{1}{s^2}}$ from \eqref{eq:newscaling} to remove the singularity. 
Let $\h{\ww}_j = 1/\sqrt{\mu_j}$. It is easy to see that $\mc{F}(\h{\ww}_j,0,0,0,0) = 0$ and $\p_h\mc{F}(\h{\ww}_j,0,0,0,0) = I$. The implicit function theorem gives the unique $h(\h{\ww},a,b,s)$ near $(\h{\ww}_j,0,0,0)$ with $S^1$ equivariance. 
The plugging of $h(\dd)$ into \eqref{eq:proj03} implies that $\mc{L}(\h{\ww},a,b,s) = 0$, where 
\begin{align*}
    \mc{L}(\h{\ww},a,b,s) := a - s^2\frac{\h{\ww}^2 \big(\l 1_D, \kk_D \left[ N(u) \right] \r_D - \eta_{\eps \h{\ww}} f_0(a,b, h) |D| \big) + \h{\ww}^2 \l 1_D, \rr^{\eps \h{\ww}}[N(u)] \r_D}{s^2 + \h{\ww}^2 \frac{1}{2\pi} \big(s^2\ln \h{\omega} - 1 + s^2\h{\gamma} \big) |D|}\,.
\end{align*}
Noting that $\mc{L}(\h{\ww}_j, 0, 0, 0) = 0$ and $\p_a\mc{L}(\h{\ww}_j, 0, 0, 0) = 1$, and applying the implicit function theorem, we find the function $a(\h{\ww},b,s)$. Finally, with $h(\dd)$ and $a(\dd)$, we solve \eqref{eq:proj04} via $\mc{G}(\h{\ww},b,s) = 0$, where 
\begin{multline*}
   \mc{G}(\h{\ww},b,s) := (1 -  \h{\ww}^2 \h{\ww}_j^{-2}) - b^{-1} a \h{\ww}^2 \l v_j, \w{\phi} \r_D \\ -  b^{-1} \h{\ww}^2 \l v_j, \kk_D \left[|u|^2 u\right] \r_D   -  b^{-1} \h{\ww}^2 \l v_j, \rr^{\eps \h{\ww}}[N(u)] \r_D\,.
\end{multline*}
We compute 
\begin{align*}
    \mc{G}(\h{\ww}_j,0,0) = - \h{\ww}^2 \l v_j, \w{\phi} \r_D \p_b a(\h{\ww}_j,0,0) = 0\,, \q \p_{\h{\ww}} \mc{G}(\h{\ww}_j,0,0) = - 2 \h{\ww}_j^{-1}\,,
\end{align*}
by using $  \p_b\mc{L} + \p_a \mc{L} \p_b a|_{(\h{\ww}_j,0,0,0)} = \p_b a(\h{\ww}_j,0,0) = 0$. Again,  the implicit function theorem readily gives $\h{\ww}(b,s)$ near $(b,s) = (0,0)$ with $\h{\ww}(0,0) = \h{\ww}_j$. Similarly, we can show that for fixed small $s$, $u(0,s) = a + b v_j + r = 0$ and $\h{\ww}(0,s) = \h{\ww}_*(s)$ with $\h{\ww}_*(s)$ being the linear subwavelength resonance, as well as $b^{-1} u (b,s) \to \vp_*$ as $b \to 0$, where $(\h{\ww}_*, \vp_*)$ is given in \eqref{screso2d2}. Moreover, we have 
\begin{align*}
    \h{\ww}(b,s) = \h{\ww}_*(s) + \Or(b^2 + |b|s)\,.
\end{align*}
The proof is complete by changing the variables to $\ww$ and $\tau$.  
\end{proof}

\section{Symmetry-breaking bifurcation of a resonator dimer} \label{sec:existasympdimer}

In \cref{sec:existasymp}, we showed that for any configuration of dielectric particles $D$ and a sufficiently small normalization constant $\mathcal{N}$, there exist nonlinear subwavelength resonances with unique corresponding resonant states bifurcating from the zero solution at linear resonances. In the case of $D$ being symmetric, thanks to the uniqueness, these small-amplitude nonlinear resonant states shall also present certain symmetries. Furthermore, motivated by \cites{aschbacher2002symmetry,kirr2008symmetry}, one may expect that as $\mathcal{N}$ increases, additional symmetry-breaking bifurcations may emerge beyond these primary small-amplitude branches originating from the zero solution. 

In this section, we examine a dimer of resonators with reflection symmetry. As previously noted, the small-amplitude resonant states constructed in \cref{sec:existasymp} must be either symmetric or antisymmetric; see \cref{lem:symresostate,coro:smallsym}. Our primary focus concerns the existence of additional bifurcations as $\mathcal{N}$ increases. For the three-dimensional case, we demonstrate in \cref{thm:2ndbirfuc} that under mild assumptions verifiable in the dilute regime, a symmetry-breaking bifurcation occurs along the principal symmetric solution branch. In contrast, in the two-dimensional setting, we find in \cref{thm:2d2ndbif} that no such secondary bifurcation exists, owing to the distinct scaling behavior of principal subwavelength resonances (recall \cref{prop:limit_scalar2d,thm:2dexist}).

Let $D = D_1 \cup D_2$ be a dimer of dielectric particles in $\mathbb{R}^d$ ($d = 2,3$) that are symmetric with respect to a hyperplane. Without loss of generality, we consider $d = 3$, and take the symmetry hyperplane to be $\{x_1 = 0\}$ and define the corresponding reflection operator for a function $\varphi$ as  
\begin{equation} \label{def:reffunc} 
  \mathcal{R}[\varphi](x_1, x_2, x_3) := \varphi(-x_1, x_2, x_3)\,.  
\end{equation}  
Let $D_0$ be a single particle (\textit{i.e.}, a smooth, bounded, open set as specified in \cref{sec:basicsetting}) centered at the origin. We construct the dimer as follows:  
\begin{align}  \label{eq:dimercase}
  D_1 = D_0 + (-L, 0, 0)\,, \quad D_2 = \mathcal{R}(D_0) + (L, 0, 0)\,,  
\end{align}  
where $\mathcal{R}(D_0)$ is the reflected domain of $D_0$ with respect to $\{x_1 = 0\}$, defined via $\mathcal{R} \chi_{D_0} = \chi_{\mathcal{R}(D_0)}$. It is straightforward to see that the dimer $D = D_1 \cup D_2$ constructed in this manner is symmetric with respect to $\{x_1 = 0\}$, satisfying $\mathcal{R} \chi_D = \chi_D$.  For later use, we also introduce the translation operator $\mathcal{T}_a$ for $a \in \mathbb{R}$, defined as  
\begin{align}  \label{def:transfunc}
    \mathcal{T}_a[\varphi](x_1, x_2, x_3) := \varphi(x_1 + a, x_2, x_3)\,.  
\end{align}
Note that both $\mathcal{R}$ and  $\mathcal{T}_a$ in \eqref{def:reffunc} and \eqref{def:transfunc} are unitary operators.  Moreover, \cref{lem:invarphaseu} demonstrates the equivariance of the nonlinear dielectric resonance problem under the  $S^1$ group action. For the case of a symmetric dimer, the problem also admits $\mathbb{Z}_2$ equivariance due to reflection symmetry. The following lemma can be established straightforwardly, so the proof is omitted.

\begin{lemma} \label{lem:z2equivar}
  Let $\kk_D^\ww[\dd]$ and $N(\dd)$ be the operators defined in \eqref{eq:helmpotential} and \eqref{eq:nonfuncl2}. It holds that 
  \begin{align*}
        \mc{R}[N(u)] = N(\mc{R}[u])\,,\q   \mc{R}\circ \kk_D^\ww[u] = \kk_D^\ww \circ \rr [u]\,,
  \end{align*}
  for any $L^2$ function $u$. It follows that if      
  $(\omega, u) \in \mathbb{C} \times H^2(D)$ is a solution of the nonlinear resonance problem
  \eqref{eq:nonLippscalar0}, then $(\omega, \rr[u])$ is also a solution. 
\end{lemma}

We have seen in \cref{thm:3dexist,thm:2dexist} that for a general $D$,  there exists nonlinear resonant states, which are unique, up to $u \to e^{i \theta} u$, with small $\mc{N}$. We next show that for symmetric dimer $D$, these nonlinear resonant states must also exhibit symmetry by the simple lemma below. 


\begin{lemma} \label{lem:symresostate}
Given a symmetric dimer of resonators $D$, suppose that $\omega$ is a nonlinear dielectric resonance for \eqref{eq:nonLippscalar0} with a unique resonant state $u$, up to a phase factor $e^{i \theta}$, $\theta \in [0, 2\pi)$. Then, $u$ is either symmetric (even) or antisymmetric (odd).  
\end{lemma}

\begin{proof}
    By \cref{lem:z2equivar}, $\rr[u]$ is also a resonant state associated with the nonlinear resonance $\ww$. Then, by the uniqueness assumption of $u$, there exists $\theta \in [0, 2\pi)$ such that $$\rr[u] = e^{i \theta} u\,,$$ which further gives, by applying $\mathcal{R}$ again,
        \begin{align*}
          u = \rr^2[u] = e^{i \theta} \rr[u] = e^{i 2 \theta} u\,.
        \end{align*}
    It follows that $\theta = 0$ or $\pi$, and hence $\rr[u] = u$ or $\rr[u] = - u$. The proof is complete. 
\end{proof}

The following corollary, for the three-dimensional case, is a direct consequence of \cref{thm:3dexist} and \cref{lem:symresostate}, as well as \cref{coro:simplenewton}, which readily yields that the resonant state $u(a,\tau)$ corresponding to the principal nonlinear resonance $\ww(a,\tau) \sim 1/\sqrt{\lad_0 \tau}$ is even, \textit{i.e.}, $\rr[u(a,\tau)]  =u(a,\tau)$, and almost positive, up to some $e^{i \theta}$. A similar result can be established for the two-dimensional case based on \cref{thm:2dexist}, though we omit the details here for simplicity.

\begin{corollary} \label{coro:smallsym}
    Let $D$ be a symmetric dimer defined as above, and $\lad_j$ be a simple eigenvalue of the Newtonian potential $\mathcal{K}_D$ with normalized eigenfunction $\vp_j$. Then, $\vp_j$ is either symmetric or antisymmetric, \textit{i.e.}, $\rr[\vp_j] = \vp_j$ or $\rr[\vp_j] = - \vp_j$. In particular, the principal eigenvalue $\lad_0$ admits a positive and symmetric normalized eigenfunction $\vp_0$. Moreover, let $(\ww(a,\tau), u(a,\tau))$ be the nonlinear dielectric resonance and resonant state given in \cref{thm:3dexist}, associated with $(\lad_j,\vp_j)$. Then, $u(a,\tau)$ has the same symmetry as $\vp_j$.  
\end{corollary}

\subsection{Three-dimensional case} \label{subsec:3dcase}

For convenience, we denote by $\lad_+$ the principal eigenvalue of $\kk_D$ and by $ \vp_+$ the associated normalized positive symmetric eigenfunction.  
The main result of this section is the existence of a bifurcation on the principal nonlinear resonance curve  $\ww \sim \sqrt{\lad_+ \tau}$ at a non-zero normalization constant. To establish this, we rely on \cref{assp:second}, which can be verified in the dilute regime, that is, when the distance $L$ in \eqref{eq:dimercase} is sufficiently large (see \cref{prop:veriassp}).

\begin{assumption} \label{assp:second}
Let $D = D_1 \cup D_2$ be a symmetric dimer.  The second largest eigenvalue $\lad_-$ of $\kk_D$ is simple with the antisymmetric normalized eigenfunction $\vp_-$. 
\end{assumption}

\begin{proposition} \label{prop:veriassp}
Let $D = D_1 \cup D_2$ be a symmetric dimer of dielectric particles 
with $D_i$, $i = 1,2$, given in \eqref{eq:dimercase}, and denote by $\lad_0$ the largest eigenvalue of $\kk_D$ that is simple with eigenfunction $\vp_0$ satisfying $\vp_0 > 0$ and $\norm{\vp_0}_{L^2(D_0)} = 1$. Then, for sufficiently large $L$, there exists exactly two eigenvalues $\lad_{\pm}$ of the Newtonian potential $\kk_{D}$ around $\lad_0$, satisfying 
\begin{align*}
  \text{$\lad_\pm = \lad_0 \pm k_I + \Or(L^{-2}) \to \lad_0$\q as\ \, $L \to \infty$,\q and\ \, $\lad_- < \lad_0 < \lad_+$}\,,
\end{align*}
where 
\begin{align*}
    k_I := \l \mathcal{R} \mc{T}_L \vp_0, \kk_D[\mc{T}_{L} \vp_0] \r_D = \Or(L^{-1})> 0\,.
\end{align*}
Moreover, the normalized eigenfunctions $\vp_\pm$ associated with $\lad_\pm$ satisfy
\begin{align*}
  \mathcal{R} [\vp_\pm] = \pm \vp_\pm\,,\q \norm{\vp_\pm}_{L^2(D)} = 1\,,
\end{align*} 
and as $L \to 0$, 
\begin{align} \label{eq:coneigfunc}
  \|\vp_\pm - (\mc{T}_{L} \vp_0 \pm \mathcal{R} \mc{T}_L \vp_0)/\sqrt{2} \|_{L^2(D)} \to 0\,,    
\end{align}
where the operators $\rr[\dd]$ and $\mc{T}_L[\dd]$ are defined in \eqref{def:reffunc} and \eqref{def:transfunc}.
\end{proposition}

The proof of \cref{prop:veriassp} follows from the standard perturbation theory and is provided in \cref{app:A} for completeness. \cref{thm:3dexist}, along with \cref{coro:smallsym}, ensures the existence of two solution curves to the nonlinear resonance problem \eqref{eq:nonLippscalar0}, parametrized by small $\mathcal{N}$ \eqref{eq:normalizecond}, which bifurcate from the zero solution at the linear subwavelength resonances $ \omega_{*,\pm} = 1/\sqrt{\lambda_\pm \tau} + \mathcal{O}(\tau^{-1})$, with symmetric and antisymmetric resonant states, respectively. We next show in \cref{thm:2ndbirfuc} that under further \cref{asspcob} below (that also holds in the dilute regime), when the normalization constant $\mc{N}$ increases, due to the hybridization of these two symmetric and antisymmetric modes and the cubic nonlinearity, a symmetry-breaking bifurcation could occur on the symmetric solution branch emanating from $(\ww_{*,+},0)$ at a critical $\mc{N}_c$. 

For our purpose, again, we consider the scaled Lippmann-Schwinger equation \eqref{eq:nonLippscalar2}:
\begin{align} \label{eq:nonlineardimer}
        u = \hat{\omega}^2 \mathcal{K}_D^{\epsilon \hat{\omega}} \left[N(u)\right]\,,
\end{align}
for $(\h{\ww}, u) \in \C \times H^2(D)$, under the constraint: 
\begin{align} \label{normalconst}
    \int_D |u(x)|^2 \ud x = \mc{N}\,,
\end{align}
as in \cref{sec:existasymp}, where $D = D_1 \cup D_2$ is a dimer with reflection symmetry. Under \cref{assp:second}, recall the largest two eigenvalues $0< \lad_- < \lad_+$ of the Newtonian potential $\kk_D$ with normalized even and odd modes $\vp_-$ and $\vp_+$, respectively:
\begin{equation} \label{eq:symmode}
    \rr[\vp_+] = \vp_+\,,\q  \rr[\vp_-] = \vp_-\,.
\end{equation}
Then, we define the corresponding $L^2$ projections by 
\begin{align*}
  P_+ [u] = \langle \vp_+, u \rangle_D\,\vp_+, \quad P_- [u] = \langle \vp_-, u \rangle_D \,\vp_-, \quad P_r = I - P_+ - P_-\,.   
\end{align*}
Similarly to \eqref{parau}, we assume the following solution ansatz to \eqref{eq:nonlineardimer}:
\begin{align} \label{eq:solansa}
  u = c_+ \vp_+ + c_- \vp_- + r \in H^2(D)\,,
\end{align}
where 
\begin{align}\label{anszsol2}
 c_\pm = \l \vp_\pm, u \r_D \in \C\,,\q r = P_r u \in H_\sim^2(D) = H^2(D) \cap P_r L^2(D)\,.
\end{align}
For simplicity, we define the constants, for $k,l,m,n = \pm$, 
\begin{align} \label{def:constakl}
  a_{klmn} := \l \vp_k, \vp_l \overline{\vp_m} \vp_n \r_D = \int_D \overline{\vp_k} \vp_l \overline{\vp_m} \vp_n \ud x\,.
\end{align}
Noting that the integral kernel of $\mathcal{K}_{D}$ is real, we can also let the eigenfunctions $\vp_\pm$ be real. Moreover, thanks to the symmetry \eqref{eq:symmode}, we find $a_{klmn} \in \R$ and 
\begin{align*}
  a_{klmn} = 0  \q \text{if there are one $-$ or three $-$ in $(k,l,m,n)$,}
\end{align*}
and denote 
\begin{align} \label{def:coeffA}
a_{klmn} = 
\begin{dcases}
    A_{++} = \norm{\vp_+^2}_{L^2(D)}^2  &\q \text{if there are four $+$},\\
    A_{+-} = \norm{\vp_+ \vp_-}_{L^2(D)}^2 &\q \text{if there are two $+$},\\
    A_{--} = \norm{\vp_-^2}_{L^2(D)}^2  &\q \text{if there are zero $+$}.
\end{dcases}  
\end{align}
With the help of \eqref{eq:solansa}, it is easy to write, by a direct expansion of $u$, 
\begin{align} \label{expkdcub}
  \mc{K}_{D}\left[|u|^2 u \right] = \underbrace{\sum_{l,m,n = \pm} c_l \overline{c_m} c_n \mc{K}_{D} \left[ \vp_l \widebar{\vp_m} \vp_n \right]}_{=: F(c_+,c_-)} + R(c_+,c_-, r)\,.
\end{align}

We now reformulate \eqref{eq:nonlineardimer} as 
\begin{align} \label{eq:nonlineardimer2}
  (1 - \h{\ww}^2 \kk_D)[u] = \h{\ww}^2 \kk_D[|u|^2 u] + \h{\ww}^2 \rr^{\eps \h{\ww}}[N(u)]\,,
\end{align}
with $\rr^{\eps \h{\ww}} = \kk_D^{\eps \h{\ww}} - \kk_D$, and derive, by applying projections $P_+$, $P_-$, and $P_r$ to \eqref{eq:nonlineardimer2},
\begin{align}
    &  (1 - \h{\ww}^2 \lad_+) c_+ = \h{\ww}^2 \lad_+ \sum_{l,m,n} c_l \overline{c_m} c_n a_{+lmn} + \h{\ww}^2 \l \vp_+, R(c_+,c_-,r) \r_D + \h{\ww}^2 \l \vp_+, \rr^{\eps \h{\ww}}[N(u)]  \r_D \,, \label{eq:leading1}  \\
    & (1 - \h{\ww}^2 \lad_-) c_- = \h{\ww}^2 \lad_- \sum_{l,m,n} c_l \overline{c_m} c_n a_{-lmn} + \h{\ww}^2 \l \vp_-, R(c_+,c_-,r) \r_D + \h{\ww}^2 \l \vp_-, \rr^{\eps \h{\ww}}[N(u)] \r_D \,,   \label{eq:leading2}  \\
    &  (1 - \h{\ww}^2 \kk_D)[r] = \h{\ww}^2 
    P_r (F(c_+,c_-) + R(c_+,c_-,r))
    + \h{\ww}^2 P_r \rr^{\eps \h{\ww}}[N(u)]\,, \label{eq:leading3}
\end{align}
where 
\begin{equation} \label{coeff1}
    \begin{aligned}
          \sum_{l,m,n} c_l \overline{c_m} c_n a_{+ l m n} & = |c_+|^2 c_+ a_{+ + + +} + |c_-|^2 c_+ (a_{+ + - -}  + a_{+ - - +})  + c_- \overline{c_+} c_- a_{+ - + -} \\
  & = |c_+|^2 c_+ A_{++} + \left(2 |c_-|^2 c_+ + c_-^2 \overline{c_+}  \right) A_{+-}\,,
    \end{aligned}
\end{equation}
and 
\begin{equation} \label{coeff2}
    \begin{aligned}
          \sum_{l,m,n} c_l \overline{c_m} c_n a_{- l m n} & =  |c_-|^2 c_- a_{- - - -} + |c_+|^2 c_- (a_{- - + +}  + a_{- + + -})  + c_+ \overline{c_-} c_+ a_{- + - +} \\
    & = |c_-|^2 c_- A_{--} + \left(2 |c_+|^2 c_- + c_+^2 \overline{c_-}  \right) A_{+-}\,,     
    \end{aligned}
\end{equation}
and the functions $F(\dd)$ and $R(\dd)$ in \eqref{eq:leading3} are defined by \eqref{expkdcub}. Now, we are interested in the solution $(c_+, c_-, r) \in \C^2 \times H_\sim^2(D)$ to the system \eqref{eq:leading1}--\eqref{eq:leading3} subject to 
\begin{align} \label{eq:constcoeff}  
     |c_+|^2 + |c_-|^2 + \int_D |r|^2 \ud x = \mc{N}\,,
\end{align}
from \eqref{normalconst}.
The following resolvent estimate of $\kk_D$ is standard. 
\begin{lemma} \label{lem:resolv}
    Let $\h{\ww} \in \C$ such that ${\dist}(\h{\ww}^{-2}, \si(\kk_D)\backslash\{\lad_{\pm}\}) > 0$. Then the operator $1 - \h{\ww}^2 \kk_D$ is invertible on the space $H_\sim^2(D)$ defined in \eqref{anszsol2}. Moreover, there exists a constant $C > 0$ such that given $f \in H_\sim^2(D)$, the following estimate holds, for $u = (1 - \h{\ww}^2 \kk_D)^{-1} [f]$,
    \begin{equation*}
        \norm{u}_{H^2(D)} \le \left(\frac{C}{{\dist}(\h{\ww}^{-2}, \si(\kk_D)\backslash\{\lad_{\pm}\})} + 1\right)\norm{f}_{H^2(D)}\,.
    \end{equation*}
     In particular, define the open set 
    \begin{align} \label{def:omegafre}
        \Omega: = \{\h{\ww} \in \C\,;\ {\dist}(\h{\ww}^{-2}, \si(\kk_D)\backslash\{\lad_{\pm}\}) > d_* \}\,,
    \end{align}
    with constant $d_* := \tfrac{1}{2}  {\rm dist}(\lad_{\pm},\, \sigma(\kk_D)\backslash \{\lad_\pm\})$. It holds that 
      \begin{equation} \label{eq:resolkd}
        \norm{u}_{H^2(D)} \le \left( C \frac{1}{d_*} + 1 \right) \norm{f}_{H^2(D)}\,,\q \text{for $\h{\ww} \in \Omega$ uniformly.}
    \end{equation}
\end{lemma}

\begin{proof}
We consider the equation $(1 - \h{\ww}^2 \kk_D)[u] = f$ for $f \in H_\sim^2(D)$. Since $\kk_D$ is compact self-adjoint on $L^2(D)$, we have 
    \begin{equation*}
        \norm{u}_{L^2(D)} \le \frac{|\h{\ww}|^{-2}}{{\dist}(\h{\ww}^{-2}, \si(\kk_D)\backslash\{\lad_{\pm}\})} \norm{f}_{L^2(D)}\,.
    \end{equation*}
    Then it follows that 
    \begin{align*}
        \norm{u}_{H^2(D)} & \le |\h{\ww}|^2 \norm{\kk_D[u]}_{H^2(D)} + \norm{f}_{H^2(D)} \\
    & \le C |\h{\ww}|^2 \norm{u}_{L^2(D)} + \norm{f}_{H^2(D)} \\
    & \le \left(\frac{C}{{\dist}(\h{\ww}^{-2}, \si(\kk_D)\backslash\{\lad_{\pm}\})} + 1\right)\norm{f}_{H^2(D)}\,.
    \end{align*}
    The proof is complete by the above estimates. 
\end{proof}

Similarly to the proof of \cref{thm:smallampl}, we first solve $r \in H_\sim^2(D)$ in terms of parameters $(\h{\ww}, c_+, c_-, \eps)$ via \eqref{eq:leading3}. For this, we define the mapping:
\begin{multline} \label{eq:nondimerf}
    \mc{F}(\h{\ww}, c_+, c_-, r, \eps)
 := r -  
   (\h{\ww}^{-2} - \kk_D)^{-1}  P_r (F(c_+,c_-) + R(c_+,c_-,r))
    \\ - (\h{\ww}^{-2} - \kk_D)^{-1} P_r \rr^{\eps \h{\ww}}[N(u)] 
    : \Omega \times \C \times \C \times H_\sim^2(D) \times \R_+\mapsto H_\sim^2(D)\,,
\end{multline}
where $\Omega$ is the set defined in \eqref{def:omegafre}. It is clear that the solution to \eqref{eq:leading3} is given by the zero set of $\mc{F}(\dd)$. We now compute 
\begin{align*}
    \mc{F}(\h{\ww},0,0,0,\eps) = 0\,, \q \p_r\mc{F}(\h{\ww},0,0,0,\eps) = I - (\h{\ww}^{-2} - \kk_D)^{-1} P_r \rr^{\eps \h{\ww}}\,.
\end{align*}
For sufficiently small $\eps$ such that 
\begin{align*} 
    \norm{(\h{\ww}^{-2} - \kk_D)^{-1} P_r \rr^{\eps \h{\ww}}}_{H^2(D) \to H^2(D)} = \Or\left(\frac{\eps}{d_*} + \eps \right) < 1\,,
\end{align*}
by the resolvent estimate \eqref{eq:resolkd}, the linear mapping $\p_r\mc{F}(\h{\ww},0,0,0,\eps)$ is invertible. Then the implicit function theorem gives the function $r(\dd) \in H_\sim^2(D)$ locally in a neighborhood of $(\h{\ww},0,0,\eps)$. Next, we show that such a neighborhood can be chosen independently of $(\h{\ww},\eps)$. 

\begin{proposition} \label{prop:implicitr}
Let the subset $\Omega \subset \C$ be defined in \eqref{def:omegafre} with $$d_* = \tfrac{1}{2}  {\rm dist}(\lad_{\pm},\, \sigma(\kk_D)\backslash \{\lad_\pm\}).$$ Then, there exists constants $c_* > 0$ and $\eps_* > 0$,  depending on $d_*$, such that for 
\begin{align*}
    \h{\ww} \in \Omega\,, \q |c_+| + |c_-| < c_*\,,\q \eps < \eps_*\,,
\end{align*}
we can uniquely solve $r(\h{\ww},c_+,c_-,\eps)$ from \eqref{eq:leading3}, which is a real-analytic mapping and satisfies the estimate: for some $C > 0$ depending on $d_*$, 
\begin{align} \label{est:rforu}
  \norm{r(\h{\ww},c_+,c_-,\eps)}_{H^2(D)} \le C \left(|c_+|^2 + |c_-|^2 + \eps \right) (|c_+| + |c_-|)\,.
\end{align}
\end{proposition}

\begin{proof}
We continue the above arguments and consider the nonlinear mapping $\mc{F}$ in \eqref{eq:nondimerf}. 
It suffices to estimate $(\h{\ww}, c_+, c_-, r, \eps)$ such that $\p_r \mc{F}(\dd)$ is invertible. We compute 
\begin{align*}
   \p_r\mc{F} = I -
   (\h{\ww}^{-2} - \kk_D)^{-1}  P_r \kk_D \left[\p_r |u|^2 u \right]
    -  (\h{\ww}^{-2} - \kk_D)^{-1} P_r \rr^{\eps \h{\ww}} \left[I + \p_r |u|^2 u \right].
\end{align*}
For $u = c_+ \vp_+ + c_- \vp_- + r$ as in \eqref{eq:solansa}, it is easy to see  
\begin{align} \label{auxeq:drcubu}
   \left\|\p_r (|u|^2 u)\right\|_{L^2(D) \to L^2(D)} \le 3 \norm{u}_{L^\infty(D)}^2 \le C \norm{u}_{H^2(D)}^2\,,
\end{align}
and then, it follows from \eqref{eq:resolkd} that 
\begin{align} \label{keyest1}
  \left\|\p_r \mc{F} - I \right\|_{H^2(D) \to H^2(D)} \le C\left(\frac{1}{d_*} + 1 \right) \left(\norm{u}^2_{H^2(D)} + \eps\right),
\end{align}
using \eqref{auxeq:drcubu} and $\norm{\p_r (|u|^2 u)}_{H^2(D) \to L^2(D)} \le \norm{\p_r (|u|^2 u)}_{L^2(D) \to L^2(D)}$, where $C > 0$ is a generic constant. Then, for sufficiently small $\norm{u}_{H^2(D)}$ and $\eps$, the map $\p_r \mc{F}$ is invertible on $H_\sim^2(D)$ by \eqref{keyest1}. In this case, it follows that 
\begin{align*}
    \w{\mc{F}}(\h{\ww},c_+, c_-, r, \eps) :=   \mc{F}(\h{\ww},c_+, c_-, r, \eps) - r 
\end{align*}
is Lipschitz in $r \in H_\sim^2(D)$ with Lipschitz constant $0 < \kappa_{\rm Lip} < 1$, uniformly in small $\norm{u}_{H^2(D)}$ and $\eps$. We next use \eqref{eq:leading3}, namely, $\mc{F}(\h{\ww},c_+, c_-, r, \eps) = 0$ and find 
\begin{multline*}
  \norm{r}_{H^2(D)} - \norm{\w{\mc{F}}(\h{\ww},c_+, c_-, 0, \eps)}_{H^2(D)} \le \norm{- r - \w{\mc{F}}(\h{\ww},c_+, c_-, 0, \eps)}_{H^2(D)}  \\ = \norm{\w{\mc{F}}(\h{\ww},c_+, c_-, r, \eps) - \w{\mc{F}}(\h{\ww},c_+, c_-, 0, \eps)}_{H^2(D)} \le \kappa_{\rm Lip} \norm{r}_{H^2(D)}\,,
\end{multline*}
and hence, for some $C > 0$,
\begin{equation} \label{auxestr}
    \begin{aligned}
         \norm{r}_{H^2(D)} & \le \frac{1}{1 - \kappa_{\rm Lip}} \norm{\w{\mc{F}}(\h{\ww}, c_+, c_-, 0, \eps)}_{H^2(D)} \\ 
  & \le C \left(\frac{1}{d_*} + 1 \right) \left(\norm{c_+ \vp_+ + c_- \vp_-}_{H^2(D)}^2 + \eps \right) \norm{c_+ \vp_+ + c_- \vp_-}_{L^2(D)}\,,
    \end{aligned}
\end{equation}
by \eqref{eq:resolkd} and \eqref{eq:nondimerf}, as well as \eqref{eq:priestcub},
which implies 
\begin{equation} \label{keyest2}
    \begin{aligned}
        \norm{u}_{H^2(D)} & \le \norm{c_+ \vp_+ + c_- \vp_-}_{H^2(D)} + \norm{r}_{H^2(D)}  \\
    & \le C \left(1 + \left(\frac{1}{d_*} + 1 \right) \left(|c_+|^2 + |c_-|^2 + \eps\right)\right) \left(|c_+| + |c_-|\right)\,.
    \end{aligned}
\end{equation}
Therefore, we can let $|c_+| + |c_-|$ and $\eps$ be small enough such that $\norm{u}_{H^2(D)}$ is sufficiently small and $\p_r \mc{F}$ is invertible. The implicit
function theorem guarantees the existence of $r$. The estimate \eqref{est:rforu} simply follows from \eqref{auxestr}. The proof is complete.
\end{proof}

By the uniqueness of $r(\dd)$ and $S^1\times \mathbb{Z}_2$ equivariance of \eqref{eq:nonlineardimer} shown in \cref{lem:invarphaseu,lem:z2equivar}, we have the following properties of the function $r(\dd)$. 

\begin{corollary} \label{coro:properr}
Under the assumptions of \cref{prop:implicitr}, the function $r(\h{\ww},c_+,c_-,\eps)$ satisfies 
\begin{equation} \label{eq:equivar}
    r(\h{\ww}, e^{i \theta} c_+, e^{i \theta} c_-, \eps) =  e^{i \theta} r(\h{\ww}, c_+, c_-, \eps)\,,\q \theta \in \R\,,
\end{equation}
and  
\begin{equation} \label{eq:conjugate}
    \overline{r\left(\h{\ww}, c_+,c_-, \eps \right)} = r\left(\overline{\h{\ww}}, \overline{c_+},\overline{c_-}, \eps \right).
\end{equation}
Moreover, $r(\h{\ww}, c_+, 0, \eps)$ is even, while $r(\h{\ww}, 0, c_-, \eps)$ is odd, namely, 
\begin{align} \label{eq:symmrfun}
    \rr[r(\h{\ww}, c_+, 0, \eps)] = r(\h{\ww}, c_+, 0, \eps)\,,\q \rr[r(\h{\ww}, 0, c_-, \eps)] = - r(\h{\ww}, 0, c_-, \eps)\,.
\end{align}
\end{corollary}

\begin{proof}
The property \eqref{eq:equivar} is by noting that if $(c_+,c_-,r)$ solves \eqref{eq:leading3}, then so does $e^{i\theta} (c_+,c_-,r)$ for any $\theta \in \R$. Taking the complex conjugate of \eqref{eq:leading3} with the uniqueness of $r$ readily gives \eqref{eq:conjugate}. We only show $\rr[r(\h{\ww}, c_+, 0, \eps)] = r(\h{\ww}, c_+, 0, \eps)$, and the other one can be proved in the same way. We apply the reflection $\rr[\dd]$ to \eqref{eq:leading3} with $c_- = 0$, and find 
\begin{align*}
     \rr[r] = \hat{\omega}^2 P_r \mathcal{K}_D^{\epsilon \hat{\omega}} \rr \left[N(u)\right],
\end{align*}
due to $\rr P_r \kk_D^{\eps \h{\ww}} = P_r \kk_D^{\eps \h{\ww}} \rr$, where $u = c_+ \vp_+ + r$ and 
\begin{align*}
   \rr[N(u)] = N(c_+ \vp_+ + \rr[r])\,.
\end{align*}
Then again, the uniqueness of $r$ yields $\rr[r] = r$ when $c_- = 0$. 
\end{proof}

Substituting $r(\h{\ww},c_+,c_-,\eps)$ solved from \eqref{eq:leading3} into equations \eqref{eq:leading1},\eqref{eq:leading2}, and \eqref{eq:constcoeff} gives the closed equations for $(\h{\ww}, c_+,c_-)$. Before formulating it (see \eqref{eq:leading11}, \eqref{eq:leading22}, and \eqref{normaleq} below), we analyze the terms 
$\l \vp_\pm, R(c_+,c_-,r) \r_D$ and $\l \vp_\pm, \rr^{\eps \h{\ww}}[N(u)] \r_D$ in equations \eqref{eq:leading1} and \eqref{eq:leading2} with $r = r(\h{\ww},c_+,c_-,\eps)$. 

Note from the symmetry property \eqref{eq:symmrfun} of $r(\dd)$ and the definition
\begin{align} \label{def:rcpm}
  R(c_+,c_-,r) = \kk_{D}\left[|u|^2 u \right] - \kk_D\left[|c_+ \vp_+ + c_- \vp_-|^2 (c_+ \vp_+ + c_- \vp_-)\right] 
\end{align}
that the function $R(c_+,c_-,r)|_{c_+ = 0}$ is odd, and hence $\l \vp_+, R(c_+, c_-, r)|_{c_+ = 0} \r_D = 0$. Similarly, we have $\l\vp_-, R(c_+, c_-, r)|_{c_- = 0} \r_D = 0$. Then,  by Taylor expansion, there exist complex-valued real analytic functions $f_+$ and $f_-$ such that 
\begin{align} \label{def:fpm}
    \langle \vp_\pm, R(c_+, c_-, r) \rangle_D  = c_\pm f_\pm(\h{\ww}, c_+, c_-, \eps)\,,
\end{align}
with the following properties observed from \cref{coro:properr}: for $\theta \in \R$, 
\begin{align} \label{invar:fpm}
       f_\pm(\h{\ww}, e^{i\theta} c_+, e^{i\theta} c_-, \eps) = f_\pm(\h{\ww}, c_+, c_-, \eps)\,,\q 
      \overline{f_\pm(\h{\ww}, c_+, c_-, \eps)} = f_\pm(\overline{\h{\ww}}, \overline{c_+}, \overline{c_-}, \eps)\,,
\end{align}
and 
\begin{align} \label{est:fpm}
   |f_\pm(\h{\ww}, c_+, c_-, \eps)| \le C \left(|c_+|^2 + |c_-|^2 + \eps \right) \left(|c_+|^2 + |c_-|^2 \right),
\end{align}
by \eqref{est:rforu}, \eqref{def:rcpm}, and \eqref{def:fpm}. Similarly, by \eqref{eq:symmrfun}, the function $N(u)|_{c_+ = 0}$ is odd while $N(u)|_{c_- = 0}$ is even. It follows that 
 $$ \l \vp_\pm, \rr^{\eps \h{\ww}}[N(u)]|_{c_\pm = 0} \r_D = 0\,,$$
which defines functions $g_\pm$ by 
\begin{align} \label{def:gpm}
  \l \vp_\pm, \rr^{\eps \h{\ww}}[N(u)] \r_D  = c_\pm g_\pm(\h{\ww}, c_+, c_-, \eps)\,.
\end{align}
Moreover, we have 
\begin{align} \label{invar:gpm}
       g_\pm(\h{\ww}, e^{i\theta} c_+, e^{i\theta} c_-, \eps) = g_\pm(\h{\ww}, c_+, c_-, \eps)\,,\q 
      \overline{g_\pm(\h{\ww}, c_+, c_-, \eps)} = g_\pm(\overline{\h{\ww}}, \overline{c_+}, \overline{c_-}, \eps)\,,
\end{align}
and 
\begin{align} \label{est:gpm}
   g_\pm(\h{\ww}, c_+, c_-, \eps) = \Or(\eps)\,.
\end{align}

The following lemmas are crucial for the following discussion. Due to the $S^1$ equivariance in \eqref{eq:equivar}, we often, without loss of generality, assume either $c_+ \in \mathbb{R} $ or $c_- \in \mathbb{R}$ during the analysis.

\begin{lemma} \label{lem:symderivr}
Letting $c_- \in \R$, the derivative $\p_{c_-}^j r(\h{\ww},c_+,0,\eps)$, $j \ge 0$, is antisymmetric (odd) for odd $j$, and it is symmetric (even) for even $j$. Similarly, let $c_+ \in \R$, and then the derivative $\p_{c_+}^j r(\h{\ww},0, c_-, \eps)$, $j \ge 0$, is odd for even $j$, and it is even for odd $j$. 
\end{lemma}

\begin{proof}
We only prove the first statement. We consider \eqref{eq:leading3} with $r = r(\h{\ww},c_+,c_-,\eps)$, $c_- \in \R$.  Taking the derivatives of the equation in $c_-$ gives, for $j \ge 0$,
\begin{align} \label{eq:generalderivj}
    \p_{c_-}^j r\Big|_{c_- = 0} = \h{\ww}^2 P_r \kk_D^{\eps \h{\ww}}\left[ \p_{c_-}^j u + \p_{c_-}^j |u|^2 u \right] \Big|_{c_- = 0} 
\end{align}
that can uniquely determine $ \p_{c_-}^j r|_{c_- = 0}$. Here 
\begin{align} \label{eq:genrucub}
    \p_{c_-}^j |u|^2 u = \sum_{\substack{\alpha + \beta + \gamma = j \\ \alpha, \beta, \gamma \ge 0}} \frac{j!}{\alpha! \beta! \gamma!} \p_{c_-}^\alpha u\, \overline{ \p_{c_-}^\beta u}\,  \p_{c_-}^\gamma u \,.
\end{align}
For $j = 0$, the fact that $r(\h{\ww},c_+,0,\eps)$ is even has been proved in \cref{coro:properr}. For $j = 1$, we find 
\begin{align} \label{eq:1stdevr}
      \p_{c_-} r|_{c_- = 0} = \h{\ww}^2 P_r \kk_D^{\eps \h{\ww}}\left[\vp_-  + \p_{c_-} r  + \p_{c_-} |u|^2 u \right]|_{c_- = 0}\,,
\end{align}
from \eqref{eq:generalderivj}, where at $c_- = 0$, 
\begin{equation} \label{symarg}
    \begin{aligned}
        \p_{c_-} |u|^2 u & = 2 |u|^2 (\p_{c_-} u)  + u^2 \overline{\p_{c_-} u} \\
& = \text{(even func.) $\dd$ (odd func.$+ \p_{c_-} r$)} + \text{(even func.) $\dd$ (odd func.$+ \overline{\p_{c_-} r}$)} \\
& = \text{odd func.} + \text{(even func.) $\dd$ $ \p_{c_-} r$} + \text{(even func.) $\dd$ $ \overline{\p_{c_-} r}$}\,,
    \end{aligned}
\end{equation}
since $u = c_+ \vp_+ + r(\h{\ww},c_+,0,\eps)$ is even by \eqref{eq:symmrfun}, and $\p_{c_-} u = \vp_- + \p_{c_-} r$. Noting that $$\rr P_r \kk_D^{\eps \h{\ww}} = P_r \kk_D^{\eps \h{\ww}} \rr$$ and that $P_r \kk_D^{\eps \h{\ww}}$ preserves the symmetry, we can derive, by \eqref{eq:1stdevr} and \eqref{symarg}, 
\begin{align} \label{symarg2}
      \p_{c_-} r|_{c_- = 0} = \text{odd func.} + \h{\ww}^2 P_r \kk_D^{\eps \h{\ww}}\left[\text{(even func.) $\dd$ $ \p_{c_-} r$} + \text{(even func.) $\dd$ $ \overline{\p_{c_-} r}$} \right]\,.
\end{align}
Therefore, the uniqueness of solution readily implies that $\p_{c_-} r|_{c_- = 0}$ is odd. 

Now, assume that 
$\p_{c_-}^j r$ at $c_- = 0$ is odd for any odd $j \le J$, and it is even for any even $j \le J$. We first consider even $J$ and show that $\p_{c_-}^{J + 1} r(\h{\ww},c_+,0,\eps)$ is odd. It is easy to see that $\p_{c_-}^\alpha u |_{c_- = 0}$ with $\alpha \le J$ is odd for odd $\alpha$, and even for even $\alpha$, and that an odd integer $J + 1$ can only split as 
\begin{align*}
    J + 1 = \text{odd + odd + odd, or, odd + even + even}. 
\end{align*}
Thus, similarly to \eqref{symarg}, from \eqref{eq:genrucub}, we have 
\begin{align*}
     \p_{c_-}^{J + 1} |u|^2 u = \text{odd func.} + \text{(even func.) $\dd$ $ \p_{c_-}^{J + 1} r$} + \text{(even func.) $\dd$ $ \overline{\p_{c_-}^{J + 1} r}$}\,,
\end{align*}
which implies, by \eqref{eq:generalderivj}, 
\begin{align*}
        \p_{c_-}^{J + 1} r|_{c_- = 0} = \text{odd func.} + \h{\ww}^2 P_r \kk_D^{\eps \h{\ww}}\left[\text{(even func.) $\dd$ $ \p_{c_-}^{J + 1} r$} + \text{(even func.) $\dd$ $ \overline{\p_{c_-}^{J + 1} r}$} \right].
\end{align*}
It follows that $\p_{c_-}^{J + 1} r|_{c_- = 0}$ is odd, as desired. We now consider odd $J$ and show that $\p_{c_-}^{J + 1} r(\h{\ww},c_+,0,\eps)$ is even. Similarly, noting that an even integer $J + 1$ only split as 
\begin{align*}
    J + 1 = \text{even + even + even, or, odd + odd + even},
\end{align*}
we have 
\begin{align*}
     \p_{c_-}^{J + 1} |u|^2 u = \text{even func.} + \text{(even func.) $\dd$ $ \p_{c_-}^{J + 1} r$} + \text{(even func.) $\dd$ $ \overline{\p_{c_-}^{J + 1} r}$}\,,
\end{align*}
and thus
\begin{align*}
        \p_{c_-}^{J + 1} r|_{c_- = 0} = \text{even func.} + \h{\ww}^2 P_r \kk_D^{\eps \h{\ww}}\left[\text{(even func.) $\dd$ $ \p_{c_-}^{J + 1} r$} + \text{(even func.) $\dd$ $ \overline{\p_{c_-}^{J + 1} r}$} \right]\,,
\end{align*}
implying that $\p_{c_-}^{J + 1} r|_{c_- = 0}$ is even. This completes the proof.
\end{proof}

\begin{lemma} \label{lem:expfpm}
Let the real analytic functions $f_\pm(\dd)$ admit the expansion in $(c_+,c_-)$:
\begin{align} \label{eq:powerserfpm}
    f_\pm (\h{\ww}, c_+,c_-, \eps) = \sum_{k,l,m,n \ge 0} b_{klmn}^\pm(\h{\ww},\eps) c_+^k \wb{c_+}^l c_-^m \wb{c_-}^n\,.
\end{align}
Then, the indices $(k,l,m,n)$ are constrained by 
\begin{align} \label{properfpm1}
     k + l  + m + n \ge 2\,, \q k + m = l + n\,,
\end{align}
and the coefficients $b_{klmn}^\pm(\h{\ww},\eps)$ satisfies 
\begin{align} \label{properfpm3}
   \overline{b_{klmn}^\pm(\h{\ww},\eps)} = b_{klmn}^\pm(\wb{\h{\ww}},\eps)\,, \q 
\end{align}
and 
\begin{align} \label{properfpm2}
    b_{klmn}^\pm(\h{\ww},\eps) = 0\,,\q \text{for odd $\,m + n$}\,,
\end{align}
as well as
\begin{align} \label{properfpm4}
b_{klmn}^\pm(\h{\ww},\eps) = \Or(\eps)\,,\q \text{for $\,  k + l  + m + n = 2$\,.}  
\end{align}
\end{lemma}

\begin{proof}
With the power series expansion \eqref{eq:powerserfpm}, the equivariance property \eqref{invar:fpm} gives $k + m = l + n$ and $$\overline{b_{klmn}^\pm(\h{\ww},\eps)} = b_{klmn}^\pm(\wb{\h{\ww}},\eps)\,.$$ 
The estimate \eqref{est:fpm} implies $b_{klmn}^\pm(\h{\ww},\eps) = \Or(\eps)$ for  $ 2 \le k + l  + m + n \le 3$, and $b_{klmn}^\pm(\h{\ww},\eps) = 0$ for  $ k + l  + m + n \le 1$. We have proved \eqref{properfpm1}, \eqref{properfpm2}, and \eqref{properfpm4}. Next, we show that $b_{klmn}^\pm(\h{\ww},\eps) = 0$ if $m + n$ is odd. We first consider the case of $b_{klmn}^+(\h{\ww},\eps)$. Without loss of generality, by \eqref{invar:fpm}, we can let $c_-$ be real and write the Taylor expansion of $f_+(\dd)$ in $c_- \in \R$ near $c_- =  0$:
\begin{align*}
      f_+ (\h{\ww}, c_+, c_-, \eps) = \sum_{j \ge 0} \p_{c_-}^j f_+(\h{\ww},c_+,0,\eps) c_-^j\,.
\end{align*}
Then, it suffices to prove $\p_{c_-}^j f_+(\h{\ww},c_+,0,\eps) = 0$ for odd $j$. For this, from \cref{lem:symderivr}, we see that $\p_{c_-}^j |u|^2 u|_{c_- = 0}$ is even for even $j$ and odd for odd $j$, and so is $\p_{c_-}^j R(\dd)|_{c_- = 0}$ by \eqref{def:rcpm}. It readily implies $\p_{c_-}^j f_+(\h{\ww},c_+,0,\eps) = 0$ for odd $j$, and hence $b_{klmn}^+(\h{\ww},\eps) = 0$ if $m + n$ is odd. 
In the same way, we can show $b_{klmn}^-(\ww) = 0$ for odd $m + n$ (equivalently, for odd $k + l$ by \eqref{properfpm1}) considering the derivative of $f_-(\h{\ww},c_+,c_-,\eps)$ in $c_+ \in \R$ and \cref{lem:symderivr}. 
\end{proof}

\begin{lemma} \label{lem:expgpm}
Let the real analytic functions $g_\pm(\dd)$ admit the expansion in $(c_+,c_-)$:
\begin{align} \label{eq:powersergpm}
    g_\pm (\h{\ww}, c_+,c_-, \eps) = \sum_{k,l,m,n \ge 0} g_{klmn}^\pm(\h{\ww},\eps) c_+^k \wb{c_+}^l c_-^m \wb{c_-}^n\,.
\end{align}
Then, we have $k + m = l + n$, and the coefficients $g_{klmn}^\pm(\h{\ww},\eps)$ satisfies 
\begin{align} \label{propergpm3}
   g_{klmn}^\pm(\h{\ww},\eps) = \Or(\eps)\,,\q   \overline{g_{klmn}^\pm(\h{\ww},\eps)} = g_{klmn}^\pm(\wb{\h{\ww}},\eps)\,, 
\end{align}
and 
\begin{align} \label{propergpm2}
    g_{klmn}^\pm(\h{\ww},\eps) = 0\,,\q \text{for odd $m + n$}\,.
\end{align}
\end{lemma}

\begin{proof}
It suffices to prove $g_{klmn}^\pm(\h{\ww},\eps) = 0$ for odd $m + n$, as the others are straightforward. We only consider the case of $g_{klmn}^+(\h{\ww},\eps)$. Similarly to the proof of \cref{lem:expfpm}, by \eqref{invar:gpm}, we let $c_-$ be real and write the expansion of $g_+$ in $c_-$: 
\begin{align*}
    g_+ (\h{\ww}, c_+, c_-, \eps) = \sum_{j \ge 0} \p_{c_-}^j g_+(\h{\ww},c_+,0,\eps) c_-^j\,.
\end{align*}
Then, note that
\begin{align*}
     \l \vp_+, \rr^{\eps \h{\ww}} [  \p_{c_-}^j N(u)|_{c_- = 0} ] \r_D  = c_+ \p_{c_-}^j g_+(\h{\ww},c_+,0,\eps)\,,
\end{align*}
and that $\p_{c_-}^j N(u)|_{c_- = 0} = \p_{c_-}^j (u + |u|^2 u)|_{c_- = 0}$ is even for even $j$ and odd for odd $j$. Moreover, note also that $\rr^{\eps \h{\ww}}$ preserves symmetry. This gives $\p_{c_-}^j g_+(\h{\ww},c_+,0,\eps) = 0$ for odd $j$, equivalently, \eqref{propergpm2} holds. 
\end{proof}

Thanks to \cref{lem:expfpm}, using the polar form $c_\pm = p_\pm e^{i \theta_\pm}$ with $p_\pm \ge 0$ and $\Delta \theta = \theta_+ - \theta_-$, we can reformulate the expansion \eqref{eq:powerserfpm} as follows: 
\begin{align} \label{expfpm}
      f_\pm (\h{\ww}, c_+, c_-, \eps) & = \sum_{\substack{ k + l  + m + n \ge 2 \\ k + m = l + n}} b_{klmn}^\pm(\h{\ww},\eps) p_+^{k + l} p_-^{m + n} e^{i\theta_+(k - l)} e^{i\theta_-(m-n)} \notag \\
      & = \sum_{i + s \ge 1\,,\, t \in \ZZ} b_{ist}^\pm(\h{\ww},\eps) p_+^{2 i} p_-^{2 s} e^{i 2 t \Delta \theta} =: f_\pm (\h{\ww}, p_+, p_-, \Delta \theta, \eps)\,,
\end{align}
where in the second equality, we have used $b_{klmn}^\pm(\h{\ww},\eps) = 0$ for odd $m + n$, and that $k - l$, $k + l$, and $m - n$ are even if $m + n$ is even. Similarly, we can reformulate \eqref{eq:powersergpm} as 
\begin{align} \label{expgpm}
      g_\pm (\h{\ww}, c_+,c_-, \eps) & = \sum_{\substack{ k + l  + m + n \ge 0 \\ k + m = l + n}} g_{klmn}^\pm(\h{\ww},\eps)  p_+^{k + l} p_-^{m + n} e^{i\theta_+(k - l)} e^{i\theta_-(m-n)} \notag \\
      & = \sum_{i + s \ge 0\,,\, t \in \ZZ} g_{ist}^\pm(\h{\ww},\eps) p_+^{2 i} p_-^{2 s} e^{i 2 t \Delta \theta} =: g_\pm (\h{\ww}, p_+, p_-, \Delta \theta, \eps). 
\end{align}

We are now ready to solve the following equations in $(\h{\ww}, c_+,c_-)$, by inserting $r(\h{\ww},c_+,c_-,\eps)$ into equations \eqref{eq:leading1}, \eqref{eq:leading2}, and \eqref{eq:constcoeff}, and using \eqref{def:fpm} and \eqref{def:gpm}: 
\begin{multline} \label{eq:leading11}
   (\h{\ww}^{-2} - \lad_+) c_+ = \lad_+ \left(|c_+|^2 c_+ A_{++} + \left(2 |c_-|^2 c_+ + c_-^2 \overline{c_+}  \right) A_{+-}\right) \\ + c_+ f_+(\h{\ww}, c_+, c_-, \eps) + c_+ g_+(\h{\ww}, c_+, c_-, \eps)\,,
\end{multline}
\begin{multline} \label{eq:leading22}
    (\h{\ww}^{-2} - \lad_-) c_- = \lad_- \left(|c_-|^2 c_- A_{--} + \left(2 |c_+|^2 c_- + c_+^2 \overline{c_-}  \right) A_{+-}\right) \\  + c_- f_-(\h{\ww}, c_+, c_-, \eps)  + c_- g_-(\h{\ww}, c_+, c_-, \eps)\,,
\end{multline}
with the normalization constraint: 
\begin{align} \label{normaleq}
  (1 + \Or(\eps))(|c_+|^2 + |c_-|^2) + \Or\left( |c_+|^2 + |c_-|^2 \right)^3  = \mc{N}\,.
\end{align}

We first construct the pure-state solutions to \eqref{eq:leading11}--\eqref{normaleq} without hybridization. The argument is essentially the same as the proof of \cref{thm:smallampl}, and thus we only sketch it for completeness. Without loss of generality, we assume that $c_- = 0$. The analysis for the case of $c_+ = 0$ is similar and, therefore, omitted. 
If $c_- = 0$, \eqref{eq:leading22} holds directly, and \eqref{eq:leading11} is simplified as, for $c_+ \neq 0$, 
\begin{align*}
    \h{\ww}^{-2} - \lad_+ = \lad_+ |c_+|^2  A_{++} +  f_+(\h{\ww}, c_+, 0, \eps) + g_+(\h{\ww}, c_+, 0, \eps)\,.
\end{align*}
Its solution is given by the zero set of the following mapping:
\begin{align*}
    \mc{G}_+(\h{\ww},c_+,\eps) =  \h{\ww}^{-2} - \lad_+ - \lad_+ |c_+|^2  A_{++} - f_+(\h{\ww}, c_+, 0, \eps) - g_+(\h{\ww}, c_+, 0, \eps)\,.
\end{align*}
Again, by equivariance \eqref{invar:fpm}, we let $c_+ \in \R$. It is easy to compute 
\begin{align*}
    & \mc{G}_+(1/\sqrt{\lad_+},0,0) = 0\,, \\
    &  \p_{\h{\ww}}\mc{G}_+(1/\sqrt{\lad_+},0,0) = - 2 \lad_+^{3/2} - \p_{\h{\ww}} f_+(1/\sqrt{\lad_+}, 0, 0, 0) = - 2 \lad_+^{3/2}\,,
\end{align*}
by $g_+(\h{\ww}, c_+, 0, 0) = 0$, and for any $\h{\ww} \in \Omega$ and small $\eps$, 
\begin{equation*}
    f_+(\h{\ww}, 0, 0, \eps) = 0\,,
\end{equation*}
from the estimate \eqref{est:fpm}. By the implicit function theorem, there is a unique solution $\h{\ww} = \h{\ww}_+(c_+, \eps)$ for $(c_+, \eps)$ near $0$. Therefore, we have constructed a family of solutions $(\h{\ww}, e^{i \theta} u)$ to \eqref{eq:nonlineardimer} with symmetric $u$ (\textit{i.e.}, $\rr[u] = u$), parameterized by small $(c_+, \eps) \in \R \times \R_+$, 
\begin{align*}
    (c_+, \eps) \to (\h{\ww},u) = \left(\h{\ww}_+(c_+,\eps), e^{i \theta} c_+ \vp_+ + e^{i \theta} r(\h{\ww}_+(c_+, \eps), c_+, 0, \eps)\right)\,,
\end{align*}
by equivariance \eqref{eq:equivar} and symmetry \eqref{eq:symmrfun}. Furthermore, in analogy with \cref{coro:nonsmall}, one can show that $\h{\ww}_+(0,\eps)$ for a fixed small $\eps$ is the scaled linear dielectric resonance $\h{\ww}_{*,+} = 1/\sqrt{\lad_+} + \Or(\eps)$. In brief, recalling \cref{prop:limit_scalar3d}, let $(\h{\ww}_{*,+}, \vp_{*,+})$ satisfy 
\begin{equation} \label{eq:linearp}  
    \varphi_{*,+} = \hat{\omega}_{*,+}^2 \mathcal{K}_D^{\eps \hat{\omega}_{*,+}}[\varphi_{*,+}]\,,
\end{equation}  
with $\vp_{*,+}$ being symmetric and
\begin{align*}
 \varphi_{*,+} = \varphi_+ + \mathcal{O}(\eps)\,, \quad \| \varphi_+ \|_{L^2(D)}^2 = 1\,, \quad \varphi_* - \varphi_+ \perp \varphi_+\,.   
\end{align*}
We next prove that
\begin{align*}
    \mc{G}_+(\h{\ww}_{*,+}, 0, \eps) = \h{\ww}_{*,+}^{-2} - \lad_+ - \l \vp_+, c_+^{-1} \rr^{\eps \h{\ww}_{*,+}}[N(u)(\h{\ww}_{*,+},c_+,0,\eps)] \r_D = 0\,.
\end{align*}
For this, we compute 
\begin{align*}
      \lim_{c_+ \to 0} c_+^{-1} N(u)(\h{\ww},c_+,0,\eps)  = \vp_+ + \p_{c_+} r(\h{\ww}, 0, 0,\eps)
\end{align*}
and by \eqref{eq:nondimerf}, 
\begin{align*}
    \p_{c_+} r(\h{\ww},0, 0, \eps) & = - \p_r \mc{F}\left(\h{\ww}, 0, 0, 0, \eps\right)^{-1} \p_{c_+} \mc{F}\left(\h{\ww}, 0, 0, 0, \eps\right) \\
      & = \left(1 - \h{\ww}^2 \kk_D - \h{\ww}^2 P_r \rr^{\eps \h{\ww}}\right)^{-1} \h{\ww}^2 P_r \rr^{\eps \h{\ww}}[\vp_+]\,.  
\end{align*}
Moreover, we have $ \p_{c_+} r(\h{\ww}_{*,+},0, 0, \eps) = \vp_{*,+} - \vp_+$, and thus $$ \lim_{c_+ \to 0} c_+^{-1} N(u)(\h{\ww}_{*,+},c_+,0,\eps) = \vp_{*,+}.$$ We can conclude that
\begin{align*}
    \mc{G}_+(\h{\ww}_{*,+}, 0, \eps) = \h{\ww}_{*,+}^{-2} - \lad_+ - \l \vp_+,  \rr^{\eps \h{\ww}_{*,+}}[\vp_{*,+}] \r_D = 0\,,
\end{align*}
taking the inner product of \eqref{eq:linearp} with $\vp_+$. 
We summarize the above discussion as follows.
  
\begin{proposition} \label{sol:purestate}
Under \cref{assp:second}, let $0< \lad_- < \lad_+$  be the largest two eigenvalues of $\kk_D$ with normalized even and odd modes $\vp_-$ and $\vp_+$. Then, the nonlinear resonance problem \eqref{eq:nonlineardimer} admits two associated families of solutions, parameterized as, for sufficiently small $(c_\pm, \eps) \in \R \times \R_+$,
\begin{align} \label{puresol1}
    (c_+, \eps) \to \left(\h{\ww}_+(c_+,\eps), u_+(c_+,\eps) = e^{i \theta} c_+ \vp_+ + e^{i \theta} r(\h{\ww}_+(c_+, \eps), c_+, 0, \eps)\right)\,, \q \theta \in [0, 2 \pi)\,,
\end{align}
and 
\begin{align} \label{puresol2}
    (c_-, \eps) \to   \left(\h{\ww}_-(c_-,\eps), u_-(c_-,\eps) =  e^{i \theta} c_- \vp_- + e^{i \theta} r(\h{\ww}_-(c_-, \eps), 0, c_-, \eps)\right)\,, \q \theta \in [0, 2 \pi)\,,
\end{align}
where $\h{\ww}_\pm(c_\pm,\eps)$ has the asymptotics $$\h{\ww}_\pm(c_\pm,\eps) = \h{\ww}_{*,\pm}(\eps) + \Or(c_\pm^2 + |c_\pm|\eps)$$ with $$\h{\ww}_{*,\pm}(\eps) = 1/\sqrt{\lad_\pm} + \Or(\eps)$$ being the scaled linear resonance associated with $\lad_{\pm}$; the modes $u_+$ and $u_-$ are even and odd, respectively; the function $r(\dd)$ is given as in \cref{prop:implicitr}. 
Moreover, these two families of solutions are the only nontrivial solutions to \eqref{eq:nonlineardimer} when $|c_+| + |c_-|$ and $\eps$ are sufficiently small. 
\end{proposition}

\begin{proof}
    We only prove that for sufficiently small $|c_+| + |c_-|$ and $\eps$, \eqref{puresol1} and \eqref{puresol2} are the only solutions to the system \eqref{eq:leading11}--\eqref{normaleq}. If there exists an extra solution, then both $c_+$ and $c_-$ are non-zero. Then, we can rewrite \eqref{eq:leading11}--\eqref{eq:leading22} as \eqref{eq:lead111}--\eqref{eq:lead222} below, which further gives, by subtracting these two equations and the estimates \eqref{est:fpm} and \eqref{est:gpm}, 
    \begin{align*}
        \lad_+ - \lad_- = \Or(|c_+|^2 + |c_-|^2 + \eps)\,,
    \end{align*}
    which contradicts as $|c_+|^2 + |c_-|^2 + \eps \to 0$. 
\end{proof}

We next show the existence of symmetry-breaking bifurcation along the symmetric solution branch $(\h{\ww}_+(c_+, \eps), u_+(c_+, \eps))$, under \cref{asspcob} below for the coefficients \eqref{def:coeffA}.  

\begin{assumption} \label{asspcob} 
Under \cref{assp:second}, let the constants $A_{++}$, $A_{+-}$, and $A_{--}$ be defined in \eqref{def:coeffA}. We assume that 
\begin{align} \label{rela1}
    \lad_+ A_{++}  - 3 \lad_-   A_{+-} < 0\,,  \q  \lad_+ A_{++}  - \lad_-   A_{+-} > 0\,,
\end{align}
and 
\begin{align*}
    \text{$\frac{\lad_+ - \lad_-}{3 \lad_- A_{+-} - \lad_+ A_{++}}$\, is sufficiently small\,.}
\end{align*}
\end{assumption}

\begin{remark}
From \cref{prop:veriassp}, we know that the eigenvalues $\lad_{\pm}$ bifurcate from the principal eigenvalue $\lad_0$ of $\kk_{D_0}$, and there holds $\lad_\pm \to \lad_0$ as $L \to \infty$, where $D_0$ is the reference particle in \eqref{eq:dimercase}. Additionally, by the convergence of the associated eigenfunctions \eqref{eq:coneigfunc}, we obtain
\begin{align*}
    A_{+\pm} \to \frac{1}{2}\int_{D_0} |\vp_0|^4 \ud x\,,\  \text{ as $\,L \to \infty$}\,,
\end{align*}
Therefore, the assumptions $$\lad_+ A_{++}  - 3 \lad_-   A_{+-} < 0 \q \text{and} \q \frac{\lad_+ - \lad_-}{3 \lad_- A_{+-} - \lad_+ A_{++}} = o(1),$$ which are crucial for proving the existence of a symmetry-breaking bifurcation along the symmetric solution branch (see the proof of \cref{thm:2ndbirfuc}), hold uniformly for all sufficiently large $L$. 

On the other hand, note that in the $L \to 0$ limit, $\lambda_+ A_{++} - \lambda_- A_{+-} \to 0$. To verify condition $\lambda_+ A_{++} - \lambda_- A_{+-} > 0$, which is required to preclude symmetry-breaking bifurcations on the antisymmetric branch, one would need higher-order $1/L$ expansions of $\lambda_\pm$ and $A_{+\pm}$.  However, such expansions might introduce additional assumptions, and we omit them here for simplicity. 
\end{remark}

On the new bifurcated branch, one may expect that there will be a hybridization of the even and odd modes $ \varphi_\pm$, namely $c_+\,, c_- \neq 0$. Thus, we can divide equations \eqref{eq:leading11} and \eqref{eq:leading22} by $c_\pm$, respectively, and then use the polar form $c_\pm = p_\pm e^{i \theta_\pm}$ with \eqref{expfpm}--\eqref{expgpm} to rewrite the equations for the new asymmetric solution as follows:
\begin{multline}  \label{eq:lead111}
 F_+ (\h{\ww}, p_+, p_-, \Delta \theta, \eps) :=  \h{\ww}^{-2} - \lad_+ - \lad_+ \left(p_+^2 A_{++} + p_-^2 \left(2 + e^{- i 2 \Delta \theta}  \right) A_{+-}\right) \\ - f_+(\h{\ww}, p_+, p_-, \Delta \theta, \eps) - g_+(\h{\ww}, p_+, p_-, \Delta \theta, \eps) = 0\,,
\end{multline}
\begin{multline}  \label{eq:lead222} 
  F_- (\h{\ww}, p_+, p_-, \Delta \theta, \eps) :=  \h{\ww}^{-2} - \lad_-  - \lad_- \left(p_-^2  A_{--} + p_+^2  \left(2  + e^{2 i \Delta \theta}  \right) A_{+-}\right) \\ - f_-(\h{\ww}, p_+, p_-, \Delta \theta, \eps) - g_-(\h{\ww}, p_+, p_-, \Delta \theta, \eps) = 0\,.
\end{multline}
We observe, by \eqref{def:gpm},
\begin{align*}
    g_\pm(\h{\ww}, p_+, p_-, \Delta \theta, 0) = 0\,,    
\end{align*}
and that the asymmetric solution curve would converge to the bifurcation point on the symmetric branch as $c_- \to 0$ (\textit{i.e.}, $p_- \to 0$). This motivates us to first determine the bifurcation point in the limiting regime $\eps = 0$ (\textit{i.e.}, $\tau \to \infty$), that is, to solve the following equations in $(\h{\ww},p_+, \Delta\theta)$:
\begin{align} \label{eqbipt}
    F_+ (\h{\ww}, p_+, 0, \Delta \theta, 0) = 0\,, \q F_- (\h{\ww}, p_+, 0, \Delta \theta, 0) = 0\,.
\end{align}
Then, we shall use the implicit function theorem to find the function $(\h{\ww}, p_+,\Delta \theta)$ in small $(p_-,\eps)$ from equations \eqref{eq:lead111} and \eqref{eq:lead222} above.

To solve \eqref{eqbipt}, note that we can restrict $\h{\ww}$ to be real and $b_{ist}^-(\h{\ww},0) \in \R$ for $\h{\ww }\in \R$ by \eqref{properfpm3}. Taking the imaginary part of \eqref{eq:lead222} and using \eqref{properfpm4} give 
\begin{align} \label{eq:seriseq}
     \lad_- A_{+-} p_+^2 \sin (2 \Delta \theta) + \sum_{i + s \ge 2,\, t \in \mathbb{Z}} b_{ist}^-(\h{\ww},0) p_+^{2 i} p_-^{2 s} \sin(2 t \Delta \theta) = 0\,,
\end{align}
whose left-hand side is a power series in $p_+$. It follows that \eqref{eq:seriseq} with $p_- = 0$ holds if and only if 
\begin{align*}
    \sin (2 \Delta \theta) = 0\,,
\end{align*}
that is, 
\begin{align*}
  \Delta \theta \in \left\{0, \frac{\pi}{2}, \pi, \frac{3 \pi}{2}\right\}. 
\end{align*}
We next consider the following two cases separately. 

\medskip 

\noindent 
\underline{\emph{Case I: $\Delta \theta = \{0,\pi\}$}}. Noting that
\begin{align} \label{eq:opi}
    F_\pm(\h{\ww}, p_+, p_-, \Delta \theta, \eps) = F_\pm(\h{\ww}, p_+, p_-, \Delta \theta + \pi, \eps)\,,
\end{align}
it suffices to consider $\Delta \theta = 0$. We have, by \eqref{eqbipt}, 
\begin{align*}
    F_+ (\h{\ww}, p_+, 0, 0, 0) & =  \h{\ww}^{-2} - \lad_+ - \lad_+ p_+^2 A_{++} - f_+(\h{\ww}, p_+, 0, 0, 0) \\
    & = \h{\ww}^{-2} - \lad_+ - \lad_+ p_+^2 A_{++} + \Or(p_+^4)  = 0\,,
\end{align*}
and similarly, 
\begin{align*}
    F_- (\h{\ww}, p_+, 0, 0, 0) = \h{\ww}^{-2} - \lad_-  - 3 \lad_-  p_+^2  A_{+-} + \Or(p_+^4) = 0\,.
\end{align*}
It readily follows the equation in $p_+ > 0$: 
\begin{align} \label{eq:pstar}
    \lad_+ - \lad_- + \left(\lad_+ A_{++}  - 3 \lad_-   A_{+-}\right)  p_+^2 + \Or(p_+^4) = 0\,,
\end{align}
which, by \cref{asspcob}, admits a unique positive solution:
\begin{align} \label{eq:pstar2}
       p_{+,*} = \sqrt{\frac{\lad_+ - \lad_-}{3 \lad_- A_{+-} - \lad_+ A_{++}}} + \text{h.o.t.} \,.
\end{align}
Then, it is easy to solve 
\begin{align} \label{eq:wstar}
    \h{\ww}_* = \sqrt{\frac{1}{ \lad_+ (1 + p_{+,*}^2 A_{++}) + \Or(p_{+, *}^4)}} = \Or(1) > 0\,,
\end{align}
such that 
\begin{align*}
     F_+ (\h{\ww}_*, p_{+,*}, 0, 0, 0) = 0\,, \q F_- (\h{\ww}_*, p_{+,*}, 0, 0, 0) = 0\,.
\end{align*}

Next, we compute the Jacobian of $(\Re F_+, \Re F_-, \Im F_+, \Im F_-)$ in real variables $\Re \h{\ww}$, $\Im\h{\ww}$, $p_+$, and $\Delta \theta$ at the point $(\h{\ww}_*, 0, p_{+,*}, 0)$ with $\h{\ww}_*$ and $p_{+,*}$ given in \eqref{eq:pstar2}--\eqref{eq:wstar}. For this, we first calculate, letting $p_{-,*} := 0$,
    \begin{align} \label{dev1}
          & \frac{\p F_\pm}{\p \Delta \theta}((\h{\ww}_*, 0), p_{+,*}, 0, 0, 0) \notag \\ = & \pm i 2 \lad_\pm p_{\mp, *}^2 A_{+-} - \underbrace{\frac{\p  f_\pm}{\p \Delta \theta}((\h{\ww}_*, 0), p_{+,*}, 0, 0, 0)}_{= \Or(p_{+,*}^4)} - \underbrace{\frac{\p g_\pm}{\p \Delta \theta}((\h{\ww}_*, 0), p_{+,*}, 0, 0, 0)}_{ = 0}  \notag \\
     = & \pm i 2 \lad_\pm p_{\mp, *}^2 A_{+-} +  \Or(p_{+,*}^4)\,,
    \end{align}
by
\cref{lem:expfpm,lem:expgpm} with \eqref{expfpm} and \eqref{expgpm}. 
Since $\pm i 2 \lad_\pm p_{\mp, *}^2 A_{+-}$ is an imaginary number, it follows that
\begin{align*}
     \frac{\p F_\pm}{\p \Delta \theta}((\h{\ww}_*, 0), p_{+,*}, 0, 0, 0) =  \frac{\p \Im F_\pm}{\p \Delta \theta}((\h{\ww}_*, 0), p_{+,*}, 0, 0, 0) +  \Or(p_{+,*}^4)\,.
\end{align*}
Similarly, we have 
\begin{align} \label{dev2}
     \frac{\p F_+}{\p p_+}((\h{\ww}_*, 0), p_{+,*}, 0, 0, 0) = - 2 \lad_+ p_{+,*} A_{++} + \Or(p_{+,*}^3)\,,
\end{align}
and 
\begin{align} \label{dev3}
     \frac{\p F_-}{\p p_+}((\h{\ww}_*, 0), p_{+,*}, 0, 0, 0) = - 6 \lad_- p_{+,*} A_{+-} + \Or(p_{+,*}^3)\,,
\end{align}
which gives, by noting that both $- 2 \lad_+ p_{+,*} A_{++} $ and $- 6 \lad_- p_{+,*} A_{+-} $ are real, 
\begin{align*}
    \frac{\p F_\pm}{\p p_+}((\h{\ww}_*, 0), p_{+,*}, 0, 0, 0) = \frac{\p \Re F_+}{\p p_+}((\h{\ww}_*, 0), p_{+,*}, 0, 0, 0) + \Or(p_{+,*}^3)\,.
\end{align*}
Finally, it is simple to calculate 
\begin{align} \label{dev4}
  \frac{(\p \Re F_\pm, \Im F_\pm)}{\p (\Re \h{\ww}, \Im \h{\ww})}((\h{\ww}_*, 0), p_{+,*}, 0, 0, 0) = \mm -2 \h{\ww}_{*}^{-3} & 0 \\ 0 & -2 \h{\ww}_{*}^{-3} \nn + \Or(p_{+,*}^4)\,.
\end{align}
Then, using \eqref{dev1}--\eqref{dev4}, we readily find that the Jacobian $\frac{\p (\Re F_+, \Re F_-, \Im F_+, \Im F_-)}{\p (\Re \h{\ww},  \Im\h{\ww}, p_+, \Delta \theta)}$ at the point $((\Re \h{\ww}, \Im \h{\ww}), p_{+}, p_-, \Delta \theta, \eps) = ((\h{\ww}_*, 0), p_{+,*}, 0, 0,0)$ can be approximated by a $4 \times 4$ matrix with two $2 \times 2$ blocks $\frac{\p (\Re F_+, \Re F_-)}{\p (\Re \h{\ww}, p_+)}$ and $\frac{\p (\Im F_+, \Im F_-)}{\p (\Im \h{\ww}, \Delta  \theta )}$ with an error matrix of order $\Or(p_{+,*}^3)$. Specifically, there holds, at $((\h{\ww}_*, 0), p_{+,*}, 0, 0,0)$,  
\begin{align*}
    \frac{\p (\Re F_+, \Re F_-, \Im F_+, \Im F_-)}{\p (\Re \h{\ww},  \Im\h{\ww}, p_+, \Delta \theta)} = 
\renewcommand\arraystretch{1.5}
\begin{bmatrix}
    \frac{\p \Re F_+}{\p \Re \h{\ww}} & 0 &  \frac{\p \Re F_+}{\p p_+} & 0 \\
     \frac{\p \Re F_-}{\p \Re \h{\ww}} & 0 &  \frac{\p \Re F_-}{\p p_+} & 0  \\
     0 &  \frac{\p \Im F_+}{\p \Im \h{\ww}} &  0 & \frac{\p \Im F_+}{\p \Delta \theta} \\
        0 &  \frac{\p \Im F_-}{\p \Im \h{\ww}} &  0 & \frac{\p \Im F_-}{\p \Delta \theta}
   \end{bmatrix} + \Or(p_{+,*}^3)\,.
\end{align*}
It is easy to see that the leading-order matrix of the Jacobian is invertible, by \cref{asspcob}, and  
\begin{align*}
    & \det \left( \frac{\p (\Re F_+, \Re F_-)}{\p (\Re \h{\ww}, p_+)} ((\h{\ww}_*, 0), p_{+,*}, 0, 0,0)\right) \\ = & 4 \h{\ww}_*^{-3} p_{+,*} \left(3 \lad_-  A_{+-} - \lad_+ A_{++} \right) + \Or(p_{+,*}^3) \neq 0\,,
\end{align*}
and 
\begin{align*}
    & \det \left( \frac{\p (\Im F_+, \Im F_-)}{\p (\Im \h{\ww}, \Delta  \theta )} ((\h{\ww}_*, 0), p_{+,*}, 0, 0,0) \right)
     \\ = & i 4 \h{\ww}_*^{-3} \lad_- p_{+, *}^2 A_{+-}  + \Or(p_{+,*}^4)  \neq 0\,,
\end{align*}
which further implies that the Jacobian $\frac{\p (\Re F_+, \Re F_-, \Im F_+, \Im F_-)}{\p (\Re \h{\ww},  \Im\h{\ww}, p_+, \Delta \theta)}$ is invertible since 
\begin{align*}
    &\det \left( \frac{\p (\Re F_+, \Re F_-, \Im F_+, \Im F_-)}{\p (\Re \h{\ww},  \Im\h{\ww}, p_+, \Delta \theta)} \right) \\ = &  \det \left( \frac{\p (\Re F_+, \Re F_-)}{\p (\Re \h{\ww}, p_+)} \right) \dd \det \left( \frac{\p (\Im F_+, \Im F_-)}{\p (\Im \h{\ww}, \Delta  \theta )} \right) + \Or(p_{+,*}^3) \neq 0\,.
\end{align*}
Therefore, by the implicit function theorem, there exists a unique solution $ (\h{\ww}, p_+, \Delta \theta)$ to \eqref{eq:lead111} and \eqref{eq:lead222} near $(\h{\ww}_*, p_{+,*}, 0)$, depending on small $(p_-, \eps)$. Recall \eqref{eq:opi}; another solution exists near the point $(\h{\ww}_*, p_{+,*}, \pi)$.

We now compute the asymptotics of $(\h{\ww}, p_+,\Delta \theta)$ in $(p_-,\eps)$. For this, we first note that
\begin{equation*}
    \frac{\p F_\pm}{\p p_-}((\h{\ww}_*, 0), p_{+,*}, 0, 0, 0) = 0\,,
\end{equation*}
and hence 
\begin{equation} \label{eq:expdelta}
    \Delta \theta = \Or(p_-^2 + \eps)\,.
\end{equation}
It follows that, by \eqref{eq:lead111}--\eqref{eq:lead222} with estimates \eqref{est:fpm}--\eqref{est:gpm}, 
\begin{align}
    & \h{\ww}^{-2} - \lad_+ - \lad_+ \left(p_+^2 A_{++} + 3 p_-^2 A_{+-}\right)  = \Or(\eps + p_-^4 + p_{+,*}^4)\,, \label{sbauxeq1} \\
    &  \h{\ww}^{-2} - \lad_-  - \lad_- \left(p_-^2  A_{--} + 3 p_+^2  A_{+-}\right) = \Or(\eps + p_-^4 + p_{+,*}^4)\,, \label{sbauxeq2} 
\end{align}
from which we solve 
\begin{equation} \label{sbauxeq3} 
    \begin{aligned}
           p_+^2  & = \frac{\lad_- - \lad_+}{\lad_+ A_{++} - 3 \lad_- A_{+-}} + \frac{\lad_- A_{--} - 3 \lad_+ A_{+-}}{\lad_+ A_{++} - 3 \lad_- A_{+-}}  p_-^2  + \Or(\eps + p_-^4 + p_{+,*}^4) \\
    & = p_{+,*}^2 + \frac{\lad_- A_{--} - 3 \lad_+ A_{+-}}{\lad_+ A_{++} - 3 \lad_- A_{+-}}  p_-^2  + \Or(\eps + p_-^4 + p_{+,*}^4)\,,
    \end{aligned}
\end{equation}
using $\frac{\lad_- - \lad_+}{\lad_+ A_{++} - 3 \lad_- A_{+-}} =  p_{+,*}^2 + \Or( p_{+,*}^4)$ by \eqref{eq:pstar}. Plugging \eqref{sbauxeq3} into \eqref{sbauxeq1} gives 
\begin{equation} \label{sbauxeq4}
     \begin{aligned} 
    \h{\ww}^2 & = \frac{1}{\lad_+ + \lad_+ \left(p_+^2 A_{++} + 3 p_-^2 A_{+-}\right) + \Or(\eps + p_-^4 + p_{+,*}^4)} \\ 
    & = \frac{1}{\lad_+ (1 + p_{+,*}^2 A_{++}) + \lad_+ \left(\frac{\lad_- A_{--} - 3 \lad_+ A_{+-}}{\lad_+ A_{++} - 3 \lad_- A_{+-}} A_{++} + 3 A_{+-}\right)  p_-^2 + \Or(\eps + p_-^4 + p_{+,*}^4)} \\
    & = \h{\ww}_*^2 -  \lad_+ \h{\ww}_*^4 \left(\frac{\lad_- A_{--} - 3 \lad_+ A_{+-}}{\lad_+ A_{++} - 3 \lad_- A_{+-}} A_{++} + 3 A_{+-}\right)  p_-^2 +  \Or(\eps + p_-^4 + p_{+,*}^4)\,, 
\end{aligned} 
\end{equation}
by using $1/\h{\ww}_*^2 = \lad_+ (1 + p_{+,*}^2 A_{++}) + \Or(p_{+, *}^4)$ from \eqref{eq:wstar}.

\medskip

\noindent
\underline{\emph{Case II: $\Delta \theta = \{\frac{\pi}{2},\frac{3 \pi}{2}\}$}}. Recalling \eqref{eq:opi}, without loss of generality, we consider only $\Delta \theta = \frac{\pi}{2}$. Then, as in Case I, if there is a second bifurcation on the symmetric branch, the following equation from \eqref{eqbipt} with $\Delta \theta = \frac{\pi}{2}$ should admit a solution: 
\begin{align*}
    & \h{\ww}^{-2} - \lad_+ - \lad_+ p_+^2 A_{++} = \Or(p_+^4)\,, \\
    & \h{\ww}^{-2} - \lad_-  - \lad_- p_+^2 A_{+-} = \Or(p_+^4)\,,
\end{align*}
which implies 
\begin{align} \label{eq:bifruc2}
    \lad_+ - \lad_- + \left(\lad_+ A_{++}  - \lad_-   A_{+-}\right)  p_+^2 + \Or(p_+^4) = 0\,.
\end{align}
However, recalling $\lad_+ > \lad_-$, and the condition $\lad_+ A_{++}  - \lad_-   A_{+-} > 0$ from \cref{asspcob}, the existence of a small positive solution $p_+$ to \eqref{eq:bifruc2} is precluded. 

\smallskip

Combining all the arguments above, we have proved the following main result. 

\begin{theorem} \label{thm:2ndbirfuc}
    Under \cref{assp:second} and \cref{asspcob}, for the scaled nonlinear resonance problem \eqref{eq:nonlineardimer} with normalization \eqref{normalconst}, equivalently, the system \eqref{eq:leading1}--\eqref{eq:leading3} with \eqref{eq:constcoeff}, then 
    \begin{itemize}
        \item When $|c_+| + |c_-|$ is sufficiently small, there exist only two families of nontrivial solutions, given in \cref{sol:purestate}, with the corresponding resonant states being symmetric and antisymmetric, respectively. 
        \item Let $c_\pm = p_{\pm} e^{i\theta_\pm}$ with $\Delta \theta = \theta_+ - \theta_-$. For any fixed $\eps$, on the symmetric branch \eqref{puresol1}, there exists a critical $|c_{+,{\rm crit}}(\eps)| = p_{+,{\rm crit}}(\eps) > 0$ such that 
        at $c_+ = c_{+,{\rm crit}}(\eps)$, there is a symmetry-breaking bifurcation, parameterized by small $(p_-,\eps)$, up to $u \to e^{i \theta} u$, 
        \begin{align} \label{eq2ndbsta}
            (p_-, \eps) \to & (\h{\ww},u) = \left(\h{\ww}(p_-, \eps), c_+ \vp_+ + p_- \vp_- + r(\h{\ww}, c_+, p_-, \eps)\right)\,,
        \end{align}
          with $r(\dd)$ from \cref{prop:implicitr}, and
    \begin{align} \label{eq2ndbsta2}
        c_+ = p_+(p_-,\eps) e^{i \Delta\theta(p_-,\eps) + i \theta_0}\,,\q \theta_0 = 0, \pi\,,
    \end{align}
    where the asymptotic expansions of $\Delta \theta$, $p_+$, and $\h{\ww}$ are given in \eqref{eq:expdelta}, \eqref{sbauxeq3}, and \eqref{sbauxeq4}, respectively. In particular, we have, by letting $p_- \to 0$ in \eqref{sbauxeq3},  
    \begin{align*}
      p_{+,{\rm crit}}(\eps)^2 =   p_{+,*}^2  + \Or(\eps + p_{+,*}^4)\,,
    \end{align*}
    with $p_{+,*}$ given in \eqref{eq:pstar2}. 
    \item There is no symmetry-breaking bifurcation on the antisymmetric solution branch. 
    \end{itemize}
\end{theorem}

\begin{remark}
According to \eqref{eq2ndbsta} and \eqref{eq2ndbsta2}, for a nonlinear resonance $\ww$ associated with the symmetry-breaking bifurcation, there exist two corresponding resonant states (up to  $S^1$ symmetry):  
\begin{align} \label{eq:asymsta}
    u = p_+(p_-, \eps) e^{i \Delta \theta(p_-, \eps)} \vp_+ \pm p_- \vp_- + r(\h{\ww}, p_+(p_-, \eps) e^{i \Delta \theta(p_-, \eps)}, \pm p_-, \eps)\,.
\end{align}
Here, the $\pm$ sign indicates that these asymmetric states tend to localize in one of the dielectric particles of the symmetric dimer.
\end{remark}

The variables $(\eps, \h{\ww})$ can easily be transformed back to $(\tau, \ww)$ using \eqref{eq:scalefreq}. Furthermore, similar to \cref{rem:invernorcont}, it follows from \eqref{normaleq} that the mapping from $|c_+| + |c_-|$ to $\mc{N}$ is invertible in the high contrast regime (\textit{i.e.}, when $\tau = 1 / \eps^2$ is large). Consequently, the nonlinear resonance $\ww$ and the state  $u$ can also be parameterized in terms of the normalization constant $\mc{N}$. \cref{thm:2ndbirfuc} demonstrates that when $\mc{N}$ is sufficiently small, there exist only two solution branches, with symmetric and antisymmetric states, bifurcating from the zero solution of the linear resonances $\ww_\pm \sim \sqrt{\tau \lad_\pm}$, respectively. As $\mathcal{N}$ increases beyond a critical threshold $\mathcal{N}_{\rm crit}$, a secondary bifurcation emerges on the symmetric branch, giving rise to two asymmetric resonant states \eqref{eq:asymsta}, where the critical threshold is given by
\begin{align*}
\mathcal{N}_{\rm crit} \sim \frac{\lambda_+ - \lambda_-}{3 \lambda_- A_{+-} - \lambda_+ A_{++}}\,.
\end{align*}

\subsection{Two-dimensional case}

In this section, we show that in the two-dimensional case, under an assumption similar to \cref{assp:second}, such a symmetry-breaking bifurcation does not occur on either the small-amplitude symmetric or antisymmetric solution branches. This result arises from the differing scalings of the principal subwavelength resonance $\ww = \Or(\sqrt{1/\tau \ln \tau})$ and other resonances $\ww = \Or(1)$, as proved in \cref{prop:limit_scalar2d}.

\begin{assumption} \label{asspsnd2c}
    Let $D = D_1 \cup D_2 \subset \R^2$ be a symmetric dimer and let $\w{\kk}_D$ be the operator in \eqref{eq:limopscal}. The largest eigenvalue $\lad_-$ of $\w{\kk}_D$ is simple with the antisymmetric normalized eigenfunction $\vp_-$. 
\end{assumption}

As noted in \cref{rem:krthm3d}, the Krein-Rutman Theorem does not guarantee the simplicity of the principal eigenvalue of $\w{\kk}_D$ or the symmetry of its associated eigenfunction. Verification of \cref{asspsnd2c} in specific cases remains an open question for future research. However, we expect that the simplicity of $\lad_-$ holds in a generic sense, as suggested in \cite{teytel1999rare}.
The following lemma provides the resolvent estimate for $\w{\kk}_D$, analogous to \cref{lem:resolv}.

\begin{lemma} \label{lem:resol2d}
Let $\h{\ww} \in \mathbb{C}$ satisfy $\mathrm{dist}(\h{\ww}^{-2}, \sigma(\w{\kk}_D)\backslash\{\lambda_{-}\}) > 0$, and define  
\begin{align} \label{def:omegafre2}
    \widetilde{\Omega} := \{\h{\ww} \in \mathbb{C}\,;\ \mathrm{dist}(\h{\ww}^{-2}, \sigma(\w{\kk}_D)\backslash\{\lambda_{-}\}) > \w{d}_*\},
\end{align}
where $\w{d}_* := \tfrac{1}{2}\mathrm{dist}(\lambda_{-}, \sigma(\w{\kk}_D)\backslash\{\lambda_-\})$. 
Then, for all $\h{\ww} \in \widetilde{\Omega}$ and any $f \in H^2(D) \cap L^2_0(D)$ with $f \perp \varphi_-$, the solution $u = (1 - \h{\ww}^2 \w{\kk}_D)^{-1}[f]$ satisfies the uniform estimate
\begin{equation} \label{eq:resolkd2}
    \|u\|_{H^2(D)} \leq C_{\w{d}_*} \|f\|_{H^2(D)},
\end{equation}
where the constant $C_{\w{d}_*} > 0$ depends only on $\w{d}_*$.
\end{lemma}

\begin{theorem} \label{thm:2d2ndbif}
    Under \cref{asspsnd2c}, consider the nonlinear resonance problem \eqref{eq:nonLippscalar0} with the normalization condition \eqref{eq:normalizecond}. When $\mc{N}$ is sufficiently small, there exist only two nontrivial solution families, with symmetric and antisymmetric resonant states $u_\pm$, bifurcating from the zero solution of linear subwavelength resonances $$\ww_+ \sim \sqrt{4 \pi /(|D| \tau \ln \tau)} \q  \text{and} \q \ww_- \sim \sqrt{1/\tau \lad_-},$$ respectively. Furthermore, no symmetry-breaking bifurcations occur along either branch as $\mc{N}$ increases. 
\end{theorem}

\begin{remark}
The nonlinear subwavelength resonances $\omega_\pm$ and associated states $u_\pm$ admit asymptotic expansions analogous to those derived in \cref{thm:2dexist}. For brevity, we omit the explicit statements in \cref{thm:2d2ndbif}.  While \cref{thm:2d2ndbif} rules out symmetry-breaking bifurcations along the principal resonance branches, such bifurcations may still arise through hybridization between other subwavelength resonances $\omega \sim \sqrt{\tau \mu_j}$, where $\{\mu_j\}$ are the eigenvalues of $\widetilde{\mathcal{K}}_D$.
\end{remark}

\begin{proof}
The proof closely follows the approach in \cref{subsec:3dcase} for the three-dimensional case (also the one for \emph{Case II} in the proof of \cref{thm:2dexist}); thus, we only outline the main steps below. 

It suffices to consider the scaled nonlinear resonance problem \eqref{eq:nonlineardimer} with \eqref{normalconst}, where $D \subset \R^2$ is a symmetric dimer. For convenience, let $\varphi_+ = 1$, and define the projections $P_+[u] := \langle 1_D, u \rangle_D$, $P_-[u] := \langle \varphi_-, u \rangle_D \varphi_-$, and $P_r := I - P_+ - P_-$. We adopt a solution ansatz similar to \eqref{eq:solansa}–\eqref{anszsol2}:  
\begin{align*}
    u = c_+ \varphi_+ + c_- \varphi_- + r = c_+ + c_- \varphi_- + r\,,
\end{align*}  
where $c_+ = \langle 1_D, u \rangle_D$,  $c_- = \langle \varphi_-, u \rangle_D$, and $r = P_r u$. Noting that $\vp_{\pm}$ are not eigenfunctions of $\kk_D$, we modify the definition of the constants $a_{klmn}$ in \eqref{def:constakl} as follows: 
\begin{align} \label{defcontkd}
  a_{+lmn} := \l 1_D, \kk_D[\vp_l \vp_m \vp_n] \r_D \in \R\,,\q  a_{-lmn} := \l \vp_-, \kk_D[\vp_l \vp_m \vp_n] \r_D \in \R\,.
\end{align}
Similarly, with the help of the symmetry of $\vp_\pm$, we still have $a_{klmn} = 0$ if there is one $-$ or three $-$ in $(k,l,m,n)$. It follows that 
\begin{align} \label{def:coeffA2}
a_{klmn} = 
\begin{dcases}
    A_{++}  &\q \text{if $(k,l,m,n) = (+,+, +, +)$},\\
    A_{+-} &\q \text{if $k = +$, and there is one $+$ in $(l,m,n)$},\\
    A_{-+} &\q \text{if $k = -$, and there are two $+$ in $(l,m,n)$},\\
    A_{--} &\q \text{if $(k,l,m,n) = (-,-, -, -)$}.
\end{dcases}  
\end{align}
We recall 
\begin{align*}
    R(c_+,c_-,r) = \kk_{D}\left[|u|^2 u \right] - \kk_D\left[|c_+ \vp_+ + c_- \vp_-|^2 (c_+ \vp_+ + c_- \vp_-)\right].
\end{align*}
We also define $a_{klmn}^{\eps \h{\ww}}$, $A_{\pm\pm}^{\eps \h{\ww}}$, and $R^{\eps \h{\ww}}(c_+,c_-,r)$ by replacing the operator $\kk_D$ with $\kk_D^{\eps \h{\ww}}$. In particular, a direct computation gives, by $\norm{\vp_-}_{L^2(D)} = 1$,
\begin{align} \label{auxexplifor}
    A_{++}^{\eps \h{\ww}} = \l 1_D, \kk_D[1] \r_D - \eta_{\eps \h{\ww}} |D|\,,\q A_{+-}^{\eps \h{\ww}} = \l 1_D, \kk_D[\vp_-^2] \r_D - \eta_{\eps \h{\ww}}\,.
\end{align}

Using the projections $P_{\pm}$ and $P_r$, along with the preparations above, similarly to \eqref{eq:leading1}--\eqref{coeff2}, we can reformulate equation \eqref{eq:nonlineardimer} into the following system: 
\begin{multline} \label{eq:proj031}
    (1 + \h{\ww}^2 \eta_{\eps \h{\ww}} |D|) c_+ = \h{\ww}^2 \l 1_D, \kk_D[c_+ + r] \r_D + \h{\ww}^2|c_+|^2 c_+ A^{\eps \h{\ww}}_{+ +} \\ + \h{\ww}^2(2|c_-|^2 c_+   + c_-^2 \overline{c_+}) A^{\eps \h{\ww}}_{+ -}  + \h{\ww}^2 \l 1_D, R^{\eps \h{\ww}}(c_+,c_-,r)\r_D + \h{\ww}^2 \l 1_D, \rr^{\eps \h{\ww}}[N(u)] \r_D\,,
\end{multline}
\begin{multline} \label{eq:proj041}
   (1 - \h{\ww}^2 \lad_-)  c_-  = \h{\ww}^2 |c_-|^2 c_- A_{--} + \h{\ww}^2\left(2 |c_+|^2 c_- + c_+^2 \overline{c_-}  \right) A_{-+}  \\ + \h{\ww}^2 \l \vp_-, R(c_+,c_-,r)\r_D  + \h{\ww}^2 \l \vp_-, \rr^{\eps \h{\ww}}[N(u)] \r_D\,,
\end{multline}
\begin{align} \label{eq:proj151}
   (1 - \h{\ww}^2 \w{\kk}_D) r = \h{\ww}^2 P_r \kk_D \left[ c_ +\right] + \h{\ww}^2 P_r \kk_D \left[ |u|^2 u \right] + \h{\ww}^2 P_r \rr^{\eps \h{\ww}}[N(u)]\,,
\end{align}
where we have used, from the antisymmetry of $\vp_-$ and $\vp_- \perp r$, 
\begin{align*}
    \l \vp_-, \kk_D[u] \r_D = c_- \lad_-\,,\q \l 1_D, \kk_D[u] \r_D =  \l 1_D, \kk_D[c_+ + r] \r_D\,.
\end{align*}

Then, similarly to the proof of \cref{thm:2dexist}, we solve $r$ from \eqref{eq:proj151} via
\begin{multline*}
    \mc{F}(\h{\ww},c_+,c_-,r,s) := r - c_+ (\h{\ww}^{-2} - \w{\kk}_D)^{-1} P_r \w{\phi} \\ - (\h{\ww}^{-2} - \w{\kk}_D)^{-1} P_r \kk_D[|u|^2 u] - (\h{\ww}^{-2} - \w{\kk}_D)^{-1} P_r \rr^{\eps \h{\ww}}[N(u)]= 0\,,
\end{multline*}
with 
\begin{align*}
    \h{\ww} \in \w{\Omega}\,,\q \eps = e^{-\frac{1}{s^2}}\,,
\end{align*}
and $\w{\phi}$, $\w{\kk}_D$ in \eqref{auxfunc2d}--\eqref{eq:limopscal}. It is easy to see $\mc{F}(\h{\ww},0,0,0,s) = 0$ for any $\h{\ww} \in \w{\Omega}$ and small $s$. From \cref{lem:resol2d}, we estimate 
\begin{align*}
    \norm{\p_r \mc{F} - I}_{H^2(D) \to H^2(D)} \le C_{\w{d}_*} \left( \norm{u}_{H^2}^2 + (\eps \h{\ww})^2 (|\ln \eps| + \ln \h{\ww}) \right),
\end{align*}
for a constant $C_{\w{d}_*}$ depending on $\w{d}_*$. Moreover, by an argument similar to that in \eqref{auxestr}, the following a priori estimate holds for the solution of $\mathcal{F}(\h{\ww}, c_+, c_-, r, s) = 0$:  
\begin{align} \label{esttr}
    \norm{r}_{H^2(D)} \le C_{\w{d}_*} \left( |c_+| + (|c_+|^2 + |c_-|^2 + (\eps \h{\ww})^2 (|\ln \eps| + \ln \h{\ww})) (|c_+| + |c_-|) \right),
\end{align}
when $\norm{u}_{H^2(D)}$ and $\eps$ are sufficiently small. Therefore, as \cref{prop:implicitr}, there exist constants $c_* > 0$ and $s_* > 0$  depending on $\w{d}_*$ such that for $\h{\ww} \in \w{\Omega}$, $|c_+| + |c_-| < c_*$, and $s < s_*$, one can uniquely solve function $r(\dd)$ by $\mc{F}(\h{\ww},c_+,c_-,r,s) = 0$, satisfying the estimate \eqref{esttr} and the properties in \cref{coro:properr}. In particular, there holds 
\begin{align} \label{eq:symmrfun2}
    \rr[r(\h{\ww}, c_+, 0, s)] = r(\h{\ww}, c_+, 0, s)\,,\q \rr[r(\h{\ww}, 0, c_-, s)] = - r(\h{\ww}, 0, c_-, s)\,.
\end{align}

By the symmetry \eqref{eq:symmrfun2}, we have 
$$\l \vp_{-}, R(c_+,c_-,r)|_{c_- = 0} \r_D = 0 \q  \text{and} \q \l 1_D, R^{\eps \h{\ww}}(c_+,c_-,r)|_{c_+ = 0} \r_D = 0.$$ Thus, there exist functions $f_\pm$ such that 
\begin{equation} \label{auxeq112d}
    \begin{aligned} 
    &\l \vp_{-}, R(c_+,c_-,r) \r_D = c_- f_-(\h{\ww},c_+,c_-,s)\,, \\ &\l 1_D, R(c_+,c_-,r) \r_D = c_+ f_+(\h{\ww},c_+,c_-,s)\,.
\end{aligned}
\end{equation}
For the same reason, we can write 
\begin{align} \label{auxeq122d}
     \l \vp_-, \rr^{\eps \h{\ww}}[N(u)] \r_D = c_- g_-(\h{\ww},c_+,c_-,s)\,, \q \l 1_D, \rr^{\eps \h{\ww}}[N(u)] \r_D = c_+ g_+(\h{\ww},c_+,c_-,s)\,,
\end{align}
and 
\begin{align} \label{auxeq132d}
     \l 1_D, \kk_D[c_+ + r] \r_D = c_+ k(\h{\ww},c_+,c_-,s)\,,
\end{align}
for some functions $g_\pm$ and $k$. 

With  $r$  solved, we next focus on equations \eqref{eq:proj031}–\eqref{eq:proj041}, along with \eqref{auxeq112d}–\eqref{auxeq132d}, in the variables $(\h{\ww}, c_+, c_-, s)$. By the same arguments as those in \cref{subsec:2dsolution} and \cref{subsec:3dcase}, the existence of pure mode solutions can be established. Briefly, note that if $c_- = 0$, then \eqref{eq:proj041} is satisfied trivially, allowing  $\h{\ww}(c_+, s)$ to be solved from \eqref{eq:proj031}, with the corresponding $u = c_+ + r(\h{\ww}, c_+, 0, s)$ being symmetric. Similarly, by setting $c_+ = 0$, \eqref{eq:proj031} is satisfied trivially, enabling $\h{\ww}(c_-, s)$ to be solved from \eqref{eq:proj041} with antisymmetric $u = c_- \vp_- + r(\h{\ww}, 0, c_-, s)$. 

We will now prove, by contradiction, the non-existence of symmetry-breaking bifurcations along the pure symmetric and antisymmetric solution branches discussed above. Suppose that such a second bifurcation exists, for which the solution necessarily satisfies $c_\pm \neq 0$. Then, equations \eqref{eq:proj031} and \eqref{eq:proj041} reduce to the following forms, by dividing by $c_+$ and $c_-$, respectively, 
\begin{multline} \label{eq:proj0312}
    1 + \h{\ww}^2 \eta_{\eps \h{\ww}} |D| = \h{\ww}^2 k(\h{\ww},c_+,c_-,s) + \h{\ww}^2|c_+|^2 A^{\eps \h{\ww}}_{+ +} + \h{\ww}^2(2|c_-|^2   + c_-^2 e^{-i 2 \theta_+} ) A^{\eps \h{\ww}}_{+ -}  \\ + \h{\ww}^2 f_+(\h{\ww},c_+,c_-,s) + \h{\ww}^2 g_+(\h{\ww},c_+,c_-,s)\,,
\end{multline}
\begin{multline} \label{eq:proj0412}
   1 - \h{\ww}^2 \lad_-  = \h{\ww}^2 |c_-|^2 A_{--} + \h{\ww}^2\left(2 |c_+|^2 + c_+^2 e^{-i 2 \theta_-}  \right) A_{-+}  \\ + \h{\ww}^2 f_-(\h{\ww},c_+,c_-,s) + \h{\ww}^2 g_-(\h{\ww},c_+,c_-,s)\,,
\end{multline}
where $\theta_\pm$ are the phases from the polar form $c_\pm = p_\pm e^{i \theta_\pm}$. We consider two cases. 
If the symmetry-breaking bifurcation occurs along the symmetric curve, then the bifurcation point $(\h{\ww}, c_+, c_-, s)$, solving equations \eqref{eq:proj0312} and \eqref{eq:proj0412}, must satisfy $\h{\ww} = \Or(\sqrt{-\ln \eps}^{-1})$ and $c_- = 0$. Substituting $c_- = 0$ into \eqref{eq:proj0412} leads to the following:
\begin{align*}
    1 + \Or\left(\frac{1}{\ln \eps}\right) = \Or\left(\frac{1}{\ln \eps}\right),
\end{align*}
which is a contradiction as $\eps \to 0$. On the other hand, if the bifurcation occurs on the antisymmetric curve, then the associated bifurcation point $(\h{\ww}, c_+, c_-, s)$, solving equations \eqref{eq:proj0312} and \eqref{eq:proj0412}, must satisfy $\h{\ww} = \Or(1)$ and $c_+ = 0$, which further implies $\eta_{\eps \h{\ww}} = \Or(\ln \eps) $. Matching the terms of order $\Or(\ln \eps)$ in \eqref{eq:proj0312} gives, by using \eqref{auxexplifor},
\begin{align*}
    -|D| = 2|c_-|^2 + c_-^2 e^{-i 2 \theta_+} = |c_-|^2 \big( 2 + e^{-i 2 (\theta_+ - \theta_-)} \big)\,.
\end{align*}
Taking the real part of this equation leads to a contradiction, as the left-hand side is negative, while the right-hand side is positive.  Thus, the proof is complete.  
\end{proof} 

\section{Concluding remarks}

This work represents a first step toward developing a comprehensive mathematical framework for analyzing nonlinear scattering resonances in dielectric resonators with high refractive indices and a Kerr-type nonlinearity. 
We have established the existence of nonlinear subwavelength dielectric resonances with small-amplitude states in both two and three dimensions, and derived their asymptotic formulas in terms of the contrast and normalization constant. For symmetric configurations, these small-amplitude nonlinear resonant states would inherit the corresponding symmetry properties. We have investigated symmetry-breaking bifurcations for a symmetric dimer, revealing a fundamental distinction between the two- and three- dimensional settings. In the three-dimensional case, we have identified conditions under which a secondary bifurcation occurs on the principal symmetric branch, giving rise to two asymmetric resonant states. However, in the two-dimensional case, the differing scalings of the principal resonance, $\Or(1/\sqrt{\tau \ln \tau})$ versus $\Or(1)$, prevent such bifurcations from occurring.

Our results open up numerous possibilities for advancing the mathematical understanding of nonlinear high-index dielectric resonators and associated photonic metamaterials. Here, we highlight some potential directions for future research.  First, this work focuses on small-amplitude resonant states, corresponding to the weak nonlinearity regime under a scaling. A natural extension would be to investigate the existence of large-amplitude resonant states, particularly symmetry-breaking bifurcations at a large $\mc{N}$, analogous to the results in \cites{jackson2004geometric,kirr2011symmetry} for the nonlinear Schr\"{o}dinger equation. Moreover, as explored in \cites{feppon2022modal,ammari2022modal,li2024large}, an important follow-up question is the modal decomposition of the nonlinear scattered wave by high-index dielectric resonators in both the frequency and time domains. This could lay the foundation for addressing more challenging problems, such as the long-time dynamics of scattered waves and the stability of nonlinear principal resonant states during time evolution.  
Another compelling direction involves studying the localization properties of nonlinear resonant states in various configurations, including periodic lattices of nonlinear dielectric resonators with linear or nonlinear defects, as well as random configurations of particles, as discussed in \cites{ilan2010band,ammari2024anderson,ammari2025disorder}. Finally, extensions to larger resonator networks and their effective-medium or graph-limit descriptions could facilitate the systematic design of nonlinear subwavelength photonic and acoustic metamaterials.

\subsection*{Acknowledgments} This work was initiated while H.A. was visiting the Hong Kong Institute for Advanced Study as a Senior Fellow. It was supported in part by National Key R$\&$D Program of China Grant No. 2024YFA1016000 (B.L.), and by the Swiss National Science Foundation grant number 2025-10004276 (H.A.). 

\begin{appendix}
\section{Auxiliary proofs} \label{app:A}  


\begin{proof}[Proof of \cref{prop:veriassp}]
We introduce the operators, for $i \neq j$, 
\begin{align*}
    \mc{R}_{D_i,D_j}[\vp] = \int_{D_i} G^0(x-y)\vp(y) \ud y:\ L^2(D_i) \to  L^2(D_j)\,,
\end{align*}
where $G^0(x-y) = \frac{1}{4 \pi |x - y|}$ is the Green's function for the Laplace operator. Moreover, there holds 
\begin{align*}
    \mc{R}_{D_i,D_j} = \Or(L^{-1})\,.
\end{align*}
Then we can reformulate the eigenvalue problem $(\lad - \kk_D)[\vp] = 0$ as: 
\begin{align*}
    \mm \lad - \kk_{D_1} &  \rr_{D_2,D_1} \\
    \rr_{D_1, D_2} & \lad - \kk_{D_2} \nn \mm \vp|_{D_1} \\ \vp|_{D_2} \nn = \mm \lad - \kk_{D_1} & 0 \\
    0 & \lad - \kk_{D_2} \nn \mm \vp|_{D_1} \\ \vp|_{D_2} \nn + \Or(L^{-1}) = 0\,.
\end{align*}
It is clear that as $L \to 0$, the leading-order eigenvalue problem is non-interacting, and it admits the principal eigenvalue $\lad_0$ with two normalized $L^2(D)$ eigenfunctions
\begin{align*}
    \vp_1 = \mc{T}_L \vp_0\,,\q  \vp_2 = \rr \mc{T}_L \vp_0\,.
\end{align*}
By the perturbation theory \cite{kato2013perturbation}, to find the first-order approximation of $\lad$, we consider the following two-dimensional secular equation:
\begin{align} \label{leadeig}
     (\lad - M) \mm a \\ b\nn = 0\,,
\end{align}
where 
\begin{align*}
    M: = \mm \l \vp_1, \kk_D[\vp_1] \r_D & \l \vp_1, \kk_D[\vp_2] \r_D \\ \l \vp_2, \kk_D[\vp_1] \r_D   & \l \vp_2, \kk_D[\vp_2] \r_D \nn = \mm \lad_0  & k_I \\ k_I  & \lad_0 \nn ,
\end{align*}
with $$ 0 < k_I := \l \vp_1, \kk_D[\vp_2] \r_D = \l \vp_2, \kk_D[\vp_1] \r_D = \Or(L^{-1})\,,$$ by symmetry and $\vp_0 > 0$. It follows that for large $L$, the eigenvalue $\lad_0$ of multiplicity two split into two simple eigenvalues $\lad_+ > \lad_-$ with asymptotics: 
\begin{align*}
    \lad_\pm = \lad_0 \pm k_I + \Or(L^{-2})\,.
\end{align*}
Moreover, note that $\lad_0 \pm k_I$ are eigenvalues of \eqref{leadeig} with corresponding eigenvectors $\frac{1}{\sqrt{2}}[1,\pm 1]$. This implies that normalized eigenfunctions $\vp_\pm$ associated with $\lad_\pm$ satisfy \eqref{eq:coneigfunc}. Since the eigenvalues $\lad_\pm$ are simple, we can conclude from \eqref{eq:coneigfunc} that $\vp_\pm$ are symmetric and antisymmetric, respectively. The proof is complete.
\end{proof}



\begin{proof}[Proof of \cref{prop:nonvanish}]
It suffices to show that for any frequency $\ww$ with $\Im \ww \ge 0$, the nonlinear equation \eqref{eq:nonLippscalar} admits only the trivial solution $u \equiv 0$. The symmetry with respect to the imaginary axis follows from \cref{lem:invarphaseu}.

For simplicity, we let $\eta = 1$ and $V(x): = \tau \chi_D (1 + |u|^2)$. Let us start with the case of $\ww \in \R \backslash \{0\}$. By \cref{lem:invarphaseu},  we only consider $\ww > 0$. From the radiation condition \eqref{eq:radiation}, we observe, as $r \to \infty$, 
\begin{equation}
    \int_{\p B_r(0)}  \left| \frac{\partial u}{\partial \nu} \right|^2 + \ww^2 |u|^2 + 2 \ww \Im \left( u \, \frac{\partial \bar{u}}{\partial \nu} \right) \ud s = \int_{\p B_r(0)} \left| \frac{\partial u}{\partial \nu} - i \ww u \right|^2 \ud s \to 0\,.
\label{eq:surface_int}
\end{equation}
An integration by parts on $B_r(0) \backslash D$ gives 
\begin{equation}
\int_{\p B_r(0)} u \, \frac{\partial \bar{u}}{\partial \nu} \ud s - \int_{\partial D} u \, \frac{\partial \bar{u}}{\partial \nu} \ud s = \int_{B_r(0) \backslash D} |\nabla u|^2 - \ww^2 |u|^2 \ud x\,,
\label{eq:green_app}
\end{equation}
by $\int_{B_r(0) \backslash D} u (\Delta \bar{u} + \ww^2 \bar{u})  \ud x = 0$. It follows that, by taking the imaginary part, 
\begin{align*}
   \Im \int_{\p B_r(0)} u \, \frac{\partial \bar{u}}{\partial \nu} \ud s = \Im \int_{\partial D} u \, \frac{\partial \bar{u}}{\partial \nu} \ud s\,,
\end{align*}
which further gives, by \eqref{eq:surface_int}, 
\begin{align} \label{app:auxeq1}
     \lim_{r \to \infty}   \int_{\p B_r(0)}  \left| \frac{\partial u}{\partial \nu} \right|^2 + \ww^2 |u|^2 \ud s  =  - 2 \ww \Im \int_{\partial D} u \, \frac{\partial \bar{u}}{\partial \nu} \ud s \ge 0\,.
\end{align}
Now, an integration by parts inside the domain $D$ leads to
\begin{align} \label{app:auxeq2}
    \int_{\p D} u \, \frac{\partial \bar{u}}{\partial \nu} \ud s = \int_{D} |\nabla u|^2 - \ww^2 V |u|^2 \ud x\,,
\end{align}
implying that 
\begin{align} \label{app:auxeq4}
    \Im \int_{\p D} u \, \frac{\partial \bar{u}}{\partial \nu} \ud s = 0\,,
\end{align}
since $V$ is real-valued. Hence, by \eqref{app:auxeq1}, we conclude that
\begin{align*}
     \lim_{r \to \infty}   \int_{\p B_r(0)} |u|^2 \ud s = 0\,.
\end{align*}
Then, Rellich’s lemma \cite{colton1998inverse} gives $u = 0$ on $\R^d \backslash D$. Applying the strong unique continuation property \cite{jerison1985unique}*{Theorem 6.3} to the equation $ - \Delta u - \ww^2 (1 + V) u = 0$ shows that $u \equiv 0$ on $\R^d$. 

The discussion for the case of $\Im \ww > 0$ is similar. By \cref{lem:invarphaseu}, we further let $\Re \ww \ge 0$. Noting that $u$ exponentially decays as $|x| \to \infty$, the formula \eqref{eq:green_app} implies 
\begin{align} \label{app:auxeq3}
     - \int_{\partial D} u \, \frac{\partial \bar{u}}{\partial \nu} \ud s = \int_{\R^d \backslash D} |\nabla u|^2 - \ww^2 |u|^2 \ud x\,.
\end{align}
Taking the imaginary parts of \eqref{app:auxeq2} and \eqref{app:auxeq3}, we obtain \eqref{app:auxeq4}, which enables us to reach the same conclusion as in the previous case. This completes the proof.
\end{proof}

\end{appendix}



\begin{bibdiv}
\begin{biblist}

\bib{ammari2025disorder}{article}{
      author={Ammari, Habib},
      author={Barandun, Silvio},
      author={Uhlmann, Alexander},
       title={Subwavelength localization in disordered systems},
        date={2025},
     journal={Proceedings of the Royal Society A},
      volume={481},
       pages={Paper No. 20250407},
}

\bib{ammari2019subwavelength}{article}{
      author={Ammari, Habib},
      author={Dabrowski, Alexander},
      author={Fitzpatrick, Brian},
      author={Millien, Pierre},
      author={Sini, Mourad},
       title={Subwavelength resonant dielectric nanoparticles with high refractive indices},
        date={2019},
     journal={Mathematical Methods in the Applied Sciences},
      volume={42},
      number={18},
       pages={6567\ndash 6579},
}

\bib{ammari2024anderson}{article}{
      author={Ammari, Habib},
      author={Davies, Bryn},
      author={Hiltunen, Erik~Orvehed},
       title={Anderson localization in the subwavelength regime},
        date={2024},
     journal={Communications in Mathematical Physics},
      volume={405},
      number={1},
       pages={Paper No. 1, 20},
}

\bib{aschbacher2002symmetry}{article}{
      author={Aschbacher, WH},
      author={Fr{\"o}hlich, J},
      author={Graf, GM},
      author={Schnee, K},
      author={Troyer, M},
       title={Symmetry breaking regime in the nonlinear hartree equation},
        date={2002},
     journal={Journal of Mathematical Physics},
      volume={43},
      number={8},
       pages={3879\ndash 3891},
}

\bib{arbabi2015dielectric}{article}{
      author={Arbabi, Amir},
      author={Horie, Yu},
      author={Bagheri, Mahmood},
      author={Faraon, Andrei},
       title={Dielectric metasurfaces for complete control of phase and polarization with subwavelength spatial resolution and high transmission},
        date={2015},
     journal={Nature nanotechnology},
      volume={10},
      number={11},
       pages={937\ndash 943},
}

\bib{ammari2025nonlinear}{article}{
      author={Ammari, Habib},
      author={Kosche, Thea},
       title={Nonlinear subwavelength resonances in three dimensions},
        date={2025},
     journal={Studies in Applied Mathematics},
      volume={154},
      number={3},
       pages={e70036},
}

\bib{ammari2024fano}{article}{
      author={Ammari, Habib},
      author={Li, Bowen},
      author={Li, Hongjie},
      author={Zou, Jun},
       title={Fano resonances in all-dielectric electromagnetic metasurfaces},
        date={2024},
     journal={Multiscale Modeling \& Simulation},
      volume={22},
      number={1},
       pages={476\ndash 526},
}

\bib{ammari2023mathematical}{article}{
      author={Ammari, Habib},
      author={Li, Bowen},
      author={Zou, Jun},
       title={Mathematical analysis of electromagnetic scattering by dielectric nanoparticles with high refractive indices},
        date={2023},
     journal={Transactions of the American Mathematical Society},
      volume={376},
      number={01},
       pages={39\ndash 90},
}

\bib{ammari2020superresolution}{article}{
      author={Ammari, Habib},
      author={Li, Bowen},
      author={Zou, Jun},
       title={Superresolution in Recovering Embedded Electromagnetic Sources in High Contrast Media},
        date={2020},
     journal={SIAM Journal on Imaging Sciences},
      volume={13},
    number={3},
    pages={1467\ndash1510},
}



\bib{ammari2022modal}{article}{
      author={Ammari, Habib},
      author={Millien, Pierre},
      author={Vanel, Alice~L},
       title={Modal approximation for strictly convex plasmonic resonators in the time domain: The maxwell's equations},
        date={2022},
     journal={Journal of Differential Equations},
      volume={309},
       pages={676\ndash 703},
}

\bib{ambrosetti1973dual}{article}{
      author={Ambrosetti, Antonio},
      author={Rabinowitz, Paul~H},
       title={Dual variational methods in critical point theory and applications},
        date={1973},
     journal={Journal of functional Analysis},
      volume={14},
      number={4},
       pages={349\ndash 381},
}

\bib{artin2011algebra}{book}{
      author={Artin, M.},
       title={Algebra},
   publisher={Pearson Prentice Hall},
        date={2011},
        ISBN={9780132413770},
         url={https://books.google.com.hk/books?id=QsOfPwAACAAJ},
}

\bib{butet2015optical}{article}{
      author={Butet, J{\'e}r{\'e}my},
      author={Brevet, Pierre-Fran{\c{c}}ois},
      author={Martin, Olivier~JF},
       title={Optical second harmonic generation in plasmonic nanostructures: from fundamental principles to advanced applications},
        date={2015},
     journal={ACS nano},
      volume={9},
      number={11},
       pages={10545\ndash 10562},
}

\bib{baruch2007high}{article}{
      author={Baruch, G},
      author={Fibich, G},
      author={Tsynkov, S},
       title={High-order numerical solution of the nonlinear helmholtz equation with axial symmetry},
        date={2007},
     journal={Journal of computational and applied mathematics},
      volume={204},
      number={2},
       pages={477\ndash 492},
}

\bib{boyd2008nonlinear}{incollection}{
      author={Boyd, Robert~W},
      author={Gaeta, Alexander~L},
      author={Giese, Enno},
       title={Nonlinear optics},
        date={2008},
   booktitle={Springer handbook of atomic, molecular, and optical physics},
   publisher={Springer},
       pages={1097\ndash 1110},
}

\bib{baranov2017all}{article}{
      author={Baranov, D.~G.},
      author={Zuev, D.~A.},
      author={Lepeshov, S.~I.},
      author={Kotov, O.~V.},
      author={Krasnok, A.~E.},
      author={Evlyukhin, A.~B.},
      author={Chichkov, B.~N.},
       title={All-dielectric nanophotonics: the quest for better materials and fabrication techniques},
        date={2017},
     journal={Optica},
      volume={4},
      number={7},
       pages={814\ndash 825},
}

\bib{chen2021complex}{article}{
      author={Chen, Huyuan},
      author={Ev{\'e}quoz, Gilles},
      author={Weth, Tobias},
       title={Complex solutions and stationary scattering for the nonlinear helmholtz equation},
        date={2021},
     journal={SIAM Journal on Mathematical Analysis},
      volume={53},
      number={2},
       pages={2349\ndash 2372},
}

\bib{corless1996lambert}{article}{
      author={Corless, Robert~M},
      author={Gonnet, Gaston~H},
      author={Hare, David~EG},
      author={Jeffrey, David~J},
      author={Knuth, Donald~E},
       title={On the lambert w function},
        date={1996},
     journal={Advances in Computational mathematics},
      volume={5},
       pages={329\ndash 359},
}

\bib{cao2022electromagnetic}{article}{
      author={Cao, Xinlin},
      author={Ghandriche, Ahcene},
      author={Sini, Mourad},
       title={The electromagnetic waves generated by dielectric nanoparticles},
        date={2022},
     journal={arXiv preprint arXiv:2209.02413},
}

\bib{ciarlet2013linear}{book}{
      author={Ciarlet, Philippe~G},
       title={Linear and nonlinear functional analysis with applications},
   publisher={SIAM},
        date={2013},
}

\bib{colton1998inverse}{book}{
      author={Colton, David~L},
      author={Kress, Rainer},
      author={Kress, Rainer},
       title={Inverse acoustic and electromagnetic scattering theory},
   publisher={Springer},
        date={1998},
      volume={93},
}

\bib{crandall1971bifurcation}{article}{
      author={Crandall, Michael~G},
      author={Rabinowitz, Paul~H},
       title={Bifurcation from simple eigenvalues},
        date={1971},
     journal={Journal of Functional Analysis},
      volume={8},
      number={2},
       pages={321\ndash 340},
}

\bib{deimling2013nonlinear}{book}{
      author={Deimling, Klaus},
       title={Nonlinear functional analysis},
   publisher={Springer Science \& Business Media},
        date={2013},
}

\bib{dyatlov2019mathematical}{book}{
      author={Dyatlov, Semyon},
      author={Zworski, Maciej},
       title={Mathematical theory of scattering resonances},
   publisher={American Mathematical Soc.},
        date={2019},
      volume={200},
}

\bib{evlyukhin2012demonstration}{article}{
      author={Evlyukhin, A.~B.},
      author={Novikov, S.~M.},
      author={Zywietz, U.},
      author={Eriksen, R.~L.},
      author={Reinhardt, C.},
      author={Bozhevolnyi, S.~I.},
      author={Chichkov, B.~N.},
       title={Demonstration of magnetic dipole resonances of dielectric nanospheres in the visible region},
        date={2012},
     journal={Nano letters},
      volume={12},
      number={7},
       pages={3749\ndash 3755},
}

\bib{evequoz2013real}{article}{
      author={Ev{\'e}quoz, Gilles},
      author={Weth, Tobias},
       title={Real solutions to the nonlinear helmholtz equation with local nonlinearity},
        date={2013},
     journal={arXiv preprint arXiv:1302.0530},
}

\bib{evequoz2015dual}{article}{
      author={Evequoz, Gilles},
      author={Weth, Tobias},
       title={Dual variational methods and nonvanishing for the nonlinear helmholtz equation},
        date={2015},
     journal={Advances in Mathematics},
      volume={280},
       pages={690\ndash 728},
}

\bib{evequoz2020dual}{article}{
      author={Ev{\'e}quoz, Gilles},
      author={Ye{\c{s}}il, Tolga},
       title={Dual ground state solutions for the critical nonlinear helmholtz equation},
        date={2020},
     journal={Proceedings of the Royal Society of Edinburgh Section A: Mathematics},
      volume={150},
      number={3},
       pages={1155\ndash 1186},
}

\bib{feppon2022modal}{article}{
      author={Feppon, Florian},
      author={Ammari, Habib},
       title={Modal decompositions and point scatterer approximations near the minnaert resonance frequencies},
        date={2022},
     journal={Studies in Applied Mathematics},
      volume={149},
      number={1},
       pages={164\ndash 229},
}

\bib{fibich2001high}{article}{
      author={Fibich, Gadi},
      author={Tsynkov, Semyon},
       title={High-order two-way artificial boundary conditions for nonlinear wave propagation with backscattering},
        date={2001},
     journal={Journal of computational physics},
      volume={171},
      number={2},
       pages={632\ndash 677},
}

\bib{fibich2005numerical}{article}{
      author={Fibich, G},
      author={Tsynkov, S},
       title={Numerical solution of the nonlinear helmholtz equation using nonorthogonal expansions},
        date={2005},
     journal={Journal of Computational Physics},
      volume={210},
      number={1},
       pages={183\ndash 224},
}

\bib{gohberg1990classes}{book}{
      author={Gohberg, I.},
      author={Goldberg, S.},
      author={Kaashoek, M.~A.},
       title={Classes of linear operators},
   publisher={Birkh{\"a}user},
        date={1990},
      volume={49},
}

\bib{griesmaier2022inverse}{article}{
      author={Griesmaier, Roland},
      author={Kn{\"o}ller, Marvin},
      author={Mandel, Rainer},
       title={Inverse medium scattering for a nonlinear helmholtz equation},
        date={2022},
     journal={Journal of Mathematical Analysis and Applications},
      volume={515},
      number={1},
       pages={126356},
}

\bib{gell2020existence}{article}{
      author={Gell-Redman, Jesse},
      author={Hassell, Andrew},
      author={Shapiro, Jacob},
      author={Zhang, Junyong},
       title={Existence and asymptotics of nonlinear helmholtz eigenfunctions},
        date={2020},
     journal={SIAM Journal on Mathematical Analysis},
      volume={52},
      number={6},
       pages={6180\ndash 6221},
}

\bib{gohberg1971operator}{article}{
      author={Gohberg, Israel~C},
      author={Sigal, Efim~I},
       title={An operator generalization of the logarithmic residue theorem and the theorem of rouch{\'e}},
        date={1971},
     journal={Mathematics of the USSR-Sbornik},
      volume={13},
      number={4},
       pages={603},
}

\bib{gilbarg1977elliptic}{book}{
      author={Gilbarg, David},
      author={Trudinger, Neil~S},
      author={Gilbarg, David},
      author={Trudinger, NS},
       title={Elliptic partial differential equations of second order},
   publisher={Springer},
        date={1977},
      volume={224},
      number={2},
}

\bib{gutierrez2004non}{article}{
      author={Guti{\'e}rrez, Susana},
       title={Non trivial l q solutions to the ginzburg-landau equation},
        date={2004},
     journal={Mathematische Annalen},
      volume={328},
      number={1},
       pages={1\ndash 25},
}

\bib{harrell1980double}{article}{
      author={Harrell, Evans~M},
       title={Double wells},
        date={1980},
     journal={Communications in Mathematical Physics},
      volume={75},
      number={3},
       pages={239\ndash 261},
}

\bib{ilan2010band}{article}{
      author={Ilan, Boaz},
      author={Weinstein, Michael~I},
       title={Band-edge solitons, nonlinear schr{\"o}dinger/gross--pitaevskii equations, and effective media},
        date={2010},
     journal={Multiscale Modeling \& Simulation},
      volume={8},
      number={4},
       pages={1055\ndash 1101},
}

\bib{jalade2004inverse}{inproceedings}{
      author={Jalade, Emmanuel},
       title={Inverse problem for a nonlinear helmholtz equation},
        date={2004},
   booktitle={Annales de l'ihp analyse non lin{\'e}aire},
      volume={21},
       pages={517\ndash 531},
}

\bib{jerison1985unique}{article}{
      author={Jerison, David},
      author={Kenig, Carlos~E},
       title={Unique continuation and absence of positive eigenvalues for schr{\"o}dinger operators},
        date={1985},
     journal={Annals of Mathematics},
      volume={121},
      number={3},
       pages={463\ndash 488},
}

\bib{jiang2022finite}{article}{
      author={Jiang, Run},
      author={Li, Yonglin},
      author={Wu, Haijun},
      author={Zou, Jun},
       title={Finite element method for a nonlinear perfectly matched layer helmholtz equation with high wave number},
        date={2022},
     journal={SIAM Journal on Numerical Analysis},
      volume={60},
      number={5},
       pages={2866\ndash 2896},
}

\bib{jackson2004geometric}{article}{
      author={Jackson, Russell~K},
      author={Weinstein, Michael~I},
       title={Geometric analysis of bifurcation and symmetry breaking in a gross—pitaevskii equation},
        date={2004},
     journal={Journal of statistical physics},
      volume={116},
      number={1},
       pages={881\ndash 905},
}

\bib{kato2013perturbation}{book}{
      author={Kato, Tosio},
       title={Perturbation theory for linear operators},
   publisher={Springer Science \& Business Media},
        date={2013},
      volume={132},
}

\bib{kivshar2018all}{article}{
      author={Kivshar, Yuri},
       title={All-dielectric meta-optics and non-linear nanophotonics},
        date={2018},
     journal={National Science Review},
      volume={5},
      number={2},
       pages={144\ndash 158},
}

\bib{koshelev2020subwavelength}{article}{
      author={Koshelev, Kirill},
      author={Kruk, Sergey},
      author={Melik-Gaykazyan, Elizaveta},
      author={Choi, Jae-Hyuck},
      author={Bogdanov, Andrey},
      author={Park, Hong-Gyu},
      author={Kivshar, Yuri},
       title={Subwavelength dielectric resonators for nonlinear nanophotonics},
        date={2020},
     journal={Science},
      volume={367},
      number={6475},
       pages={288\ndash 292},
}

\bib{kirr2011symmetry}{article}{
      author={Kirr, Eduard},
      author={Kevrekidis, Panayotis~G},
      author={Pelinovsky, Dmitry~E},
       title={Symmetry-breaking bifurcation in the nonlinear schr{\"o}dinger equation with symmetric potentials},
        date={2011},
     journal={Communications in mathematical physics},
      volume={308},
      number={3},
       pages={795\ndash 844},
}

\bib{kirr2008symmetry}{article}{
      author={Kirr, EW},
      author={Kevrekidis, PG},
      author={Shlizerman, Eli},
      author={Weinstein, Michael~I},
       title={Symmetry-breaking bifurcation in nonlinear schr{\"o}dinger/gross--pitaevskii equations},
        date={2008},
     journal={SIAM journal on mathematical analysis},
      volume={40},
      number={2},
       pages={566\ndash 604},
}

\bib{kuznetsov2016optically}{article}{
      author={Kuznetsov, A.~I.},
      author={Miroshnichenko, A.~E.},
      author={Brongersma, M.~L.},
      author={Kivshar, Y.~S.},
      author={Luk’yanchuk, B.},
       title={Optically resonant dielectric nanostructures},
        date={2016},
     journal={Science},
      volume={354},
      number={6314},
       pages={aag2472},
}

\bib{li2024large}{article}{
      author={Li, Long},
      author={Sini, Mourad},
       title={Large time behavior for acoustic resonators},
        date={2024},
     journal={arXiv preprint arXiv:2410.09630},
}

\bib{mandel2025dual}{article}{
      author={Mandel, Rainer},
       title={Dual variational methods for time-harmonic nonlinear maxwell's equations},
        date={2025},
     journal={arXiv preprint arXiv:2505.00992},
}

\bib{mandel2017oscillating}{article}{
      author={Mandel, Rainer},
      author={Montefusco, Eugenio},
      author={Pellacci, Benedetta},
       title={Oscillating solutions for nonlinear helmholtz equations},
        date={2017},
     journal={Zeitschrift f{\"u}r angewandte Mathematik und Physik},
      volume={68},
      number={6},
       pages={121},
}

\bib{meklachi2018asymptotic}{article}{
      author={Meklachi, Taoufik},
      author={Moskow, Shari},
      author={Schotland, John~C},
       title={Asymptotic analysis of resonances of small volume high contrast linear and nonlinear scatterers},
        date={2018},
     journal={Journal of Mathematical Physics},
      volume={59},
      number={8},
}

\bib{nirenberg1974topics}{book}{
      author={Nirenberg, Louis},
       title={Topics in nonlinear functional analysis},
   publisher={American Mathematical Soc.},
        date={1974},
      volume={6},
}

\bib{smirnova2020nonlinear}{article}{
      author={Smirnova, Daria},
      author={Leykam, Daniel},
      author={Chong, Yidong},
      author={Kivshar, Yuri},
       title={Nonlinear topological photonics},
        date={2020},
     journal={Applied Physics Reviews},
      volume={7},
      number={2},
}

\bib{teytel1999rare}{article}{
      author={Teytel, Mikhail},
       title={How rare are multiple eigenvalues?},
        date={1999},
     journal={Communications on Pure and Applied Mathematics: A Journal Issued by the Courant Institute of Mathematical Sciences},
      volume={52},
      number={8},
       pages={917\ndash 934},
}

\bib{tzarouchis2018light}{article}{
      author={Tzarouchis, D.},
      author={Sihvola, A.},
       title={Light scattering by a dielectric sphere: perspectives on the {M}ie resonances},
        date={2018},
     journal={Applied Sciences},
      volume={8},
      number={2},
       pages={184},
}

\bib{wu2018finite}{article}{
      author={Wu, Haijun},
      author={Zou, Jun},
       title={Finite element method and its analysis for a nonlinear helmholtz equation with high wave numbers},
        date={2018},
     journal={SIAM Journal on Numerical Analysis},
      volume={56},
      number={3},
       pages={1338\ndash 1359},
}

\bib{yang2020newtonian}{article}{
      author={Yang, Guangchong},
      author={Lan, Kunquan},
       title={Newtonian potential and positive solutions of poisson equations},
        date={2020},
     journal={Nonlinear Analysis},
      volume={196},
       pages={111811},
}

\end{biblist}
\end{bibdiv}
\end{document}